\documentclass[11pt]{article}

\usepackage[utf8]{inputenc} 
\usepackage[T1]{fontenc}    
\usepackage{url}            
\usepackage{booktabs}       
\usepackage{nicefrac}       
\usepackage{microtype}      
\usepackage{multirow}
\usepackage{array}
\usepackage[noblocks]{authblk}
\usepackage{amssymb,arydshln}
\usepackage[nodayofweek]{datetime}
\usepackage{float}
\usepackage{bm}

\usepackage{tocloft}  
\usepackage[shortlabels]{enumitem}
\setlist[itemize]{itemindent=0ex,parsep=3pt,itemsep=0pt,leftmargin=\parindent,topsep=5pt,labelwidth=0.8em,labelsep=0.7em}
\setlist[enumerate]{label={\arabic*)},itemindent=0ex,parsep=3pt,itemsep=0pt,leftmargin=\parindent,topsep=5pt,labelwidth=0.9em,labelsep=0.6em}

\usepackage[font=small,labelfont=bf]{caption}
\usepackage{amsmath,amsfonts,nccmath,mathtools,mathrsfs}
\usepackage{tcolorbox}
\usepackage{xcolor}

\usepackage[margin=1in]{geometry}

\usepackage{pgfplots}
\usepgfplotslibrary{fillbetween}
\usetikzlibrary{arrows.meta}
\tikzset{>=latex}

\DeclarePairedDelimiterX\setv[2]{\{}{\}}{#1 \;\delimsize\vert\; #2}

\pgfmathdeclarefunction{cubic}{1}{%
  \pgfmathparse{-2*(#1+2)*(#1+2)*(#1-2)+40}%
}

\pgfmathdeclarefunction{bicuadratic}{1}{%
  \pgfmathparse{(#1-1)*(#1-1)*(#1-1)*(#1-1)+10}%
}

\pgfmathdeclarefunction{cuadratic}{1}{%
    \pgfmathparse{(-(2*#1)^2)+70}%
}

\usepackage{pifont}

\usepackage[numbers,merge,sort&compress]{natbib}
\usepackage{amsthm}
\usepackage{graphicx,color}
\usepackage{subcaption,setspace}

\usepackage[colorlinks = true,
            linkcolor = blue,
            urlcolor  = blue,
            citecolor = blue,
            anchorcolor = blue]{hyperref}

\newtheorem{theorem}{Theorem}
\newtheorem{proposition}{Proposition}
\newtheorem{lemma}{Lemma}
\newtheorem{corollary}{Corollary}

\newtheorem{assumption}{Assumption}

\theoremstyle{definition}
\newtheorem{definition}{Definition}
\newtheorem{example}{Example}

\newtheorem{remark}{Remark}

\usepackage[nameinlink]{cleveref}
\crefname{equation}{}{}
\crefname{theorem}{Theorem}{Theorems}
\crefname{corollary}{Corollary}{Corollaries}
\crefname{example}{Example}{Examples}
\crefname{assumption}{Assumption}{Assumptions}
\crefname{lemma}{Lemma}{Lemmas}
\crefname{proposition}{Proposition}{Propositions}
\crefname{figure}{Figure}{Figures}
\crefname{table}{Table}{Tables}
\crefname{fact}{Fact}{Facts}
\crefname{conjecture}{Conjecture}{Conjectures}
\crefname{section}{Section}{Sections}
\crefname{appendix}{Appendix}{Appendices}
\Crefname{equation}{}{}
\Crefname{theorem}{Theorem}{Theorems}
\Crefname{corollary}{Corollary}{Corollaries}
\Crefname{example}{Example}{Examples}
\Crefname{lemma}{Lemma}{Lemma}
\Crefname{proposition}{Proposition}{Proposition}
\Crefname{figure}{Figure}{Figures}
\Crefname{table}{Table}{Tables}
\Crefname{section}{Section}{Sections}
\Crefname{appendix}{Appendix}{Appendices}

\newcommand{\tr}{{{\mathsf T}}}
\newcommand{\her}{{{\mathsf H}}}

\newcommand{\ECL}{\mathtt{ECL}}
\newcommand{\argmin}{\mathop{\mathrm{arg\,min}}\limits}

\newcommand{\LQR}{\mathtt{LQR}}

\newcommand{\lft}{\mathrm{lft}}
\newcommand{\cvx}{\mathrm{cvx}}

\graphicspath{{figure/}}

\usepackage{tikz}
\usepackage{pgfplots}
\pgfplotsset{compat=newest}

\makeatletter
\newcommand{\removelatexerror}{\let\@latex@error\@gobble}
\makeatother

\title{\bf \LARGE Revisiting Strong Duality, Hidden Convexity, 
and Gradient Dominance in the Linear Quadratic Regulator 
\thanks{This work is supported by NSF ECCS 2154650, NSF CMMI 2320697, and NSF CAREER 2340713. Emails: \texttt{y1watanabe@ucsd.edu; zhengy@ucsd.edu}.}
}

\author[1]{Yuto Watanabe} 
\author[1]{Yang Zheng}
\affil[1]{\small Department of Electrical and Computer Engineering, University of California San Diego \vspace{-2mm}}

\date{ \small \today \vspace{-3mm}} 

\begin{document}

\maketitle
\vspace{-5mm}

\begin{abstract}

\noindent The Linear Quadratic Regulator (LQR) is a cornerstone of optimal control theory, widely studied in both model-based and model-free approaches. Despite its well-established nature, certain foundational aspects remain subtle. In this paper, we revisit three key properties of policy optimization in LQR: (i) strong duality in the nonconvex policy optimization formulation, (ii) the gradient dominance property, examining when it holds and when it fails, and (iii) the global optimality of linear static policies. Using primal-dual analysis and convex reformulation, we refine and clarify existing results by leveraging Riccati equations/inequalities, semidefinite programming (SDP) duality, and a recent framework of Extended Convex Lifting (\texttt{ECL}). Our analysis confirms that LQR 1) behaves almost like a convex problem (e.g., strong duality) under the standard assumptions of stabilizability and detectability and 2) exhibits strong convexity-like properties (e.g., gradient dominance) under slightly stronger conditions. In particular, we establish a broader characterization under which gradient dominance holds using \texttt{ECL} and the notion of Cauchy directions. 
By clarifying and refining these theoretical insights, we hope this work contributes to a deeper understanding of LQR and may inspire further developments beyond LQR.

\end{abstract}

\section{Introduction}

    The Linear Quadratic Regulator (LQR) is one of the most fundamental problems in optimal control, with both theoretical significance and widespread practical applications. This classical problem seeks to design a linear control system by minimizing the integral of a quadratic cost function along the system trajectories. It is well established that, for linear time-invariant (LTI) systems in the infinite-horizon case, the optimal solution is a linear static constant state feedback controller \cite{zhou1996robust}. This provides a systematic approach for computing constant feedback gains based on a performance measure that offers a clear interpretation in terms of control effort and tracking error. 
    
    Since Kalman's seminal work \cite{kalman1960contributions}, LQR has been extensively studied and has served as a foundation for numerous advanced control methods, such as Linear Quadratic Gaussian (LQG) control \cite{doyle1988state}, $\mathcal{H}_\infty$ control \cite{green2012linear}, 
    mixed $\mathcal{H}_2/\mathcal{H}_\infty$ control \cite{khargonekar1991mixed},
    model predictive control (MPC) \cite{bemporad2002explicit}, and reinforcement learning (RL) \cite{recht2019tour}. 
    The standard and classical tools
    for analyzing LQR include  variational analysis \cite{liberzon2011calculus} and dynamic programming \cite{bellman1966dynamic}, along with the algebraic Riccati equation (ARE) \cite{willems1971least}.
   In addition, 
the \textit{completion of squares} technique \cite{zhou1996robust} can certify the globally optimal solution to the LQR problem.  
It is now~well understood that 
    linear matrix inequalities (LMIs) and semidefinite programs (SDPs) also play a central role in both the computation and characterizations of many control problems \cite{boyd1994linear}, including the classical LQR. 
    For discrete-time systems, the finite-horizon LQR can be conveniently expressed as a convex quadratic program, especially in MPC for online control \cite{rawlings2017model}.
    Given these theoretical and computational developments, LQR is widely considered a well-understood and nearly complete~problem when \textit{the system model is known}. 

    Over the past decade, there has been remarkable progress in \textit{model-free} RL, which seeks to control \textit{unknown} dynamical systems using only observed system trajectories \cite{recht2019tour}. Motivated by this, the LQR problem with unknown dynamics has re-attracted increasing attention  \cite{dean2020sample,zheng2021sample,fazel2018global,hu2023toward}. 
A key algorithmic framework driving the success of RL is policy optimization, which maps observations to control actions through appropriate parametrization. Policy optimization for LQR (as well as many other control problems) is known to be nonconvex (see \Cref{fig:lqr-landscape} for typical LQR landscapes). It is shown in   
   \begin{figure}
    \centering
    \setlength{\abovecaptionskip}{0pt}
  \begin{subfigure}{0.33\linewidth}
           \centering
\includegraphics[width=0.7\linewidth]{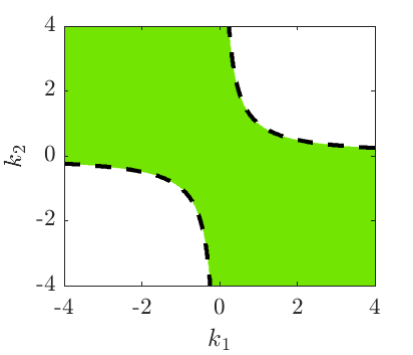}
            \caption{
            }
            \label{fig:non-convex-K}
  \end{subfigure}%
  \begin{subfigure}{0.33\linewidth}
          \centering
    \includegraphics[width=0.8\linewidth]{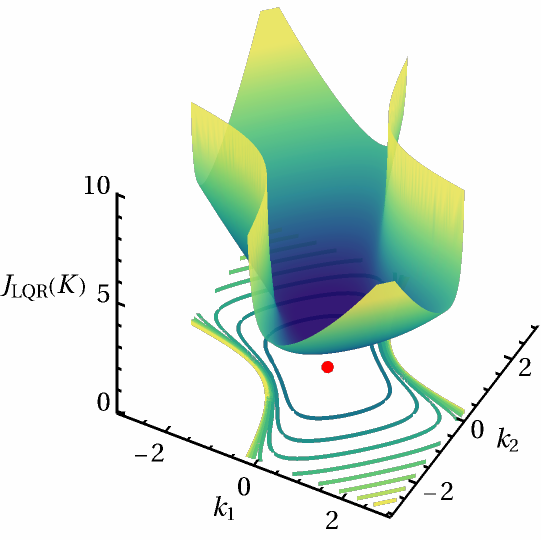}
       \caption{}
       \label{fig:non-convex-JLQR}
  \end{subfigure}
  \begin{subfigure}{0.33\linewidth}
  \centering
      \includegraphics[width=0.79\linewidth]{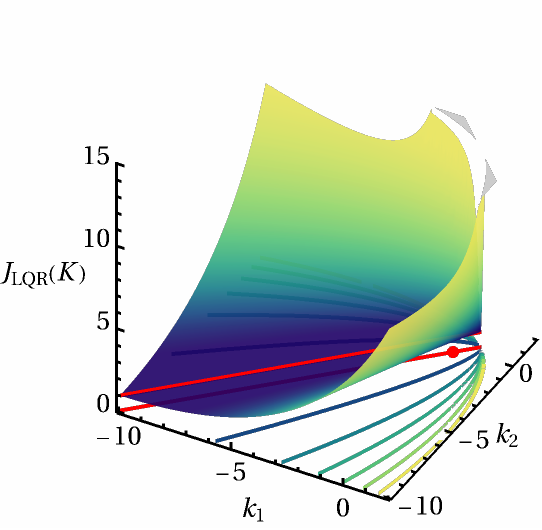}
      \caption{}
  \end{subfigure}
  \caption{Nonconvex landscapes in LQR:
    (a) a slice of the set of static feedback gains $K$
        such that 
            $A+BK$ is stable, where $A=0$ and $B=I_2$;
    (b) a ``well-behaved'' nonconvex and smooth landscape of the LQR cost from \cite[Appendix A.1]{zheng2023benign}; (c) Nonunique globally optimal LQR feedback gains in  \cref{example:W-singular}; sublevel sets of this instance are unbounded. 
    The red point $K^\star$ denotes the optimal feedback gain from solving the ARE.
    }
    \vspace{-2mm}
    \label{fig:lqr-landscape}
\end{figure}
    \cite{fazel2018global} that the LQR problem enjoys a favorable property of \textit{gradient dominance},  also known as
    the \textit{Polyak-Łojasiewicz (PL)} inequality \cite{polyak1963gradient}, which generalizes strong convexity to certain nonconvex problems. This property ensures that simple policy gradient algorithms achieve linear convergence to the global minimizer of the LQR problem \cite{fazel2018global}.
    Since this discovery, a growing body of research has focused on designing efficient \textit{model-free} algorithms \cite{mohammadi2019global,tu2019gap,malik2020derivative} and analyzing rich landscape properties of fundamental control problems, including LQR \cite{hu2023toward,talebi2024policy,tang2023analysis}. 

    As demonstrated above, 
LQR has been extensively studied from both classical model-based design and recent model-free approaches, making it a fundamental benchmark at the intersection of control and learning. Despite its well-established nature, certain fundamental aspects of LQR remain subtle and are often overlooked. These include the global optimality of linear policies, the uniqueness of the optimal feedback gain, and the nuanced, case-by-case analysis of gradient dominance. 

\begin{table}[t]
\centering
\setlength{\abovecaptionskip}{4pt}
\renewcommand{\arraystretch}{1.1}
 \begin{tabular}{crccccc}
\toprule
 & Result& $(A,B)$ & $(Q^{1/2},A)$ & $Q$ & $W$  \\
\midrule
{Strong duality} & \cref{theorem:strong-duality-1} & stabilizable & detectable & $Q \succeq 0$ & $W \succeq 0$ \\
\hdashline 
\multirow{3}{*}{Gradient dominance}& \cite{mohammadi2019global,fatkhullin2021optimizing} & stabilizable & observable & $Q \succ 0$ & $W \succ 0$ \\
& \cite{bu2020policy} & stabilizable & detectable & $Q \succeq 0$ & $W \succ 0$ \\%
& \cref{theorem:J_LQR-gradient_dominance-general}$^\dagger$   & stabilizable & detectable & $Q \succeq 0$ & $W \succeq 0$\\
\hdashline 
{Global optimality} & \cref{theorem:linear-optimality}& stabilizable & detectable & $Q \succeq 0$ & {$x_0x_0^\tr$}\\
\bottomrule
\end{tabular}
\caption{Conditions for strong duality, gradient dominance, and global optimality in LQR. The strong duality~and global optimality of linear static feedback policies only require the standard stabilizability and detectability assumption (see \Cref{assumption:controllability}). $^\dagger$Our \cref{theorem:J_LQR-gradient_dominance-general} requires an additional \cref{assumption:compactness} that could include certain LQR instances with $Q \succeq 0$  and $W \succeq 0$ (see \cref{proposition:sufficient-conditions} and \cref{example:PL-W>=0,example:PL-W-Q>=0}).}\label{tab:conditions}
\end{table}

\subsection{Our results}
    In this paper, we revisit three fundamental properties of policy optimization in LQR: 
    1) strong duality in the nonconvex policy optimization, 2) the intriguing gradient dominance property (examining both \textit{when it holds and when it fails}), and 3) the global optimality of linear static policies.

Many results have been established in the literature through various techniques and settings. Our approach is rooted in \textit{primal-dual analysis} and \textit{convex reformulation}. 
    We
    revisit the strong duality and global optimality of linear static policies from a primal and dual lens
    and leverage the convex reformulation of LQR to establish the gradient dominance property under broader assumptions than the existing studies \cite{mohammadi2019global,bu2020policy,fatkhullin2021optimizing}. 
    Our technical results are summarized below (see \cref{tab:conditions} for an overview). 
\begin{enumerate}
    \item \textbf{Strong duality in LQR.} While the policy optimization for LQR is known to be nonconvex, 
    we demonstrate that strong duality holds under stabilizability and detectability. 
    This result is not surprising, but
     a self-contained and explicit proof is not easily accessible in the literature, as existing discussions often rely on different formulations or advanced mathematical tools. 
   Early discussions in \cite{balakrishnan2003semidefinite} are informal and technical details are missing. The recent result in \cite{bamieh2024linear} only focuses on the finite-horizon case and lacks the characterization on the uniqueness of the solution.
   Our approach follows the standard Lagrange duality theory, leveraging well-established Riccati equations/inequalities and careful KKT analysis. While these techniques are standard, we carefully fill~in~the missing~details. One subtle point is that the optimal LQR feedback gain may not be unique when the covariance matrix $W$
   is only positive semidefinite ($W \succeq 0$); see \Cref{fig:lqr-landscape}(c) and \cref{example:W-singular}. In this case, the standard solution derived from the Riccati equation is just one optimal solution. If $W$~is strictly positive definite ($W \succ 0$), the optimal LQR feedback gain is unique. In this setting, we~further reveal a favorable \textit{strict complementarity} property~\cite{alizadeh1997complementarity,ding2023strict,liao2024inexact} in the associated SDPs for LQR.
    
    \item \textbf{Gradient dominance via Extended Convex Lifting.} We revisit the elegant property of \textit{gradient dominance} in LQR using a recent framework of Extended Convex Lifting (\texttt{ECL}) \cite{zheng2023benign,zheng2024benign}.~Our analysis reveals two key insights: 1) gradient dominance requires stronger conditions than those for strong duality; with the usual stabilizability and detectability, gradient dominance may fail (see \Cref{example:unbounded-sublevel-set}); 2) gradient dominance does not hold globally for continuous-time LQR even with strong assumptions of controllability and observability (see \Cref{example:PL-globally-no}). Accordingly, we establish the most general known conditions for ensuring gradient dominance in LQR: (i) compact sublevel sets and (ii) a positive definite solution to a Lyapunov equation. This setting extends beyond the existing results in  \cite{fatkhullin2021optimizing,bu2019lqr,bu2020policy}. Our proof leverages the \texttt{ECL} framework while refining its analysis using the notion of Cauchy directions  \cite{lemarechal2012cauchy}. A key insight in our proof is that the existence of gradient dominance is closely tied to a convex reformulation of LQR
    that satisfies \textit{quadratic growth} \cite{karimi2016linear,liao2024error}, a similar but weaker condition than strong convexity.
    
    \item \textbf{Global optimality of static state feedback policies.} We also revisit the global optimality of static state feedback policies among all (potentially nonlinear) stabilizing policies.~Similar to~\cite{balakrishnan2003semidefinite,bamieh2024linear}, we adopt the primal and dual analysis for LQR. 
    A key ingredient in our analysis is the~Gramian (or covariance) representation \cite{hotz1987covariance}, which captures all relevant quadratic information of system trajectories. Leveraging this representation, we reformulate LQR over all stabilizing policies, including nonlinear ones, as a finite-dimensional optimization problem. We~then construct \textit{inner} and \textit{outer} approximations: the inner approximation restricts the policies to be static~state feedback, while the outer approximation relaxes the feasible domain to an SDP. A~careful~KKT analysis reveals that these approximations preserve the same optimal value~as~the original LQR, which confirms that an optimal LQR policy is linear static state~feedback.    
\end{enumerate}

\subsection{Related works}

\noindent \textbf{Primal and dual analysis in control.} 
Optimization has long been an indispensable tool in control, and the use of (convex) duality in control theory has a rich history. 
Early primal-dual~results can be traced back to the 1970s, when Rockafellar \cite{rockafellar1970conjugate} conducted a primal-dual analysis for optimal control with convex costs and constraints.
For $\ell_1$ optimal control, Dahleh and Pearson exploited a dual problem to compute an optimal closed-loop transfer matrix \cite{dahleh19871}. Later, Rantzer introduced a new duality-based proof for the Kalman-Yakubovich-Popov (KYP) lemma in \cite{rantzer1996kalman}. 
Since the 2000s, driven by advancements in LMIs and SDPs, duality has been increasingly utilized in key control problems. The seminal work \cite{balakrishnan2003semidefinite} provides an elegant interpretation of several fundamental control problems, including the KYP lemma and LQR, through duality analysis.
Additionally, Scherer \cite{scherer2006lmi} discovered that 
the relaxation process of robust LMIs by S-lemma was exactly the Lagrangian relaxation procedure, and the connection was further examined in \cite{ebihara2009robust}. Recent studies \cite{gattami2015simple,you2014h,you2015primal} have introduced alternative methods for computing $\mathcal{H}_\infty$ norm using the SDP duality and the \textit{Gramian/covariance representation}.
Very recently, Bamieh \cite{bamieh2024linear} revisited the \textit{finite-horizon} linear quadratic control and established the global optimality of linear static policies through KKT analysis and a \textit{differential Riccati equation} arising from a differential LMI in the dual.

It is important to distinguish the \textit{Lagrange duality} in optimization from the duality between \textit{control} and \textit{estimation}. 
Still, a fundamental similarity exists: \textit{adjoint operators} play a crucial role in both contexts.  A key connection between control and estimation lies in the $2$-norm equivalence between linear operators and their adjoints \cite[Chapter 3.3.2]{green2012linear}, where adjoint operators of dynamical systems have an interpretation of observers.   
Similarly, adjoint operators emerge naturally in Lagrange dual problems in optimization \cite{rockafellar1970convex}. 

\vspace{3mm}

\noindent 
\textbf{Benign nonconvex landscape of policy optimization in control.} 
Motivated by the remarkable success of RL,  
many recent studies have revisited control costs as functions of policy parameters, examining their analytical and geometrical properties. 
A key breakthrough is the discovery of gradient dominance for the discrete-time LQR in \cite{fazel2018global}. Since then, gradient dominance has been further explored in discrete-time variants   
\cite{bu2019lqr,hambly2021policy} 
and
their continuous-time counterparts \cite{mohammadi2019global,bu2020policy,fatkhullin2021optimizing,moghaddam2024sample,zhang2025convergence}. These studies often differ in their settings and involve nuanced, often case-by-case analyses of gradient dominance.   
Beyond the classical LQR, the landscape of other fundamental control problems has also been investigated, including 
filtering \cite{umenberger2022globally}, LQG \cite{tang2023analysis,zheng2023benign}, LQ differential game \cite{zhang2019policy}, $\mathcal{H}_\infty$ control \cite{guo2022global,tang2023global}, and distributed control \cite{furieri2020learning}; see two recent surveys  \cite{hu2023toward,talebi2024policy} for a broader overview.
Very recently, 
several works \cite{sun2021learning,umenberger2022globally,zheng2024benign,tang2023global,zheng2023benign} 
have sought to identify the underlying sources that enable benign nonconvex landscapes. A common idea is to connect the original nonconvex problems to their well-established convex reformulations. 
Among these, the Extended Convex Lifting ($\ECL$) \cite{zheng2024benign,zheng2023benign} is the most~generic framework, which naturally accommodates dynamic output feedback and nonsmooth cost functions.

In this paper, we utilize the \texttt{ECL} framework and refine its analysis using the notion of Cauchy directions~\cite{lemarechal2012cauchy} for the LQR problem. This provides an alternative streamlined approach to establish gradient dominance for LQR under the most general known setting.   

\vspace{3mm}

\noindent 
\textbf{Gramian/covariance representation in control.} 
The covariance representation 
is a powerful technique that reformulates standard control problems for linear systems into SDPs by replacing 
the state, input, and disturbances with their covariance matrices.
Its deterministic counterpart is known as Gramian representation.
The use of covariance matrices in control was pioneered~by Hotz and Skelton \cite{hotz1987covariance} and later extensively discussed in the book \cite{skelton1997unified}. 
 This approach has been useful for handling power constraints in controller synthesis in a convex manner. 
In \cite{rantzer2006linear}, Rantzer applied this idea to a distributed control problem, and Gattami proposed an SDP-based controller synthesis method for linear quadratic control with power constraints \cite{gattami2009generalized}. 
Gattami and Bamieh proposed an alternative computation method of $\mathcal{H}_\infty$ norm via the covariance representation and duality  \cite{gattami2015simple}.
The covariance technique was also applied to an online learning scheme for LQR with unknown performance weights in \cite{cohen2018online}. 
Very recently, Bamieh utilized this covariance representation technique to derive the optimal controller for finite-horizon linear quadratic control \cite{bamieh2024linear}. 
Our recent work \cite{watanabe2025semidefinite} also dealt with LQ differential games and $\mathcal{H}_\infty$ control in the same spirit.

In this paper, we leverage the Gramian representation to revisit the global optimality of static state feedback policies among all (potentially nonlinear) stabilizing policies for the infinite-horizon LQR problem.

\subsection{Paper outline}
The rest of this paper is organized as follows. \Cref{section:LQR-formulation} presents a classical formulation of the linear quadratic regulator (LQR). Strong duality and hidden convexity of LQR are discussed in \Cref{section:strong-duality}. In \Cref{section:GD}, we revisit an elegant property of gradient dominance in LQR. We then investigate the global optimality of static policies in LQR from a primal and dual optimization perspective in \cref{section:global-optimality_static-linear-policies}. We finally conclude the paper in \Cref{section:conclusion}. Some auxiliary results, additional discussions,
and technical proofs are provided in the appendix.

\vspace{6pt}
\noindent \textbf{Notations.}  
We use $\mathbb{S}^n$ to denote the set of
$n\times n$ symmetric matrices. In the space $\mathbb{S}^n$, the sets of positive (semi)definite matrices are denoted by $\mathbb{S}^n_{++}$ ($\mathbb{S}^n_+$).
For $M_1,M_2\in\mathbb{S}^n$, we use $M_1\prec (\preceq) M_2$ and $M_1\succ (\succeq) M_2$ when
$M_1-M_2$ is negative (semi)definite and positive (semi)definite, respectively.
We use $\det (M)$ to denote the determinant of a square matrix $M$, and its trace is $\mathrm{tr}(M)$.
For $M \in \mathbb{S}^n$, $\lambda_\mathrm{max}(M)$ and $\lambda_\mathrm{min}(M)$ denote the maximum and minimum eigenvalues.
Given a complex matrix $N$, we use $N^\her$ to denote its conjugate transpose.
The Frobenius norm of $M$ is denoted by $\|M\|_F$, and
the maximum and minimum singular values are denoted by
$\sigma_\mathrm{max}(M)$ and
$\sigma_\mathrm{min}(M)$.

\section{Preliminaries} \label{section:LQR-formulation}

In this section, we introduce the linear quadratic regulator (LQR) formulation with a stochastic initial state and present our problem statement.

\subsection{Linear quadratic regulator (LQR)}
Consider a continuous-time linear time-invariant (LTI) dynamical system
\begin{equation}\label{eq:Dynamic}
\begin{aligned}
\dot{x}(t) &= Ax(t)+Bu(t), \qquad x(0)=x_0,
\end{aligned}
\end{equation}
where $x(t) \in \mathbb{R}^n$ is the vector of state variables and $u(t)\in \mathbb{R}^m$ is the vector of control inputs, and $x_0 \in \mathbb{R}^n$ is the initial state. We consider $B\neq0$ and assume $x_0$ is a random variable with 
zero mean and
covariance  $W = \mathbb{E}\left[x_0x_0^\tr \right] \in \mathbb{S}_+^n$.
We consider the following performance signal 
\begin{equation} \label{eq:performance-signal}
    z(t) = \begin{bmatrix}
        Q^{1/2} \\ 0
    \end{bmatrix} x(t) + \begin{bmatrix}
        0 \\ R^{1/2}
    \end{bmatrix} u(t),
\end{equation}
where $Q \succeq 0$ and $R\succ 0$ are performance weight matrices. 
A standard assumption is: 
\begin{assumption} \label{assumption:controllability}
    The performance weights satisfy $Q\in\mathbb{S}_{+}^n\setminus\{0\}$
    and
    $R\in\mathbb{S}^m_{++}$.
    Additionally, we have $(A, B)$ is stabilizable and $(Q^{1/2}, A)$ is detectable. 
\end{assumption}

Recall that $(A, B)$ is stabilizable if there exists a feedback matrix $K \in \mathbb{R}^{m \times n}$ such that $A+BK$ is Hurwitz stable, and that $(Q^{1/2}, A)$ is detectable if $(A^\tr, Q^{1/2})$ is stabilizable \cite[Chapter 3]{zhou1996robust}.
We consider a performance cost functional as 
\begin{equation*}
\mathfrak{J}_{\LQR}:= \mathbb{E}
\int_0^\infty \left[z(t)^\tr z(t)\right]dt  =
\mathbb{E}
\int_0^\infty \left[x(t)^\tr Qx(t) + u(t)^\tr Ru(t)\right]dt,  
\end{equation*}
where the expectation is taken with respect to the random initial state. 
The classical problem of the linear quadratic regulator (LQR) is formulated as
\begin{equation} \label{eq:LQR}
    \begin{aligned}
        \min_{u(t)} \quad & \mathfrak{J}_{\LQR} \\
        \text{subject to} \quad & ~\cref{eq:Dynamic},
    \end{aligned}
\end{equation}
where the input $u(t)$ can use all past state observation $x(\tau)$ with $\tau \leq t$. 
Note that the current form \cref{eq:LQR} is infinitely dimensional. It is well-known that the class of static state feedback policies of the form $u(t) = Kx(t)$ with $K\in \mathbb{R}^{m \times n}$
can achieve the global optimality for \cref{eq:LQR} \cite[Theorem 14.2]{zhou1996robust}. We will revisit the global optimality from a primal and dual perspective in \cref{section:global-optimality_static-linear-policies}. 

With the class of static state feedback policies, the policy optimization for \cref{eq:LQR} then reads as
\begin{equation} \label{eq:LQR-stochastic-noise}
\begin{aligned}
\min_{K \in \mathbb{R}^{m \times n}}\ \ &
J_{\mathtt{LQR}}(K):=\mathbb{E}\int_0^{\infty}
\left[x^\tr(t) Q x(t) + u^\tr(t) R u(t)\right]dt
\\
\text{subject to}\ \ & ~\cref{eq:Dynamic}, \quad u(t) = Kx(t), \quad
\text{$A+BK$ is Hurwitz stable.}
\end{aligned}
\end{equation}
We can equivalently reformulate \eqref{eq:LQR-stochastic-noise} as 
\begin{equation} \label{eq:LQR-H2}
\begin{aligned}
\min_{K \in \mathbb{R}^{m \times n}}\ \ & J_{\mathtt{LQR}}(K)   \\
\text{subject to}\ \ &
K\in\mathcal{K},
\end{aligned}
\end{equation}
where the set $\mathcal{K}$ represents the set of all the control gains that stabilize the closed-loop system, i.e.,
 $   \mathcal{K}:=
    \left\{
    K\in\mathbb{R}^{m\times n}
    \middle |
    A+BK \text{ is (Hurwitz) stable}
    \right\}$, and $J_{\mathtt{LQR}}(K)$ denotes the LQR cost under the policy $K \in \mathcal{K}$.

\subsection{Three fundamental aspects in LQR} 

The LQR problem \cref{eq:LQR-stochastic-noise}, as well as its stochastic process noise variant, has been extensively studied since Kalman's seminal work \cite{kalman1960contributions}. The policy optimization \cref{eq:LQR-H2} with a linear static policy parametrization has no loss of optimality (a fundamental result that, while well known, is not entirely trivial to prove). This form \cref{eq:LQR-H2} is very amenable for model-free design, and it has thus re-attracted increasing attention as a fundamental benchmark problem at the intersection of control and learning \cite{hu2023toward,recht2019tour,talebi2024policy}.   

Theoretical guarantees of policy optimization in LQR often rely on proper characterizations of its nonconvex geometry. In recent years, many favorable properties have been revealed for LQR~\cite{hu2023toward,talebi2024policy}. In this paper, we revisit three fundamental aspects of the LQR problem, including 1) the dual problem of a nonconvex matrix form of \eqref{eq:LQR-H2} and the strong duality (\Cref{section:strong-duality}); 2) the nuanced gradient dominance property in \eqref{eq:LQR-H2}, examining when it holds and when it fails (\cref{section:GD}); and 3) the global optimality of linear static policies (\cref{section:global-optimality_static-linear-policies}). 
While many of these results have been established with various techniques and settings, we here emphasize two important messages. 
\begin{enumerate}
    \item The LQR problem behaves almost as a convex problem: \textit{strong duality} holds under the usual \cref{assumption:controllability} (which serves as an analog to the Slater's condition).
    \item The LQR problem even exhibits properties of a strongly convex problem: \textit{gradient dominance} holds under a slightly stronger assumption (see \Cref{assumption:compactness} later).
\end{enumerate}
Establishing these results requires careful analyses of Riccati equations/inequalities and SDP duality. We also leverage the recent \texttt{ECL} framework \cite{zheng2023benign,zheng2024benign} to provide a streamlined proof of gradient dominance under the most general known setting. We finally remark that \cref{eq:LQR-H2} essentially involves a quadratic function and a quadratic matrix constraint in $K$, making its SDP relaxations closely resemble the classical S-lemma \cite{polik2007survey} (see  \cref{appendix:S-lemma}).    


\section{Strong Duality and Hidden Convexity in LQR}\label{section:strong-duality}

In this section, we investigate the strong duality and hidden convexity in LQR. In particular, 
we rewrite \cref{eq:LQR-H2} into a nonconvex matrix optimization problem, derive its Lagrange dual formulation, and establish strong duality between the primal and dual formulations. The celebrated algebraic Riccati equation for LQR naturally comes from solving the dual SDP and facilitates the KKT optimality analysis.

\subsection{A Lagrange dual formulation} \label{subsection:Lagrangian-dual}

Deriving a Lagrange dual formulation requires an algebraic characterization of both the cost and the stability constraint. It is well-known that for any $K \in \mathcal{K}$, the LQR cost $J_\LQR(K)$ can be evaluated as 
\begin{subequations} \label{eq:LQR-Lyapunov-H2}
\begin{equation} \label{eq:LQR-Lyapunov}
J_{\mathtt{LQR}}(K) = \mathrm{tr}(Q+K^\tr RK)X,
\end{equation}
where $X$ is the unique positive semidefinite solution to the following Lyapunov equation
\begin{equation} \label{eq:Lyapunov-H2}
(A+BK)X + X(A+BK)^\tr + W = 0.
\end{equation}
\end{subequations}
If $W\succ 0$,
we know that $K \in \mathcal{K}$ if and only if the solution $X$ to \cref{eq:Lyapunov-H2} is positive semidefinite.\footnote{We can further refine the solution $X$ to be positive definite.} 
In this case, it is now clear that 
\cref{eq:LQR-H2} is equivalent to the following nonconvex problem
\begin{equation} \label{eq:LQR-H2-SDP-1}
\begin{aligned}
p^* = \min_{K\in \mathbb{R}^{m \times n},X \in \mathbb{S}^n}\ \ & \mathrm{tr} (Q+K^\tr RK)X   \\
\text{subject to}\ \ & (A+BK)X + X(A+BK)^\tr + W = 0, X \succeq 0. 
\end{aligned}
\end{equation}
Note that $W \succ 0$ is a sufficient condition to ensure that the feasible set of \cref{eq:LQR-H2-SDP-1} is the same as $K \in \mathcal{K}$. However, even when $W \succeq 0$, \cref{eq:LQR-H2} and \cref{eq:LQR-H2-SDP-1} share the same optimal value. In this sense, these two problems are also equivalent. In the rest of this section, we allow the general case $W \succeq 0$. 

We next derive the Lagrange dual of \cref{eq:LQR-H2-SDP-1}, which can be explicitly formulated as an LMI.
We introduce a dual variable $P \in \mathbb{S}^n$ for the equality constraint in \cref{eq:LQR-H2-SDP-1}, and define a Lagrange function as
$$
L(K,X,P) = \langle Q+K^\tr RK, X \rangle + \langle (A+BK)X + X(A+BK)^\tr + W, P\rangle, \quad \forall X \succeq 0.   
$$
The dual function is then given by 
\begin{equation} \label{eq:dual-function}
g(P) := \inf_{K \in \mathbb{R}^{m \times n},X \succeq 0} L(K,X,P). 
\end{equation}
While not obvious in its current form, the dual function can be explicitly computed.

\begin{lemma} \label{lemma:Lagrangian-explicit}
    If $R \succ 0$, the dual function \cref{eq:dual-function} can be computed as 
    \begin{equation} \label{eq:dual-function-explicit-form}
    g(P) = \begin{cases}
    \langle W,P \rangle, & \text{if}\; \displaystyle \begin{bmatrix}
    A^\tr P + PA + Q & PB \\ B^\tr P & R
\end{bmatrix}\succeq 0,  \\
-\infty, & \text{otherwise}.
\end{cases}
    \end{equation}
\end{lemma}
The proof is not difficult but requires explicit calculations; we present the details in \cref{appendix:proof-of-dual-function}. 
With \cref{lemma:Lagrangian-explicit}, the Lagrange dual problem for the nonconvex LQR \cref{eq:LQR-H2-SDP-1} is of the form 
\begin{equation} \label{eq:Lagrange-dual}
    \begin{aligned}
        d^* = \max_{P \in \mathbb{S}^n} \;\;  &\langle W,P \rangle \\
        \text{subject to} \;\; & \begin{bmatrix}
    A^\tr P + PA + Q & PB \\ B^\tr P & R
\end{bmatrix}\succeq 0.
    \end{aligned}
\end{equation}
As expected, \cref{eq:Lagrange-dual} is a convex problem, since the Lagrange dual of any optimization problem~is~convex.
In particular, \cref{eq:Lagrange-dual}
takes the form of a convex semidefinite program (SDP). This type of SDPs has been extensively studied, tracing back to Willems’s seminal work  
\cite{willems1971least}.
As highlighted in many classical results (e.g., \cite{willems1971least}),
the constraint of \eqref{eq:Lagrange-dual} is equivalent to a \textit{Riccati inequality}, which has a strong connection to the standard algebraic Riccati equation (ARE) for LQR.

It is clear that any feasible solution to \cref{eq:Lagrange-dual} gives a lower bound to $p^*$ in the primal \cref{eq:LQR-H2-SDP-1}, meaning that
$d^* \leq p^*$.
This is standard by weak duality. We thus get an SDP relaxation \cref{eq:Lagrange-dual} for the~nonconvex LQR problem \cref{eq:LQR-H2-SDP-1} via the standard Lagrange dual process. An interesting result is that strong duality holds, i.e., $p^* = d^*$ under standard assumptions. We detail this result in the next~subsection.   

\begin{remark}[Positive (semi)definiteness of $W$]
Since $W$ represents the covariance matrix of  
initial state 
in \cref{eq:Dynamic}, it is always positive semidefinite. Classical LQR results also only require $W$ to be positive semidefinite. 
If $W$ is further positive definite,  any feasible point to \cref{eq:LQR-H2-SDP-1} naturally ensures the policy $K$ is stabilizing and vice versa. This is sufficient~for the equivalence between
\cref{eq:LQR-H2} and
\cref{eq:LQR-H2-SDP-1} and facilitates the Lagrange dual process in \cref{lemma:Lagrangian-explicit}.~The condition that $W\succ 0$ also guarantees other favorable properties, such as the uniqueness of the optimal solution to \Cref{eq:LQR-H2-SDP-1} (\cref{theorem:strong-duality-2}), strict complementarity in SDPs for LQR (\cref{theorem:strict-complementarity}), and gradient dominance for the nonconvex LQR policy formulation 
(\cref{theorem:J_LQR-gradient_dominance-general})
with explicit estimates of the parameters (\Cref{remark:PL-Q-W>0}).
\hfill $\square$
\end{remark}

\subsection{Strong duality in LQR} \label{subsection:strong-duality}

We now show the strong duality between the nonconvex LQR \cref{eq:LQR-H2-SDP-1} and its Lagrange dual \cref{eq:Lagrange-dual}.

\begin{theorem}[Strong duality]\label{theorem:strong-duality-1}
    With \Cref{assumption:controllability} and $W\in\mathbb{S}_+^n$, 
     the strong duality holds between the nonconvex LQR \cref{eq:LQR-H2-SDP-1} and its Lagrange dual \cref{eq:Lagrange-dual}, i.e., $p^* = d^*$.
\end{theorem}

The proof details are presented in \Cref{subsubsection:proof-strong-duality-1}. 
This result only requires \Cref{assumption:controllability} and the matrix $W$ can be positive semidefinite. \Cref{theorem:strong-duality-1} is not entirely surprising, as it is well-known that the LQR problem can be addressed through a convex reformulation. This further underscores the benign nature of the nonconvexity inherent in the LQR problem \cite{zheng2023benign,zheng2024benign}.
Our proof of \Cref{theorem:strong-duality-1} 
 leverages an equivalent nonconvex reformulation of \eqref{eq:LQR-H2-SDP-1}, which involves a nonconvex rank constraint.
By utilizing the complementary slackness for primal and dual SDPs (see \Cref{appendix:SDP-review} for a quick review), we observe that the nonconvex rank constraint is redundant, which leads to the strong duality between \eqref{eq:LQR-H2-SDP-1} and \eqref{eq:Lagrange-dual}. 
Since the nonconvex LQR \cref{eq:LQR-H2-SDP-1} involves a quadratic function and a single quadratic constraint in $K$ and $X^{1/2}$, our SDP formulation and its proof process are analogous to the classical S-lemma \cite{polik2007survey} (we present some details in \Cref{appendix:S-lemma}).

In addition to establishing strong duality, the primal-dual perspective allows us to explicitly construct the optimal solutions to both 
the LQR problem \eqref{eq:LQR-H2-SDP-1} and Lagrange dual \eqref{eq:Lagrange-dual}.
\begin{proposition}[Primal and dual optimal solutions]\label{theorem:strong-duality-2}
    With \Cref{assumption:controllability},  
     the following statements hold. 
     \begin{enumerate}
         \item An optimal solution to the dual SDP \eqref{eq:Lagrange-dual}
        is the unique positive semidefinite solution $P^\star$ to the following algebraic Riccati equation:      \begin{equation}\label{eq:Riccati-eq}
            A^\tr P+PA  - PBR^{-1}B^\tr P + Q = 0.
        \end{equation}
        \item An optimal solution to the primal nonconvex LQR \eqref{eq:LQR-H2-SDP-1} is $K^\star = -R^{-1}B^\tr P^\star$, where $P^\star$ is the positive semidefinite solution to \cref{eq:Riccati-eq}.  If $W \succ 0$, then the optimal feedback gain to \eqref{eq:LQR-H2-SDP-1} is unique. 
     \end{enumerate}
\end{proposition}
 
In \cref{theorem:strong-duality-2},
the first statement
shows that 
the Riccati equation \cref{eq:Riccati-eq} indeed solves the Lagrange dual problem \cref{eq:Lagrange-dual}. This result is well-known since \cite{willems1971least}, which directly comes from the Schur complement and a celebrated \textit{comparison theorem}.
The second statement is also well-known, which shows that we can recover an optimal solution to the primal nonconvex LQR problem from an optimal dual solution (i.e., $P^\star$ from \cref{eq:Riccati-eq}). This result requires some careful analysis of the Karush–Kuhn–Tucker (KKT) optimality condition for  
the SDP \cref{eq:SDP-primal} and its dual together with \cref{eq:Lagrange-dual}. We present the proof details
in \Cref{subsubsection:proofs-strong-duality-2}.

Classical proof techniques for LQR are typically based on a completion-of-squares technique \cite[Theorem 14.2]{zhou1996robust}  or variational analysis \cite[Theorem 6.1]{liberzon2011calculus}. We provide a quick review of the completion-of-squares technique in \cref{appendix:LQR}. Our proof technique in \Cref{subsubsection:proofs-strong-duality-2} is based on a primal and dual perspective, inspired by \cite{balakrishnan2003semidefinite}.  

\begin{remark}[Uniqueness of the optimal feedback gain]
The optimal feedback gain to \eqref{eq:LQR-H2-SDP-1} is guaranteed to be unique when $W \succ 0$. If $W \succeq 0$, then $K^\star = -R^{-1}B^\tr P^\star$ is one optimal solution, and there may exist many other optimal solutions to 
    \eqref{eq:LQR-H2-SDP-1}, as shown in \Cref{example:W-singular} 
    and \cref{proposition:non-unique-sol_line} below.
    We discuss some further details in \cref{appendix:non-uniqueness}.
    We note that with \Cref{assumption:controllability}, the optimal control input $u(t)$ to \cref{eq:LQR} is always unique \cite[Theorem 14.2]{zhou1996robust}; also see \cref{appendix:LQR}.  
    \hfill $\square$
\end{remark}

\begin{example}[Nonunique optimal feedback gains]
\label{example:W-singular}
     Consider an LQR instance with 
     \begin{equation*}
        A=\begin{bmatrix*}
            -1 & a \\
            a & -1
        \end{bmatrix*},\quad
        B= \begin{bmatrix}
            1\\ 1
        \end{bmatrix},\quad
         W = \begin{bmatrix}
         1 &-1\\ -1 &1
     \end{bmatrix}\succeq 0,\quad
     Q =  I \succ 0, \quad R=1.
     \end{equation*}
     We can choose a parameter $a > 0$ such that $A$ is stable.
     Then, \cref{assumption:controllability} is satisfied. 
     Consider the feedback gain $K=[k,k] $ with $k < -(a-1)/2$ that stabilizes the system as
     \begin{equation*}
         A+BK =
         \begin{bmatrix}
             -(1-k) & a+k\\
             a+k & -(1-k)
         \end{bmatrix}
         =\frac{1}{2}(a+2k-1)
         \begin{bmatrix}
             1 \\ 1
         \end{bmatrix}
         \begin{bmatrix}
             1 & 1
         \end{bmatrix}
         -\frac{1}{2}(1+a)
         \begin{bmatrix}
             -1 \\ 1
         \end{bmatrix}
         \begin{bmatrix}
             -1 & 1
         \end{bmatrix}.
     \end{equation*}
     For $K$,
     we can verify that the unique solution $X$ to \cref{eq:Lyapunov-H2} is independent of $k$, expressed as 
     \begin{align*}
     X=\int_0^\infty e^{(A+BK)t}We^{(A+BK)^\tr t} dt
        = \frac{1}{2(1+a)}\begin{bmatrix}
            1 & -1\\
            -1 & 1
        \end{bmatrix}.
     \end{align*}
     Thus, $J_\LQR(K)=\mathrm{tr}  (Q + K^\tr RK)X = 1/(1+a), \forall k < -(a-1)/2. $ 
     From the symmetry of the data matrices,
    it is straightforward to verify that~the optimal gain $K^\star = - R^{-1}B^\tr P^\star$ from the~ARE \cref{eq:Riccati-eq} must be of the form $K^\star = [k^\star, k^\star]$. 
      Thus, all feedback gains of the form $K=[k,k] $ with $k < -(a-1)/2$ are globally optimal to 
    \eqref{eq:LQR-H2-SDP-1}.~We illustrate the cost with $a=0.1$~in \cref{fig:lqr-landscape}~(c).
    \hfill $\square$
\end{example}

\begin{corollary}[Infinite nonunique solutions]\label{proposition:non-unique-sol_line}
Suppose \cref{assumption:controllability}.
Then, if we have two different optimal control gains $K^\star_1,\,K^\star_2\in\mathcal{K}$ to \eqref{eq:LQR-H2-SDP-1}, then
 $ K^\star_\eta = K_1^\star + \eta(K_1^\star-K_2^\star)$ 
is also optimal for any $\eta\in\mathbb{R}$ as long as $K^\star_\eta\in\mathcal{K}$.
\end{corollary}
\begin{proof}
Under \cref{assumption:controllability}, the optimal $u(t)=u^\star(t)$ to \cref{eq:LQR} is unique (see \cite[Theorem 14.2]{zhou1996robust}), implying that
    $u^\star(t)=K_1^\star x(t)=K_2^\star x(t)$ for all $ t\geq 0$.
Then, we have
\begin{equation*}
    K_\eta^\star x(t)
    = K_1^\star x(t)
    + \eta(K_1^\star-K_2^\star)x(t)= K_1^\star x(t)=u^\star(t).
\end{equation*}
Thus, $K_\eta^\star$ is also optimal as long as $K_\eta^\star\in\mathcal{K}$.
\end{proof}

\subsection{Primal and dual SDPs in LQR, and strict complementarity}\label{subsection:strict-complementarity}

We here derive an alternative SDP relaxation for \cref{eq:LQR-H2-SDP-1} from a primal Shor's relaxation \cite{shor1987quadratic}.
This primal SDP and \cref{eq:Lagrange-dual}, obtained from the Lagrange dual process, form a primal and dual pair of SDPs. This pair of SDPs not only satisfies strong duality but also enjoys a favorable property of \textit{strict complementarity}.  

For any feasible solution in \cref{eq:LQR-H2-SDP-1}, we introduce a new set of variables 
\begin{equation} \label{eq:construction-Z}
    Z_{11}=X, \quad Z_{12}=XK^\tr, \quad Z_{22}=K X K^\tr, \quad Z = \begin{bmatrix}
        Z_{11} & Z_{12} \\ Z_{12}^\tr & Z_{22}
    \end{bmatrix}, \qquad \forall X \succeq 0, K \in \mathbb{R}^{m \times n}.
\end{equation}
It is not difficult to verify that the original LQR problem \cref{eq:LQR-H2-SDP-1} with $W\succ0$ can be equivalently written as a rank-constrained SDP of the form\footnote{
One can show the equivalence between \cref{eq:LQR-H2-SDP-1} and \cref{eq:SDP-primal-rank} 
directly from the construction \cref{eq:construction-Z} when $W\succ0$.
Even when $W\nsucc0$, the problem \cref{eq:SDP-primal-rank} still preserves the same optimal value $p^*$ as \cref{eq:LQR-H2-SDP-1}, but the equivalence does not follow immediately from \cref{eq:construction-Z}.
We provide further details in \cref{section:equivalence}.
} 
\begin{equation} \label{eq:SDP-primal-rank}
    \begin{aligned}
    p^*=\min_{Z\in \mathbb{S}^{n + m}} \quad & \langle Q, Z_{11} \rangle +\langle R,Z_{22}\rangle \\
    \mathrm{subject~to} \quad & AZ_{11}+BZ_{12}^\tr +Z_{11}A^\tr +Z_{12}B^\tr + W= 0, \\ & Z=\begin{bmatrix}
        Z_{11} & Z_{12} \\ Z_{12}^\tr & Z_{22}
    \end{bmatrix}\succeq  0, \quad \mathrm{rank}(Z) \leq n.
\end{aligned}
\end{equation}
Dropping the rank constraint $\mathrm{rank}(Z) \leq n$, we get an alternative SDP relaxation 
\begin{equation} \label{eq:SDP-primal}
    \begin{aligned}
 p^*\geq  p_1^* = \min_{Z \in \mathbb{S}^{n + m}} \quad & \langle Q, Z_{11} \rangle +\langle R,Z_{22}\rangle \\    \mathrm{subject~to} \quad & AZ_{11}+BZ_{12}^\tr +Z_{11}A^\tr +Z_{12}B^\tr + W= 0, \\ & Z=\begin{bmatrix}
        Z_{11} & Z_{12} \\ Z_{12}^\tr & Z_{22}
    \end{bmatrix}\succeq  0.
\end{aligned}
\end{equation}
This primal SDP \cref{eq:SDP-primal} is also standard (known as Shor's relaxation). It is closely related to a classical change of variables  \cite{boyd1994linear}. 
We can verify that \cref{eq:SDP-primal,eq:Lagrange-dual} form a pair of primal and dual SDPs, which enjoy strong duality under standard assumptions. 

\begin{lemma} \label{lemma:primal-dual-SDPs-LQR}
Consider the SDPs \cref{eq:SDP-primal,eq:Lagrange-dual}. 
\begin{enumerate}
    \item They are primal and dual to each other for any $W \in \mathbb{S}^n$. 
    \item With \Cref{assumption:controllability} and $W \in \mathbb{S}^n_+$, strong duality holds between \cref{eq:SDP-primal,eq:Lagrange-dual}, i.e., $p_1^* = d^*$.
\end{enumerate}
\end{lemma}

The first result in \Cref{lemma:primal-dual-SDPs-LQR} follows from standard Lagrange dual arguments, which holds for any symmetric matrix $W$.
As for the second result, we show that the dual SDP \cref{eq:Lagrange-dual} is strictly feasible under \Cref{assumption:controllability} (note that we cannot guarantee \cref{eq:SDP-primal} is strictly feasible).
On the other hand,
assuming $W \in \mathbb{S}^n_+$
guarantees that the dual optimal value $d^*$ is bounded.
By Slater's condition, we know that strong duality holds between \cref{eq:Lagrange-dual,eq:SDP-primal}. We provide some details in \Cref{appendix:primal-dual-SDPs-LQR}.

In addition to strong duality, a useful rank constraint is always true for any optimal solution to \cref{eq:SDP-primal,eq:Lagrange-dual}. If we further have $W \succ 0$, the primal and dual SDPs \cref{eq:SDP-primal,eq:Lagrange-dual} enjoys a favorable notion of strict complementarity \cite{alizadeh1997complementarity,ding2023strict,liao2024inexact} in the following sense.
\begin{proposition}
\label{theorem:strict-complementarity}
    With \Cref{assumption:controllability} and $W \in \mathbb{S}^n_+$, there exists a pair of primal and dual optimal solutions $P^\star$ and $Z^\star$ to \cref{eq:SDP-primal,eq:Lagrange-dual}, such that 
    \begin{equation} \label{eq:rank-condition}
    \mathrm{rank} \left(\begin{bmatrix}
    A^\tr P^\star + P^\star A + Q & P^\star B \\ B^\tr P^\star & R
\end{bmatrix}\right) + \mathrm{rank}(Z^\star) \leq n + m.
    \end{equation}
    If  
    we further have $W \succ 0$, condition \cref{eq:rank-condition} holds with equality, i.e., strict complementarity~holds. 
\end{proposition} 
The inequality \cref{eq:rank-condition} is due to the complementary slackness. The notion of strict complementarity is the key to ensuring many nice properties in SDPs, including \textit{quadratic growth} and fast linear convergence when applying standard algorithms \cite{ding2023strict,liao2024inexact}. \Cref{theorem:strict-complementarity} confirms that the SDPs arising in LQR also enjoy this favorable property when $W\succ 0$.
The proof of \Cref{theorem:strict-complementarity} is given in \Cref{appendix:proof-strict-complementarity}.

\subsection{Technical proofs in \Cref{section:strong-duality}} \label{subsection:technical-proofs}

We here complete proof details for \Cref{theorem:strong-duality-1}, \cref{theorem:strong-duality-2}, and \cref{theorem:strict-complementarity}. 
The result in  \Cref{theorem:strong-duality-1} requires Riccati equations/inequalities and KKT analysis, and 
the proofs of \Cref{theorem:strong-duality-2}, and \cref{theorem:strict-complementarity} follow from some further refinement of  \Cref{theorem:strong-duality-1}.

Let us briefly review Riccati equations and inequalities. Given problem data $A, B, Q$, and $R$ with compatible dimensions, we define a quadratic matrix
\begin{equation} \label{eq:Ricatti-inequality}
         \mathcal{R}(P):= A^\tr P+PA-PBR^{-1}B^\tr P +Q.
\end{equation}
Then, $\mathcal{R}(P) = 0$ gives the same Riccati equation as \cref{eq:Riccati-eq}, and $\mathcal{R}(P) \succeq 0$ gives a nonstrict Riccati inequality. 
The following result is well-known  (\cite[Corollary 12.5]{zhou1996robust} and \cite[Lemma 3]{willems1971least}). 

\begin{lemma}\label{lemma:Riccati-equation}
    With \Cref{assumption:controllability}, the Riccati equation $\mathcal{R}(P) = 0$ has a unique positive semidefinite solution $P \succeq 0$. This solution further satisfies the properties below.  
    \begin{enumerate}
        \item It is stabilizing, i.e., $A - BR^{-1}B^\tr P$ is stable.
        \item It is maximal in the sense that $P \succeq \hat{P}$ for all $\hat{P} \in \mathbb{S}^n$ satisfying the Riccati inequality 
        $\mathcal{R}(\hat{P}) \succeq 0$. 
    \end{enumerate}
\end{lemma}

\Cref{lemma:Riccati-equation} is central to the LQR problem, and it only requires stabilizability and detectability. We also have the following equivalence connecting the stabilizing solution of the Riccati equation with the feasibility of the strict Riccati inequality.
\begin{lemma}{\cite[Theorem 2.23]{scherer1990riccati}} \label{lemma:Riccati-inequality}
    Suppose $R \succ 0$, $Q \in \mathbb{S}^n$, and $(A, B)$ is stabilizable. The following two statements are equivalent: 
    \begin{enumerate}
        \item There exists a symmetric solution to the strict  Riccati inequality $\mathcal{R}({P}) \succ 0$.
        \item There exists a symmetric solution $P$ to $\mathcal{R}(P) = 0$ such that $A- BR^{-1}B^\tr P$ is stable. 
    \end{enumerate}
\end{lemma}

\subsubsection{Proof of \Cref{theorem:strong-duality-1}}\label{subsubsection:proof-strong-duality-1} 
The proof requires the use of the Lagrange dual SDP \cref{eq:Lagrange-dual}, the primal Shor's SDP \cref{eq:SDP-primal}, and the rank-constrained SDP \cref{eq:SDP-primal-rank}, as well as careful KKT optimality analysis. We divide the proof into two key steps. 
\begin{itemize}
    \item \textbf{Step 1}: strong duality holds between the Lagrange dual SDP \cref{eq:Lagrange-dual} and the primal Shor's SDP \cref{eq:SDP-primal}. This is essentially stated in \Cref{lemma:primal-dual-SDPs-LQR}. Thus, we have 
    \begin{equation} \label{eq:sandwich-1}
    p_1^* = d^* \leq p^*,
\end{equation}    
where the last inequality is due to the standard weak duality for \cref{eq:Lagrange-dual,eq:LQR-H2-SDP-1}. 
\item \textbf{Step 2}: By the KKT analysis between \cref{eq:Lagrange-dual} and \cref{eq:SDP-primal}, we show that there \textit{exists} a pair of optimal solutions $P^\star$ and $Z^\star$ such that 
\begin{equation}  \label{eq:LQR-rank-constraint}
\mathrm{rank}(Z^\star) + \mathrm{rank} \left(\begin{bmatrix}
    A^\tr P^\star + P^\star A + Q & P^\star B \\ B^\tr P^\star & R
\end{bmatrix}\right)  \leq n + m.
\end{equation}
Since $R \succ 0$, we know $\mathrm{rank}(Z^\star) \leq n$, which implies this solution $Z^\star$ is feasible to the rank-constrained SDP \cref{eq:SDP-primal-rank}. Therefore, we have 
\begin{equation}\label{eq:sandwich-2}
p^* \leq p_1^*. 
\end{equation}
\end{itemize}

Combining \cref{eq:sandwich-1,eq:sandwich-2} confirms the strong duality that $p^* = d^*$ in \Cref{theorem:strong-duality-1} with $W\succ0$.
When $W\nsucc0$, the proof of \cref{theorem:strong-duality-2} ensures the equivalence between \cref{eq:LQR-H2-SDP-1} and \cref{eq:SDP-primal-rank} (see \Cref{remark:W-psd-rank});
we then obtain $p^*\leq p_1^*$,
leading to the strong duality $p^* = d^*$.

We next proceed to prove these two key steps. 

\textbf{Proof of Step 1:} \cref{lemma:primal-dual-SDPs-LQR} shows that the Lagrange dual SDP \cref{eq:Lagrange-dual} and the primal Shor's SDP \cref{eq:SDP-primal} are primal and dual to each other. Consider the dual SDP \cref{eq:Lagrange-dual}. Since $R\succ 0$, the Schur complement gives the following equivalence 
    \begin{equation} \label{eq:schur-complement}
         \begin{bmatrix}
    A^\tr P + PA + Q & PB \\ B^\tr P & R
        \end{bmatrix}\succeq 0
        \quad \Leftrightarrow \quad
        A^\tr P+PA+Q-PBR^{-1}B^\tr P \succeq 0.
    \end{equation}
Then, \cref{lemma:Riccati-equation} ensures that the unique positive semidefinite solution $P^\star$ to the~Riccati equation \cref{eq:Riccati-eq} is maximal for all feasible solutions to \eqref{eq:Lagrange-dual}. Thus, this positive semidefinite solution $P^\star$ is an optimal solution to \eqref{eq:Lagrange-dual} (which may not be unique). 

We now know that \eqref{eq:Lagrange-dual} has an optimal solution $P^\star$ with a finite cost value. Furthermore,
\Cref{lemma:Riccati-inequality} guarantees that \eqref{eq:Lagrange-dual} is strictly feasible (Slater's condition holds). Hence, strong duality holds between \cref{eq:Lagrange-dual,eq:SDP-primal}, and we further know the primal SDP \cref{eq:SDP-primal} is also solvable by standard SDP duality theory (see \cref{appendix:SDP-review} and \cite[Theorem 3.1]{vandenberghe1996semidefinite}). 

\vspace{2pt}

\textbf{Proof of Step 2:} from the analysis in Step 1, the SDP \cref{eq:SDP-primal} has at least one optimal solution $Z^\star$ which achieves $ p_1^* =  \langle Q, Z_{11}^\star  \rangle +\langle R,Z_{22}^\star \rangle$.
Let $P^\star$ be the positive semidefinite solution to the Riccati equation \eqref{eq:Riccati-eq}, which is an optimal solution to \eqref{eq:Lagrange-dual}. 
    Then, $(Z^\star,P^\star)$ is a pair of primal and dual solutions, and they satisfy the following KKT condition:
    \begin{subequations} \label{eq:KKT}
    \begin{align}
        \label{eq:KKT-LQR-H2-SDP-1}
        AZ_{11}^\star + B(Z_{12}^\star)^\tr + Z_{11}^\star A^\tr     + Z_{12}^\star B^\tr + W &= 0,\quad
        Z^\star\in\mathbb{S}_+^{n+m},
        \\
        \label{eq:KKT-Lagrange_dual}
        \begin{bmatrix}
        A^\tr P^\star + P^\star A + Q & P^\star B \\ B^\tr P^\star  & R
        \end{bmatrix} &\succeq 0,
        \\
        \label{eq:KKT-complementary_slackness}
        \left\langle
        Z^\star,\begin{bmatrix}
        A^\tr P^\star + P^\star A + Q & P^\star B \\ B^\tr P^\star  & R
        \end{bmatrix}
        \right\rangle&=0.
        \end{align}
    \end{subequations}
    It is clear that \cref{eq:KKT-LQR-H2-SDP-1} is primal feasibility,  \cref{eq:KKT-Lagrange_dual} is dual feasibility, and  \eqref{eq:KKT-complementary_slackness} is the complementary slackness. From \eqref{eq:KKT-complementary_slackness}, we see that \cref{eq:LQR-rank-constraint} must~hold (since $\langle M,N\rangle = 0$ guarantees that $\mathrm{rank}(M) + \mathrm{rank}(N) \leq d$ for any $M,N \in \mathbb{S}^d_+$).

\subsubsection{Proof of \Cref{theorem:strong-duality-2}}\label{subsubsection:proofs-strong-duality-2}

The first statement of \cref{theorem:strong-duality-2} is a direct consequence of \cref{lemma:Riccati-equation}, as we have discussed in \cref{eq:schur-complement} above. 
We show the second statement of \cref{theorem:strong-duality-2} via KKT analysis. 

As we proved in \textbf{Proof of Step 2} above, both \cref{eq:Lagrange-dual,eq:SDP-primal} are solvable. Let $Z^\star$ be an optimal solution to  \cref{eq:SDP-primal} and $P^\star$ be the positive semidefinite solution to the Riccati equation \eqref{eq:Riccati-eq}. 
This pair satisfies the KKT optimality condition \cref{eq:KKT}. 
The calculation below essentially recovers the primal solution $Z^\star$ from the dual solution $P^\star$ through KKT analysis. 
    
    Since $R \succ 0$ (which has rank $m$), we have $\mathrm{rank}(Z^\star)\leq n$ from \cref{eq:LQR-rank-constraint}. Without loss of generality, we let  
    \begin{equation} \label{eq:Z-star-low-rank}
        Z^\star = 
        \begin{bmatrix}
        U_1\\
        U_2
        \end{bmatrix}
        \begin{bmatrix}
        U_1\\
        U_2
        \end{bmatrix}^\tr.
    \end{equation}
    with some $U_1\in\mathbb{R}^{n\times n}$ and $U_2\in\mathbb{R}^{m\times n}$. 
Since $P^\star$ comes from \eqref{eq:Riccati-eq}, we have 
    \begin{equation} \label{eq:KKT-factorization}
        \begin{bmatrix}
            A^\tr P^\star + P^\star A + Q & P^\star B\\
            B^\tr P^\star & R 
        \end{bmatrix}
        =
        \begin{bmatrix}
            P^\star BR^{-1}B^\tr P^\star & P^\star B\\
            B^\tr P^\star & R
        \end{bmatrix}
        =
        \begin{bmatrix}
            P^\star B R^{-1/2}\\
            R^{1/2}
        \end{bmatrix}
        \begin{bmatrix}
            P^\star B R^{-1/2}\\
            R^{1/2}
        \end{bmatrix}^\tr.
    \end{equation}
    Combining the complementary slackness \eqref{eq:KKT-complementary_slackness} with \cref{eq:KKT-factorization,eq:Z-star-low-rank}, we obtain
    \begin{equation*}
       \begin{bmatrix}
            P^\star B R^{-1/2}\\
            R^{1/2}
        \end{bmatrix}^\tr
        \begin{bmatrix}
            U_1\\
            U_2
        \end{bmatrix}=0
        \quad
        \Rightarrow
        \quad
        R^{-1/2}B^\tr P^\star U_1 + R^{1/2}U_2 =0
        \quad\Rightarrow\quad
        U_2 = - R^{-1}B^\tr P^\star U_1.
    \end{equation*}
    Substituting this back to \cref{eq:Z-star-low-rank}, the minimizer $Z^\star$ must be in the form of 
    \begin{align}\label{eq:Z_star}
        Z^\star
        =
        \begin{bmatrix}
        U_1\\
        - R^{-1}B^\tr P^\star U_1
        \end{bmatrix}
        \begin{bmatrix}
        U_1\\
        - R^{-1}B^\tr P^\star U_1
        \end{bmatrix}^\tr
        =
        \begin{bmatrix}
            Z_{11}^\star & - Z_{11}^\star P^\star B R^{-1}\\
            -R^{-1}B^\tr P^\star Z_{11}^\star &
            R^{-1}B^\tr P^\star Z_{11}^\star P^\star B R^{-1}
        \end{bmatrix}
        .
    \end{align}
    
    Now, the primal feasibility \eqref{eq:KKT-LQR-H2-SDP-1} indicates that 
    \begin{align}\label{eq:primal-feasibility-w/Zstar}
    0=&AZ_{11}^\star + B\left(- Z_{11}^\star P^\star B R^{-1}\right)^\tr + Z_{11}^\star A^\tr     + \left(- Z_{11}^\star P^\star B R^{-1}\right) B^\tr + W \nonumber \\
    =&(A-BR^{-1}B^\tr P^\star)Z_{11}^\star + Z_{11}^\star(A-BR^{-1}B^\tr P^\star)^\tr + W. 
    \end{align}
    Since $(A-BR^{-1}B^\tr P^\star)$ is stable by \cref{lemma:Riccati-equation} and $W \succeq 0$, we know that the solution $Z_{11}^\star$ from the Lyapunov equation \eqref{eq:primal-feasibility-w/Zstar} is unique and positive semidefinite. 
    Therefore, $Z^\star$ in \eqref{eq:Z_star} from the KKT condition must be unique. 
    We define $X=Z_{11}^\star$ and $K=-R^{-1}B^\tr P^\star$. From \cref{eq:primal-feasibility-w/Zstar}, it is clear that this pair $(X, K)$ is feasible to \eqref{eq:LQR-H2-SDP-1}.  Since $Z^\star$ is optimal to \cref{eq:SDP-primal}, this solution $K=-R^{-1}B^\tr P^\star$ is also one optimal solution to  \eqref{eq:LQR-H2-SDP-1}. 
    
    At this stage, we know $Z^\star$ is unique but cannot infer the uniqueness~of~$K^\star$.
    If we further have $W \succ 0$, then $K^\star = -R^{-1}B^\tr P^\star$ is the unique optimal solution to  \eqref{eq:LQR-H2-SDP-1}.   
This uniqueness is a consequence of two facts: 1) the optimal $Z^\star$ to \cref{eq:SDP-primal} is unique, and 2) $X=Z_{11}^\star $ from \cref{eq:primal-feasibility-w/Zstar} is positive definite. {Suppose there exists another optimal solution $\hat{K} \neq -R^{-1}B^\tr P^\star$. Since $W \succ 0$, this solution must be associated with a positive definite $\hat{X} \succ 0$ in \cref{eq:Lyapunov-H2}}. From \cref{eq:construction-Z}, we can construct another optimal solution $\hat{Z}$ to \cref{eq:SDP-primal}. This contradicts the uniqueness of $Z^\star$ from the KKT analysis. This completes the proof of \Cref{theorem:strong-duality-2}.

\subsubsection{Proof of \cref{theorem:strict-complementarity}} \label{appendix:proof-strict-complementarity}
This result is a refinement of the arguments for \Cref{theorem:strong-duality-1} and \Cref{theorem:strong-duality-2}. In particular, the rank inequality \cref{eq:rank-condition} is the same as \cref{eq:LQR-rank-constraint}, which comes from the complementary slackness \cite{blekherman2012semidefinite}. 

If $W \succ 0$, 
the solution $Z_{11}^\star$ from the Lyapunov equation \eqref{eq:primal-feasibility-w/Zstar} must be positive definite. In this case, we must have $\mathrm{rank}(Z^\star) = n$, and the second statement in \Cref{theorem:strict-complementarity} holds.


\section{Gradient Dominance via Extended Convex Lifting}\label{section:GD}

In \Cref{section:strong-duality},we have seen that the nonconvex LQR \cref{eq:LQR-H2},
after suitable reformulations,
enjoys strong duality. In this section, we revisit another elegant property of \cref{eq:LQR-H2}, called \textit{gradient dominance}, from the perspective of hidden convexity. 
This gradient dominance property was first discussed in \cite{fazel2018global,mohammadi2019global}. 
Our proof strategy is based on the convex reformulation \cref{eq:SDP-primal} and the notion of Cauchy directions under a recent framework of Extended Convex Lifting (\texttt{ECL}) \cite{zheng2024benign}.

\subsection{Gradient dominance in LQR}

Consider a differentiable function $f: \mathcal{D} \to \mathbb{R}$ with an open domain $\mathcal{D} \subseteq \mathbb{R}^n$, and denote its sublevel set as $\mathcal{D}_\nu:= \{x \in \mathcal{D} \mid f(x) \leq \nu\}$. We call $f$ to be \textit{gradient dominated} over $\mathcal{D}_\nu$ if there exists a constant $\mu > 0$ such that 
\begin{equation} \label{eq:PL-f}
    \mu\left(f(x) - f(x^\star) \right)\leq \frac{1}{2} \|\nabla f(x)\|^2, \qquad \forall x \in \mathcal{D}_\nu, 
\end{equation}
where $x^\star$ is an optimal solution of $f$ over $\mathcal{D}$. Note that \cref{eq:PL-f}  is a local version, and the constant $\mu$ in \cref{eq:PL-f} may depend on the sublevel set $\mathcal{D}_\nu$. If \cref{eq:PL-f} holds for any feasible point in $\mathcal{D}$, we say that $f$ satisfies gradient dominance globally. Functions satisfying the gradient dominance \cref{eq:PL-f}, also known as the PL inequality, can go beyond convexity; see \cite{liao2024error,karimi2016linear} for different regularity conditions.   

    For functions satisfying gradient dominance, it is well-known that basic gradient descent achieves linear convergence.
We summarize a version of this result below. 
Recall that a function $f:\mathcal{D}\to\mathbb{R}$ is \textit{coercive} if 
 $   f(x)\to \infty
    \text{ whenever }
    \|x\|\to \infty \text{ or }
    x \to x_\infty \in\partial \mathcal{D},$
where $\partial \mathcal{D}$ represents the boundary~of~$\mathcal{D}$.

\begin{lemma}[{\cite[Theorem 4.2]{fatkhullin2021optimizing}}]
\label{lemma:convergence-gradient-descent-GD}
    Consider $f^* = \min_{x\in\mathcal{D}} f (x)$, where f is twice continuously differentiable and coercive, and $\mathcal{D} \subseteq \mathbb{R}^n$ is open.
    Suppose that
    the gradient dominance~\cref{eq:PL-f} holds with $\mu > 0$ over a sublevel set $\mathcal{D}_\nu$. 
    Then, $f$ is $L$-smooth over $\mathcal{D}_\nu$,
 and the gradient descent algorithm $x_{k+1} = x_k  - \alpha \nabla f(x_k)$ with a constant stepsize $0 <\alpha < 2/L$ and an initialization $x_0 \in \mathcal{D}_\nu$ satisfies  
    \begin{equation}
        f(x_{k+1}) - f^* \leq \gamma \times (f (x_k) - f^*),
    \end{equation}
    where $\gamma = 1-\mu\alpha(2-L\alpha) < 1$. If $\alpha = 1/L$, we have $\gamma = 1 - \mu/L$.
\end{lemma}

This lemma guarantees the feasibility of all iterates, as well as the linear convergence of the cost value to the global minimum $f^*$. 
We provide a proof in \cref{appendix:proof-GD-2} for self-completeness. 
Our proof is inspired by \cite[Theorem 4.2]{fatkhullin2021optimizing}.
The range of step size $\alpha$ ensuring the feasibility of all iterates is derived from 
the $L$-smoothness of $f$ and
a line search procedure over $\mathcal{D}_\nu$.

\vspace{-1mm}

\begin{remark}[Feasibility of iterates] 
The problem in \cref{lemma:convergence-gradient-descent-GD} is a constrained optimization problem. Thus, it is important to ensure that all the iterates from the gradient descent are still feasible and in the initial sublevel set. We can then bypass the testing of feasibility and computation of a projection. 
The feasibility guarantee of all iterates $x_k$ may not be obvious and requires some technical analysis; see \cref{appendix:proof-GD-2}.
A different approach showing the feasibility can be found in \cite[Theorem 1]{hu2023toward}, in which a slightly conservative range of step sizes was derived. 
    \hfill $\square$
\end{remark}

\vspace{-1mm}

Before presenting the gradient dominance for the LQR \cref{eq:LQR-H2} with the cost $J_\LQR$ evaluated as 
 \cref{eq:LQR-Lyapunov-H2}, we need another technical condition. \vspace{-1mm}
\begin{assumption} \label{assumption:compactness}
Consider the LQR \cref{eq:LQR-H2} with the cost $J_\LQR$  
 \cref{eq:LQR-Lyapunov-H2}.  The following conditions hold:
\begin{enumerate}
    \item 
    The matrix $X$ in \cref{eq:Lyapunov-H2} is positive definite 
    if and only if  $K\in\mathcal{K}$.
    \item For any finite $\nu\geq \min_{K\in\mathcal{K}}J_\LQR(K)$, the sublevel set $\mathcal{K}_\nu := \{K \in \mathcal{K} \mid J_\LQR(K) \leq \nu\}$  is compact.
\end{enumerate}
\end{assumption}

The first part of \Cref{assumption:compactness} ensures that the LQR problem admits an exact convex reformulation using a classical change of variables, and the second part is used to derive a positive constant for the gradient dominance.
\Cref{assumption:compactness} is certainly satisfied when 
\cref{assumption:controllability} and $W \succ 0$ are true, as it is widely used in \cite{fatkhullin2021optimizing,bu2019lqr,bu2020policy,zhang2025convergence}. We provide further discussions on \cref{assumption:compactness} in the next subsection.~With \cref{assumption:compactness,assumption:controllability}, we can establish a general version of gradient dominance for the LQR \cref{eq:LQR-H2}.

\begin{theorem}[Gradient dominance]\label{theorem:J_LQR-gradient_dominance-general} 
    Suppose \Cref{assumption:controllability,assumption:compactness} hold. Then, the optimal feedback gain $K^\star$ to \cref{eq:LQR-H2} is unique, and for any $\nu > J_\LQR(K^\star)$, there exists a constant $\mu > 0$ such that
    \begin{subequations} \label{eq:PL-LQR-condition}
    \begin{equation} \label{eq:PL-LQR}
        \mu\left(J_\LQR(K)-J_\LQR(K^\star)\right) 
        \leq \frac{1}{2} \|\nabla J_\LQR(K)\|_F^2, \qquad \forall K \in \mathcal{K}_\nu. 
    \end{equation}
    Moreover, we can choose the constant 
     $   \mu = \mu_\mathrm{qg} c_\LQR^2$ 
    with
        \begin{align} \label{eq:PL-constant}
\mu_{\mathrm{qg}}=        \min\left\{
       \frac{\lambda_\mathrm{min}(R)}{\overline{\kappa}}
        ,
       \frac{\lambda_\mathrm{min}(R)}{\overline{\kappa}
         \|
    \mathcal{A}_\star ^{-1}\circ\mathcal{B}
    \|_2^2}
    \right\}, \qquad c_\LQR=
        \frac{{\underline{\kappa}}}{2\sqrt{n}}
        \left(1 + \sqrt{\frac{\nu-J_\LQR(K^\star)}{\underline{\kappa}\lambda_\mathrm{min}(R)}}\right)^{-1},
    \end{align}
    where $\underline{\kappa}
    =\min_{K \in \mathcal{K}_\nu}  \lambda_\mathrm{min}(X)$, $\overline{\kappa}=
        \max_{K \in \mathcal{K}_\nu}  
    \lambda_\mathrm{max}(X)$, 
     $X$ is from  \cref{eq:Lyapunov-H2}, $\|\cdot\|_2$ denotes the operator $2$-norm, 
    and the linear maps 
    $\mathcal{A}_\star:\mathbb{S}^{n}\to\mathbb{S}^{n}$ 
    and $\mathcal{B}:\mathbb{R}^{m\times n}\to\mathbb{R}^{n\times n}$ 
    are defined by
    $\mathcal{A}_\star(X)= (A+BK^\star)X+X(A+BK^\star)^\tr $ and
$\mathcal{B}(Y)= BY+Y^\tr B^\tr$, respectively.
    \end{subequations}
    \end{theorem}

The uniqueness of the optimal LQR gain is ensured by the positive definiteness of $X$ in \Cref{assumption:compactness}. We will establish \cref{eq:PL-LQR-condition} as a special case in the recent $\ECL$ framework \cite{zheng2024benign}. When $X$ from \cref{eq:Lyapunov-H2} is always positive definite for $K\in\mathcal{K}$, we can construct an equivalent convex reformulation $f_\mathrm{cvx}$ for $J_\LQR$ using the $\ECL$ process. We show that this convex reformulation $f_\mathrm{cvx}$ further satisfies \textit{quadratic growth} and thus gradient dominance 
(see \cref{appendix:GD-review}). We refine the existing $\ECL$ framework to incorporate quadratic growth, and leverage the concept of Cauchy directions \cite{lemarechal2012cauchy}
to confirm the gradient dominance for $J_\LQR$.  
This proof strategy is different from the existing results  \cite{fazel2018global,fatkhullin2021optimizing,bu2019lqr,bu2020policy,zhang2025convergence} because of our use of \texttt{ECL} and Cauchy directions. We remark that the convex reformulation in LQR satisfies quadratic growth under \cref{assumption:compactness}, which seems to have been overlooked before. 
In fact, the existing studies \cite{mohammadi2019global,fatkhullin2025stochastic} instead imposed strong convexity, posing a limitation on the applicability.
We provide the \texttt{ECL} constructions and complete the proof in \Cref{subsubsection:LQR-PL-proof}.

The gradient dominance condition \cref{eq:PL-LQR} has a strong geometric interpretation. It~not only~ensures that any stationary point is globally optimal but also quantifies the global suboptimal gap from local gradient information. 
With the same setup in \Cref{theorem:J_LQR-gradient_dominance-general}, there exists a constant $\mu_{\mathrm{q}} > 0$ such that 
    \begin{equation*}
        J_\LQR(K)-J_\LQR(K^\star)  \geq \frac{\mu_{\mathrm{q}}}{2}\|K - K^\star\|_F^2, \qquad \forall K \in \mathcal{K}_\nu, 
    \end{equation*}
which gives a quadratic lower bound for $J_\LQR(K)$. Similarly,  
with the suitable choice of the smoothness constant $L>0$,
we also have 
 \begin{equation*} 
        J_\LQR(K)-J_\LQR(K^\star)  \leq \frac{L}{2}\|K - K^\star\|_F^2, \qquad \forall K \in \mathcal{K}_\nu,
    \end{equation*}
which gives a quadratic upper bound.

\begin{remark}[Gradient dominance in LQR] \label{remark:PL-LQR}
    {The gradient dominance property in LQR was first revealed in \cite{fazel2018global} for the discrete-time setting. The continuous-time counterpart was established in~\cite{mohammadi2019global}. This elegant property leads to many favorable results in model-free policy gradient algorithms, as seen in \Cref{lemma:convergence-gradient-descent-GD}; see \cite{hu2023toward,talebi2024policy} for recent surveys. Gradient dominance has been further discussed in later works  \cite{fatkhullin2021optimizing,bu2019lqr,bu2020policy,zhang2025convergence}. All these results require $W\succ 0$, which is a sufficient condition for \cref{assumption:compactness}. In the continuous-time setting, the works \cite{mohammadi2019global,fatkhullin2021optimizing,zhang2025convergence} further require the performance weight matrix $Q$ to be strictly positive definite.
    The result in \cite{bu2020policy} allows a positive semidefinite $Q \succeq 0$. Our gradient dominance result in \Cref{theorem:J_LQR-gradient_dominance-general} arguably provides the most general setup as our \cref{assumption:compactness} allows for $W\succeq 0$ and $Q\succeq 0$ in certain cases; see \cref{subsection:conditions-on-compactness}. Our proof for \Cref{theorem:J_LQR-gradient_dominance-general} is based on a recent \texttt{ECL} framework, introduced in \cite{zheng2024benign}, and we refine the \texttt{ECL} analysis by utilizing the notion of Cauchy directions \cite{lemarechal2012cauchy}.
    Note that when both $Q$ and $W$ are positive definite, we have more explicit estimates of the constants in \cref{theorem:J_LQR-gradient_dominance-general},
    as detailed in \Cref{remark:PL-Q-W>0}.  
    \hfill $\square$}
\end{remark}
\begin{remark}[Gradient dominance when both $Q$ and $W$ are positive definite]\label{remark:PL-Q-W>0}
When the matrices $Q$ and $W$ are positive definite,
\cref{assumption:compactness} is always satisfied. In this case, we can derive
more explicit estimates of $\overline{\kappa}$ and $\underline{\kappa}$ 
on $\mathcal{K}_\nu$
as
$
0< \frac{\kappa}{\nu}\leq \underline{\kappa}\leq \overline{\kappa} \leq \frac{\nu}{\lambda_\mathrm{min}(Q)}
$ 
with
\begin{equation*}
        \kappa=\frac{\lambda_{\min }^2(W)}{4}\left(\frac{\|A\|_2}{\sqrt{\lambda_{\min }(Q)}}+\frac{\|B\|_2}{\sqrt{\lambda_{\min }(R)}}\right)^{-2}.
\end{equation*}
The bounds of $\underline{\kappa}$ and $\kappa$ are from \cite{zare2019proximal}, and
the bound of $\overline{\kappa}$ follows from a simple calculation. \hfill $\square$
\end{remark}

\subsection{Sufficient conditions for \cref{assumption:compactness}} \label{subsection:conditions-on-compactness}

In addition to the standard \Cref{assumption:controllability}, \cref{theorem:J_LQR-gradient_dominance-general} requires another two conditions in \cref{assumption:compactness}: the compactness of the sublevel set and the positive definiteness of $X$ from \cref{eq:Lyapunov-H2}. 
These two conditions are important in establishing \cref{eq:PL-LQR-condition}. We here show examples to show that these two conditions cannot be relaxed in general, and also provide sufficient conditions to ensure \cref{assumption:compactness}.

We first continue the LQR instance in \cref{example:W-singular} to show that the compactness of $\mathcal{K}_\nu$ cannot be relaxed in general. 

\begin{figure}
\setlength{\abovecaptionskip}{0pt}
    \centering
    \includegraphics[width=\linewidth]{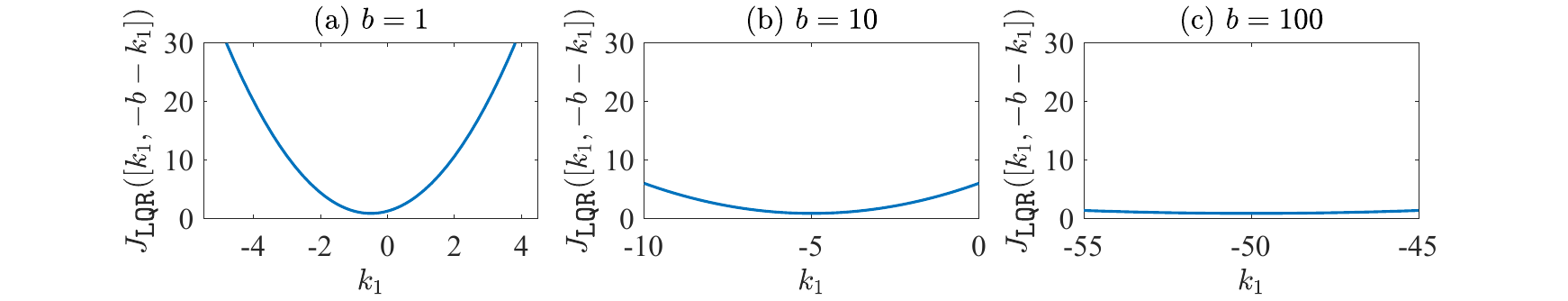}
    \caption{The slices of $J_\LQR([k_1,k_2])$
    with $k_2=-b-k_1$
    in \cref{example:W-singular} for     
    $b=1$, $10$, and $100$.
    We set the range of $k_1$ as $k_1\in\left[-5-{b}/{2},5-{b}/{2}\right]$ for each $b$.
    As $b$ increases,
    the local shape becomes flatter.}
    \label{fig:J_LQR_for_three-c}
\end{figure}

\begin{example}[No guarantee of gradient dominance over an unbounded sublevel set $\mathcal{K}
_\nu$] \label{example:unbounded-sublevel-set}
Here, we show that gradient dominance fails over $\mathcal{K}
_\nu$ for the LQR instance in \cref{example:W-singular}. From \Cref{fig:lqr-landscape}(c), it is clear that its sublevel set  $\mathcal{K}
_\nu$ is unbounded (as the optimal feedback gains are unbounded). 
\cref{fig:J_LQR_for_three-c}  
shows the slices of $J_\LQR(K)$ for
$K=[k,-b-k]$ with $b=1$, $10$, and $100$. We can see that the larger $b$ is, the flatter
$J_\LQR$ becomes. This implies that $J_{\mathrm{LQR}}$ cannot be uniformly lower bounded by a quadratic function over $\mathcal{K}_\nu$, i.e., quadratic growth fails. (It is known that quadratic growth is generally a weaker property than gradient dominance as \cref{appendix:GD-review}.) There exists no uniform constant $\mu>0$ satisfying \cref{eq:PL-LQR} over $\mathcal{K}_\nu$.
Indeed, for
$K=[k,-b-k]$, we can compute
\begin{equation*}
    \frac{\|\nabla J_\LQR(K)\|_F^2}{2\left(J_\LQR(K)-J_\LQR(K^{\star})\right)}=
    \mathcal{O}(b^{-1})\to 0\quad
    \text{when}\quad b\to\infty.
\end{equation*}
Thus, gradient dominance fails. \hfill $\square$
\end{example}

The LQR instance in \Cref{example:W-singular} satisfies \cref{assumption:controllability}, and we also have $W \succeq 0$ and $Q \succ 0$. Thus, the discussions in \Cref{example:W-singular,example:unbounded-sublevel-set} lead to the following corollary. 

\begin{corollary}
Suppose \cref{assumption:controllability} holds and $W \succeq 0$. The LQR cost function \cref{eq:LQR-Lyapunov-H2} may be neither coercive nor gradient dominated even if $Q \succ 0$.
\end{corollary}

Recall from \cref{example:W-singular} that if $W \succeq 0$, the optimal LQR feedback gains may not be unique.  
The positive definiteness of $X$ in \Cref{assumption:compactness} is also very important. Especially, 
to derive a universal positive constant $\mu>0$,
we require that the minimum eigenvalue of $X$ to be bounded away from zero over $\mathcal{K}_\nu$.
For the continuous-time LQR, the minimum eigenvalue of $X$ may approach zero over the entire domain $\mathcal{K}$ even if $(A, B)$ is controllable and $W \succ 0$. Thus, we may not have global gradient dominance over $\mathcal{K}$. We provide another example from \cite[Example 3.1]{fatkhullin2021optimizing} to illustrate this point.     

\begin{example}[No gradient dominance globally] \label{example:PL-globally-no}
    Consider the one-dimensional single integrator 
    with $A=0,\,Q=R=W=2B=1$. We consider a state feedback gain $K=-k\in\mathbb{R}$. The stabilizing region is $\mathcal{K}=(0, \infty)$. 
    It is straightforward to compute from the Lyapunov equation  \cref{eq:Lyapunov-H2} that $X = 1/k > 0, \forall k > 0$. However, $X$ is not bounded away from zero on the entire domain since we have $\lim_{k \to \infty} X = 0$. Thus, gradient dominance fails to hold globally.  
    Indeed, it is not difficult to see 
 that  
   $
        J_\LQR(K) = k+1/k,
        \nabla J_\LQR(K)
        = 1-1/{k^2}, 
        J_\LQR(K^\star) = 2.
   $ 
    It is now clear that 
    $$
    \lim_{k \to \infty} \frac{J_\LQR(K)-J_\LQR(K^\star)}{\|\nabla J_\LQR(K)\|^2} =  \lim_{k \to \infty} \frac{k + 1/k - 2}{(1 -1/k^2)^2} = \infty. 
    $$
    Thus, the gradient dominance \cref{eq:PL-LQR} does not hold globally over $\mathcal{K}$. Note that $J_\LQR(K)$ becomes more like a linear function as $k$ increases, which cannot be lower bounded by a quadratic function. \hfill $\square$
\end{example}

From the discussions in \cref{example:PL-globally-no}, we have the following corollary. 
\begin{corollary}
    For continuous-time LQR, there may exist no positive constant $\mu$ such that \Cref{eq:PL-LQR} holds globally for $K \in \mathcal{K}$ even when $(A,B)$ is controllable, $(Q^{1/2}, A)$ is observable, $W \succ 0$, and $R \succ 0$.
\end{corollary}

\begin{remark}[Global gradient dominance in discrete-time LQR
\cite{fazel2018global,bu2019lqr}]\label{remark:discrete-time-gradient-dominance}
The discrete-time LQR satisfies gradient dominance \textit{globally} when $W \succ 0$. 
This is because the discrete-time Lyapunov equation directly implies a lower bound for its solution. 
For a discrete-time system
$x_{k+1}=Ax_k+Bu_k$ with $u_k=Kx_k$,  its associated Lyapunov equation is $(A+BK)X(A+BK)^\tr - X + W=0$. 
For the solution $X$, it always holds that
\begin{equation*}
    X = \sum_{k=0}^\infty 
    (A+BK)^k W \left((A+BK)^\tr\right)^k 
    \succeq W \qquad \Rightarrow \qquad \lambda_\mathrm{min}(X)\geq \lambda_\mathrm{min}(W).
\end{equation*}
This may not be true 
in the continuous-time case.
Indeed, we have 
$W\succeq X$ in
\cref{example:W-singular} and
$X= 1/k$ in \cref{example:PL-globally-no}.
We thus need the compactness of $\mathcal{K}_\nu$
to ensure a positive lower bound of $\lambda_\mathrm{min}(X)$ for the continuous-time LQR. This is also related to an implicit sampling time in discrete-time systems, and the sampling time is always strictly positive. 
Recently, this discrepancy has been discussed in \cite{cui2024small}. A related difference for the coercivity of $\mathcal{H}_\infty$ cost in continuous-time and discrete-time is discussed in \cite[Fact~4.1]{zheng2024benign}.
\hfill $\square$
\end{remark}

The two conditions in \Cref{assumption:compactness} are thus important to establish gradient dominance for LQR. We next provide sufficient conditions to ensure \Cref{assumption:compactness}; see \cref{appedix:proof-proposition-sufficient-conditions-compactness} for a proof.

\begin{proposition} \label{proposition:sufficient-conditions}
    Suppose \cref{assumption:controllability} holds, and one of the following conditions is true.
    \begin{enumerate}
        \item[a)] The covariance matrix $W$ is positive definite;
        \item[b)] 
        The matrix $Q$ is positive definite, $(A, B)$ is controllable, and $\operatorname{Im} B \subseteq \operatorname{Im} W$;
        \item[c)] $(A, B)$ is controllable and
        $\operatorname{Im} B \subseteq \operatorname{Im} W$. For any sequence $\{K_j\}\subset\mathcal{K}$ such that
        $\lim_{j\to\infty}\|K_j\|_F=\infty$, we have
        ${\lim\inf}_{j\to\infty} \|P_j^{1/2}W^{1/2}\|_F/\|K_j\|_F^{1/2}>0,$ where $P_j$ is the solution to the Lyapunov equation    \begin{equation}\label{eq:dual-Lyapunov-eq}
            (A+BK_j)^\tr P_j+P_j(A+BK_j)+K_j^\tr RK_j+Q =0.
        \end{equation}
    \end{enumerate}
    Then, we have that 1) the sublevel set $\mathcal{K}_\nu := \{K \in \mathcal{K} \mid J_\LQR(K) \leq \nu\}$ with any finite $\nu \geq J_\LQR(K^\star)$ is compact and 2) $X$ from  \cref{eq:Lyapunov-H2} is positive definite if and only if $K \in \mathcal{K}$. 
\end{proposition} 

As discussed in \Cref{remark:PL-LQR}, all existing results in \cite{mohammadi2019global,fatkhullin2021optimizing,bu2020policy,zhang2025convergence} require $W \succ 0$, which is the same as the first condition in \cref{proposition:sufficient-conditions}. The recent results \cite{mohammadi2019global,fatkhullin2021optimizing,zhang2025convergence} further require $Q$ to be strictly positive definite. 
In the second condition of \cref{proposition:sufficient-conditions}, the part of $\operatorname{Im} B \subseteq \operatorname{Im} W$ and $(A, B)$ being controllable is to guarantee the positive definiteness of $X$, and
the positive definiteness of $Q$ is to ensure the compactness of $\mathcal{K}_\nu$. 
The third condition (c) further relaxes the requirement of $Q\succ0$.
We here provide two LQR instances with $W \succeq 0$ which still satisfy the gradient dominance property.

\begin{figure}
    \centering
    \setlength{\abovecaptionskip}{0pt}
  \begin{subfigure}{0.48\linewidth}
           \centering
\includegraphics[width=0.67\linewidth]{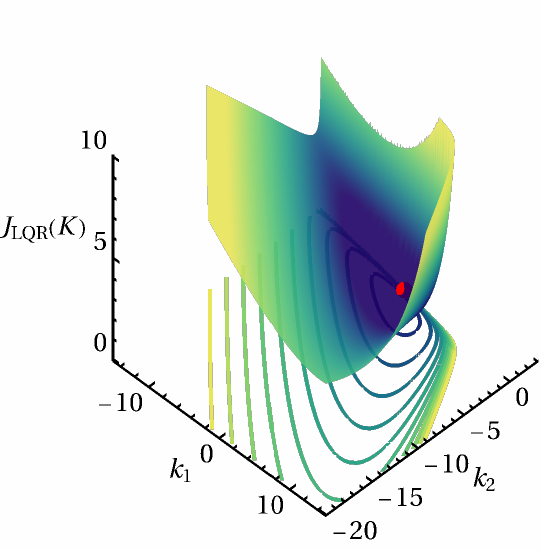}
            \caption{
            }
            \label{fig:PL-W>=0}
  \end{subfigure}%
  \begin{subfigure}{0.48\linewidth}
          \centering
    \includegraphics[width=0.7\linewidth]{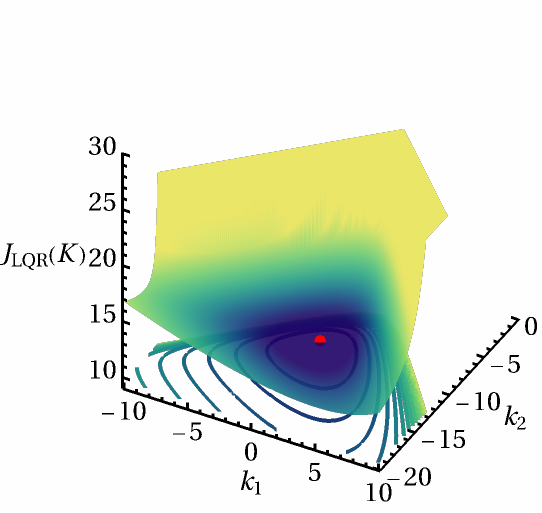}
       \caption{}
       \label{fig:PL-WQ>=0}
  \end{subfigure}
    \caption{``Well-behaved'' nonconvex LQR landscapes for \cref{example:PL-W>=0} with $a=0.1$ and \cref{example:PL-W-Q>=0} that
    use positive semidefinite $W$.
    By \cref{theorem:J_LQR-gradient_dominance-general}, the optimal feedback gain is unique, and gradient dominance is satisfied over $\mathcal{K}_\nu$.} \vspace{-4mm}
    \label{fig:PL_prop4-1}
\end{figure}

\begin{example}[Gradient dominance with $W\succeq0$ and $Q \succ 0$]
\label{example:PL-W>=0}
Consider an LQR instance with 
\begin{equation*}
    A =\begin{bmatrix}
        -10 &a\\
        a& -1
    \end{bmatrix},\quad
    B = \begin{bmatrix}
        1\\ 1
    \end{bmatrix},
    \quad
    W = BB^\tr, \quad
    Q=I,\quad R=1.
\end{equation*}
Notice that $B$, $Q$, and $R$ are the same as \cref{example:W-singular}. One main difference is that $W = BB^\tr$  and $(A, B)$ is controllable in this LQR instance. 
We thus verify that this LQR instance satisfies the second condition in \cref{proposition:sufficient-conditions}.
The LQR cost $J_\LQR(K)$ for $K=[k_1,k_2]$ is shown in \cref{fig:PL-W>=0}.
As expected from \Cref{theorem:J_LQR-gradient_dominance-general}, the optimal solution is unique, and gradient dominance is satisfied in this case. This LQR instance with $W\succeq0$ goes beyond the existing results on gradient dominance. \hfill $\square$
\end{example}

\begin{example}[Gradient dominance with both $W\succeq0$ and $Q \succeq 0$] \label{example:PL-W-Q>=0}
Consider another LQR instance with problem data:
 \begin{equation*}
        A= \operatorname{diag}(-5,5),\quad
        B=\begin{bmatrix}
            1\\
            1
        \end{bmatrix},\quad
        Q=\operatorname{diag(0,1)},\quad R=1,\quad
        W=BB^\tr.
    \end{equation*}
It is easy to verify that $(A,B)$ is controllable, and $\operatorname{Im}B\subseteq \operatorname{Im}W$.
However, neither $Q$ nor $W$ is positive definite in this example, which violates conditions (a) and (b) in \cref{proposition:sufficient-conditions}.
Nonetheless,
we can verify that 
the LQR cost $J_\LQR$ in \cref{fig:PL-WQ>=0} indeed exhibits gradient dominance. \hfill $\square$
\end{example}

\begin{remark}[Strong duality versus gradient dominance in LQR]
    Both strong duality and gradient dominance reflect hidden convexity in LQR. However, strong duality in \cref{theorem:strong-duality-1} only requires the usual assumption of stabilizability and detectability, while gradient dominance in \cref{theorem:J_LQR-gradient_dominance-general} requires slightly stronger assumptions.    
    This difference is unsurprising as the inequality in gradient dominance \cref{eq:PL-LQR} applies to all the points in a sublevel set and strong duality essentially only relates to the optimal values between the primal and dual problems. Another way to see the difference is that strong duality is a property in \textit{usual convex optimization}, while gradient dominance is more like a property in \textit{strongly convex optimization} (indeed, we require at least {quadratic growth}).  \hfill $\square$
\end{remark}

\subsection{Extended Convex Lifting ($\ECL$)
and gradient dominance} \label{subsection:ECL-GD}

In this section, we introduce the $\ECL$ framework \cite{zheng2024benign} 
and present a general result of gradient dominance (\cref{theorem:gradient-dominance-ECL}) for differentiable functions. The LQR result in \cref{theorem:J_LQR-gradient_dominance-general} will be a special~case.  

\subsubsection{Gradient dominance in $\ECL$}

The idea of \texttt{ECL} is intuitive and closely related to convex re-parameterization. Let $f$ be a potentially
nonconvex function, and suppose there exists a $C^1$ diffeomorphism $y = \Phi(x)$ such that $h(y) := f(\Phi^{-1}(y))$ is convex. In other words, $h$ is a convex parameterization of $f$ under the new coordinates $y = \Phi(x)$. Then, it is clear that 1) minimizing $f (x)$ is equivalent to minimizing $h(y)$, which is a convex problem; 2) all stationary points of $f$ are globally optimal. If 
both $f$ and $h$ are differentiable and
$h$ further
satisfies gradient dominance
     $ \mu \left(h(y) - h(y^\star) \right)
     \leq\frac{1}{2} \|
     \nabla h(y)\|^2,$ 
we have 
\begin{equation} \label{eq:PL-convex-reparameterization}
\begin{aligned}
f(x) - f(x^\star) &= f(\Phi^{-1}(y)) - f(\Phi^{-1}(y^\star))\\
&= h(y) - h(y^\star) \\
&\leq \frac{1}{2\mu} \|\nabla h(y)\|^2 \\
& = \frac{1}{2\mu} \|\nabla \Phi^{-1}(y)^\tr  \nabla f(\Phi^{-1}(y))\|^2 
 \leq \frac{1}{2\mu} \|\nabla \Phi^{-1}(y)\|_2^2 \|\nabla f(x)\|^2,
\end{aligned}
\end{equation} 
where the third and fourth lines follow from the gradient dominance of $h$ and the chain rule, respectively.  
If the quantity $\|\nabla \Phi^{-1}(y)\|_2$ is bounded, then we have established the gradient dominance of $f(x)$ despite its nonconvexity. In particular, the parameter $\mu$ in \cref{eq:PL-convex-reparameterization} plays the same role of $\mu_\mathrm{qg}$ in \cref{eq:PL-constant}, and the bound on $\|\nabla \Phi^{-1}(y)\|_2$ is the same as $1/c_\LQR$ in \cref{eq:PL-constant}.

The analysis strategy in \cref{eq:PL-convex-reparameterization} has good intuition and simplicity, which was recently explored in \cite{fatkhullin2025stochastic}.
 However, we cannot directly apply it {to general control problems} for two main reasons: 1) the bijection $\Phi$ may not exist due to the natural appearance of Lyapunov variables; 2) the gradient may not be a good tool to connect the original function and its convex re-parameterization when their domains are not \textit{full dimensional} (e.g., constraints are involved). The \texttt{ECL} \cite{zheng2024benign} is a general framework to resolve the first difficulty, and we will refine \texttt{ECL} via the notion of Cauchy direction \cite{lemarechal2012cauchy} to address the second difficulty. 

For convenience, we introduce a simplified \texttt{ECL} that is sufficient for the LQR; see \cite{zheng2024benign} for a more general setup. Recall that the epigraph of a function $f: \mathcal{D} \to \mathbb{R}$ is defined as follows 
$$
\operatorname{epi}_{\geq}(f) : = \{(x,\gamma) \mid x \in \mathcal{D}, f(x) \leq \gamma\}.
$$

\begin{definition}[Extended Convex Lifting]
Let $f: \mathcal{D} \rightarrow \mathbb{R}$ be a continuous function, where $\mathcal{D} \subseteq \mathbb{R}^d$ is an open set. We say that the tuple $\left(\mathcal{L}_{\mathrm{lft}}, \mathcal{F}_{\mathrm{cvx}}, \Phi\right)$ is an $\ECL$ of $f$, if the following conditions hold:
\begin{enumerate}
    \item  $\mathcal{L}_{\mathrm{lft}} \subseteq \mathbb{R}^d \times \mathbb{R} \times \mathbb{R}^{d_{\xi}}$ is a lifted set with an extra variable $\xi \in \mathbb{R}^{d_{\xi}}$, such that the canonical projection of $\mathcal{L}_{\mathrm{lft}}$ onto the first $d+1$ coordinates 
    satisfies $\operatorname{epi}_{\geq}(f) = \pi_{x, \gamma}\left(\mathcal{L}_{\mathrm{lft}}\right)$.\footnote{
    Let $\mathcal{S} \! \subseteq \! \mathbb{R}^{n_1 \! +\! n_2}$.
    The canonical projection onto the first $n_1$ coordinates is 
$\pi_x(\mathcal{S})
 \!=\!\{x \in \mathbb{R}^{n_1}|\exists y \!\in\! \mathbb{R}^{n_2} \text{ s.t. }(x,y)\!\in\!\mathcal{S}\}.$} 
    \item $\mathcal{F}_{\mathrm{cvx}} \subseteq \mathbb{R} \times \mathbb{R}^{d_e}$ is a convex set and $\Phi$ is a $C^2$ diffeomorphism from $\mathcal{L}_{\mathrm{lft}}$ to $\mathcal{F}_{\mathrm{cvx}}$.\footnote{Precisely, this requires that $\Phi$ can be extended to a $C^2$ function $\tilde{\Phi}$, defined on some open subset $U \subseteq \mathbb{R}^d \times \mathbb{R} \times \mathbb{R}^{d_{\xi}}$ with $U \supseteq \mathcal{L}_{\mathrm{lft }}$, such that $\tilde{\Phi}$ is a bijection from $U$ to $\tilde{\Phi}(U)$ and its inverse $\tilde{\Phi}^{-1}$ is also $C^2$.
    }
    \item For any $(x, \gamma, \xi) \in \mathcal{L}_{\mathrm{lft}}$, we have $\Phi(x, \gamma, \xi)=\left(\gamma, \zeta\right)$ for some $\zeta \in \mathbb{R}^{d_e}$ and $\left(\gamma, \zeta\right) \in \mathcal{F}_{\mathrm{cvx}}$. In other words, the mapping $\Phi$ directly outputs $\gamma$ in the first component.
\end{enumerate}
\end{definition}
Compared to a direct convex parameterization $y = \Phi(x)$, the \texttt{ECL} deals with epi-graphs and allows for an extra lifting variable $\xi$ that corresponds to Lyapunov variables. We remark that the requirement $\operatorname{epi}_{\geq}(f) = \pi_{x, \gamma}\left(\mathcal{L}_{\mathrm{lft}}\right)$ is sufficient for state-feedback control such as LQR, but may be too restrictive for other control applications; see \cite{zheng2024benign} for more details. 
For a function with an $\ECL$, we immediately obtain the following guarantee for the global optimality of stationary points.
\begin{proposition}{\cite[Theorem 3.2]{zheng2024benign}}\label{proposition:global-optimality-ECL}
    Let $f: \mathcal{D} \rightarrow \mathbb{R}$ be a differentiable function, equipped with an $\ECL$   $\left(\mathcal{L}_{\mathrm{lft}}, \mathcal{F}_{\mathrm{cvx}}, \Phi\right)$.
    Then, any stationary point $x^\star \in \mathcal{D}$ (i.e., $\nabla f\left(x^\star \right)=0$) is a global minimizer over $\mathcal{D}$. 
\end{proposition}

The original $\ECL$ only requires $\mathcal{F}_{\mathrm{cvx}}$ to be convex and does not impose further geometric properties. It can thus only guarantee the global optimality of stationary points in \cref{proposition:global-optimality-ECL}. If $\mathcal{F}_{\mathrm{cvx}}$ satisfies further strong convexity-like properties, then the original nonconvex function $f$ would satisfy gradient dominance. To quantify this, we assume the set $\mathcal{F}_{\mathrm{cvx}}$ is closed and  
define a convex function 
\begin{equation} \label{eq:ECL-convex-function-general}
f_{\mathrm{cvx}}(\zeta)  
=\min_{(\gamma,\zeta)\in\mathcal{F}_\cvx}\ \gamma 
\end{equation}
with $\mathrm{dom}(f_{\mathrm{cvx}})=\pi_\zeta(\mathcal{F}_\cvx)$. 
We make the following regularity assumption of $\mathcal{F}_\mathrm{cvx}$ and $f_\mathrm{cvx}$.
\begin{assumption} \label{assumption:convex-set}
    The following conditions hold.
    \begin{enumerate}
        \item The convex set $\mathcal{F}_\mathrm{cvx}$ satisfies an epi-graph property, i.e., 
$$(\gamma+\epsilon,\zeta)\in\mathcal{F}_\cvx
, \qquad \forall (\gamma,\zeta)\in\mathcal{F}_\cvx, \epsilon>0,$$ and 
for any $(x,\gamma,\xi)\in\mathcal{L}_\mathrm{lft}$,
we have $f_\mathrm{cvx}(\zeta)\geq f(x)$, where $\zeta=(\pi_\zeta\circ \Phi)(x,\gamma,\xi)$.
\item 
The minimum in the right-hand-side of \cref{eq:ECL-convex-function-general} is achieved for any $\zeta \in \pi_\zeta(\mathcal{F}_\cvx)$.
\item The sublevel set  $\mathcal{X}_\nu := \{\zeta \in \mathrm{dom}(f_{\mathrm{cvx}}) \mid  f_\cvx(\zeta) \leq \nu\}$ is compact and $f_\cvx$ satisfies $\mu_\mathrm{qg}$-quadratic growth 
over $\mathcal{X}_\nu$.
    \end{enumerate}
\end{assumption}

We are now ready to state the main gradient dominance result in the \texttt{ECL} framework.
\begin{theorem}\label{theorem:gradient-dominance-ECL}
 Let $f: \mathcal{D} \rightarrow \mathbb{R}$ be a continuously differentiable function, equipped with an $\ECL$   $\left(\mathcal{L}_{\mathrm{lft}}, \mathcal{F}_{\mathrm{cvx}}, \Phi\right)$ that satisfies \cref{assumption:convex-set}. 
Let $\nu > f^* = \min_{x
\in\mathcal{D}} f(x)$. 
The following statements are~true.
\begin{enumerate}
    \item The sublevel set of $f$: $\mathcal{D}_\nu = \{x \in \mathcal{D} \mid f(x) \leq \nu\}$ is compact, and 
     there exists 
    a positive constant $\mu\in(0,\mu_\mathrm{qg}
 c_\Phi^2]$ with some $c_\Phi>0$ such that
    \begin{equation}\label{eq:PL-fx-ECL}
    \mu\times \left( f(x) -f^*\right)
    \leq 
    \frac{1}{2}
    \|\nabla f(x)\|^2, \qquad \forall x \in \mathcal{D}_\nu. 
\end{equation}
    \item
    If the diffeomorphism $\Phi$ is further of the form
$\Phi(x,\gamma,\xi) = (\gamma,\Psi(x,\xi))$ 
with a diffeomorphism  $\Psi$
(i.e., the second output of $\Phi$ does not depend on $\gamma$),
    \cref{eq:PL-fx-ECL} holds, and
    the constant $c_\Phi$ can be chosen as 
\begin{equation}\label{eq:C_Phi}
    c_\Phi =
    \frac{1}{2}
    \min_{\zeta
    \in \mathcal{X}_\nu 
    }
    \left\| \nabla\Psi^{-1}(\zeta) \right\|_2^{-1} > 0. 
\end{equation}
The constant $c_\Phi$ is strictly positive due to the compactness of $\mathcal{X}_\nu$.
\end{enumerate}
\end{theorem}

We remark that the constant $\mu=\mu_\mathrm{qg}
 c_\Phi^2$ in \Cref{theorem:gradient-dominance-ECL} resembles the structure of the constant in \cref{eq:PL-convex-reparameterization}, where $c_\Phi^2$ plays the same role as $1/\|\nabla \Phi^{-1}(y)\|_2^2$.  
Notice that we do not require strong convexity of $f_\cvx$, which is stronger than quadratic growth \cite{liao2024error}. Certain convex reformulations of $J_\LQR$ may only satisfy quadratic growth but not strong convexity.  
The proof of \Cref{theorem:gradient-dominance-ECL} relies on the notion of Cauchy directions, which provide feasible descent directions for the lifted convex representation $f_{\mathrm{cvx}}$ on its constrained domain. The proof details are given in \Cref{subsection:proof-LQR-GD}.

\subsubsection{Proof of \Cref{theorem:gradient-dominance-ECL} via Cauchy directions
}\label{subsection:proof-LQR-GD}
  
First, we introduce the definition of Cauchy directions.
\begin{definition}
Let $h:\mathbb{R}^{d}\to\mathbb{R}\cup\{\infty\}$ be a proper function. We say $g_\mathrm{c}\in\mathbb{R}^{d}$ is a \textit{Cauchy direction} of $h$ at $x\in\mathrm{dom}(h)$ if there exists a constant $\tau_0>0$ such that 
$x-\tau g_\mathrm{c} \in \mathrm{dom}(h), \forall \tau\in[0,\tau_0]$ and
\begin{equation*}
    \lim_{\tau \downarrow 0}
    \frac{h(x-\tau g_\mathrm{c}) - h(x)}{\tau} \leq - \|g_\mathrm{c}\|^2.
\end{equation*}
\end{definition}

\begin{figure}
    \centering
  \begin{subfigure}{0.42\linewidth}
           \centering
        \includegraphics[width=0.7\linewidth]{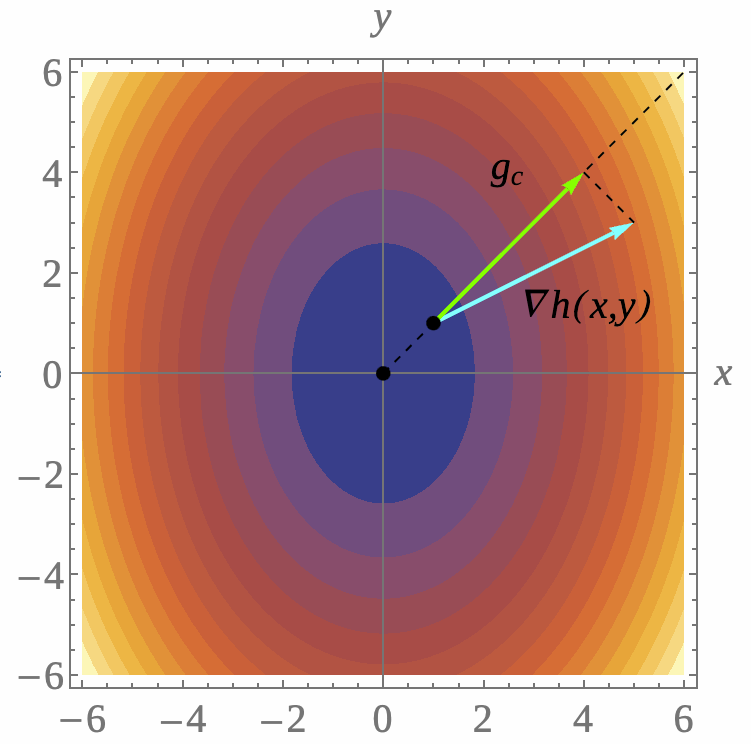}
            \caption{$h_1(x,y)=2x^2+y^2$}
            \label{fig:Cauchy-direction-2d}
  \end{subfigure}%
\begin{subfigure}{0.42\linewidth}
  \centering
  \resizebox{0.8\columnwidth}{!}{%
    \begin{tikzpicture}
      \begin{axis}[
        width=8cm,
        height=6cm,
        xmin=-0.12, xmax=2.25,
        ymin=-0.18, ymax=1.4,
        xlabel={$x$},
        ylabel={$y$},
        axis lines=center,
        thick,
        every tick/.style={thick},
        xtick={0,1,2},
        ytick={0,0.5,1},
        tick label style={font=\small},
        clip=false
      ]
        \addplot[black, thick, domain=0:1.5, samples=160]
          ({x^2},{(x-1)^2});

        \addplot[blue, dashed, thick]
          coordinates {(0,1) (1,0)};

        \draw[->, blue, very thick]
          (axis cs:0,1) -- (axis cs:0.65,0.35);

        \node[blue, anchor=west] at (axis cs:0.72,0.67)
          {$(-g_{\mathrm c},-|g_{\mathrm c}|^2)$};
        \node[blue, anchor=west] at (axis cs:0.9,0.50)
          {$=(1,-1)$};

        \addplot[only marks, mark=*, mark size=2pt, black]
          coordinates {(0,1) (1,0)};

        \node[anchor=south west] at (axis cs:0.03,1.03)
          {$(0,1)$};
        \node[anchor=south west] at (axis cs:1.03,0.04)
          {$(1,0)$};

        \node[red, anchor=west] at (axis cs:0.70,1.22)
          {$\partial h_2(0)=\emptyset$};
      \end{axis}
    \end{tikzpicture}%
  }
  \caption{$h_2(x)=(\sqrt{x}-1)^2$, $x\ge0$.}
  \label{fig:Cauchy-direction-1d}
\end{subfigure}
\caption{Illustration of Cauchy directions.
  (a) $h_1(x,y)=2x^2+y^2$.
  The blue and green arrows denote
  $\nabla h_1(1,1)=[4,2]^\top$ and the Cauchy direction
  $g_{\mathrm c}$ in
  \cref{eq:Cauchy-direction-differentiable-h}, respectively.
  (b) $h_2(x)=(\sqrt{x}-1)^2$ for $x\ge0$,
  and $h_2(x)=+\infty$ otherwise.
  At $x_0=0$, the finite Cauchy direction $g_{\mathrm c}=-1$
  exists despite $\partial h_2(0)=\emptyset$.
  The blue arrow depicts the direction
  $(-g_{\mathrm c},-|g_{\mathrm c}|^2)=(1,-1)$.}
\label{fig:Cauchy-direction}
\end{figure}

A Cauchy direction generalizes the gradient for smooth functions and subgradient for convex functions. Indeed, if $h$ is continuously differentiable at $x$, we have 
$
\lim_{\tau \downarrow 0}
    \frac{h(x-\tau g_\mathrm{c}) - h(x)}{\tau} = -\langle \nabla h(x), g_\mathrm{c} \rangle.
$
Then, the gradient $g_\mathrm{c} = \nabla h(x)$ is a Cauchy direction. If $h$ is convex, we can often use a subgradient $g \in \mathbb{R}^{d}$ as a Cauchy direction. 
Moreover, for convex functions, we can always construct a Cauchy direction that bounds the suboptimality. 
We have the next lemma, adapted from~\cite[Lemma H.1]{umenberger2022globally}.

\begin{lemma}\label{proposition:Umenberger-lemma_H.1}
Let $h: \mathbb{R}^{d} \to \mathbb{R}\cup\{\infty\}$ be a proper convex function. Suppose $x^\star\in\argmin_x \,h(x)$ is attained.
Then, $\forall x\in\mathrm{dom}(h)$ and $x \neq x^\star$, there exists a Cauchy direction $g_\mathrm{c}$ at $x$ which satisfies   
\begin{subequations}
  \begin{equation} \label{eq:Cauchy-direction}
        \|g_{\mathrm{c}}\| \geq \frac{h(x)-h(x^\star)}{\|x - x^\star\|}.
    \end{equation}
Moreover, we can choose this Cauchy direction as 
    \begin{equation}\label{eq:Cauchy-direction-construction}
        g_{\mathrm{c}} = \frac{\eta}{\|x-x^\star\|^2} \times (x-x^\star)
    \end{equation}
    with $\eta = \max\{1,h(x)-h(x^\star)\}$ if $\alpha := \lim_{t\downarrow 0}
            {(h\left(x-t(x-x^\star)\right) - h(x))}/{t} = - \infty$, otherwise $\eta = |\alpha|$. 
\end{subequations}
\end{lemma}

We note that if $h$ is differentiable at $x \in \mathrm{dom}(h)$, the Cauchy direction in \cref{eq:Cauchy-direction-construction} can be chosen as
\begin{equation}\label{eq:Cauchy-direction-differentiable-h}
    g_\mathrm{c}
    =  
        \left \langle
        \nabla h(x),\frac{x-x^\star}{\|x-x^\star\|}
        \right\rangle \times 
        \frac{x-x^\star}{\|x-x^\star\|}.
\end{equation}
The construction in \cref{eq:Cauchy-direction-construction} goes beyond subgradients and is always finite.
 Note also that the Cauchy directions in \cref{eq:Cauchy-direction-construction} can be defined with any $x^\star\in\arg\min_{x\in\operatorname{dom}(h)}h(x)$;
see \cite[Appendix H]{umenberger2022globally} for further detail. \Cref{fig:Cauchy-direction} illustrates the construction of Cauchy directions.

\vspace{6pt}

\noindent \textbf{Proof of \Cref{theorem:gradient-dominance-ECL}.}
\label{subsubsection:proof-GD-ECL} 
Our proof idea is to employ Cauchy directions to connect the original nonconvex function $f$ and its lifted convex parameterization $f_\mathrm{cvx}$. It turns out that Cauchy directions can be preserved in the \texttt{ECL} procedure. 
In particular,  the quadratic growth in $f_\cvx$ is inherited by the original nonconvex function $f$ in the form of gradient dominance.

Given any point $\zeta \in \mathcal{X}_\nu$,
let $\hat{\zeta}^\star$ denote 
\begin{equation}\label{eq:zeta-hat}
    \hat{\zeta}^\star
    = {\arg\min}_{\hat \zeta \in \mathcal{S}} \ \|\zeta -\hat\zeta\|
\end{equation}
with $\mathcal{S} :=  \arg\min_{\zeta \in \pi_\zeta(\mathcal{F}_\cvx)} \ f_\cvx(\zeta)$.
Note that 
the set $\mathcal{S}$ is closed and convex
from \cite[Theorems 4.6 and 7.6]{rockafellar1970convex}
because $f_\cvx$ is proper, closed, and convex from the definition and $f^*>-\infty$.
Then, the construction of the Cauchy direction $g_{\mathrm{c}}$ in \cref{eq:Cauchy-direction-construction} satisfies 
\begin{equation}\label{eq:proof_Cauchy_weakPL}
    \|g_{\mathrm{c}}\|\|\zeta - \hat\zeta^\star\| \geq f_\mathrm{cvx}(\zeta) - f_\mathrm{cvx}(\hat{\zeta}^\star)
    =f_\mathrm{cvx}(\zeta) - f_\mathrm{cvx}(\zeta^\star)
    \geq \frac{\mu_\mathrm{qg}}{2}\|\zeta - \hat\zeta^\star\|^2,
\end{equation}
where 
$\zeta^\star$ represents any optimal solution in $\mathcal{S}$.
Here,
the first inequality is due to \cref{eq:Cauchy-direction} and the second inequality is due to the assumption of quadratic growth. Thus, this Cauchy direction satisfies 
$  \|\zeta- \hat\zeta^\star\| \leq \frac{2}{\mu_\mathrm{qg}} \|g_{\mathrm{c}}\|$.
Substituting this into the first inequality \cref{eq:proof_Cauchy_weakPL} yields
\begin{equation} \label{eq:strong-convexity-PL}
     f_\mathrm{cvx}(\zeta) - f_\mathrm{cvx}(\zeta^\star)\leq \frac{2}{\mu_\mathrm{qg}} \|g_{\mathrm{c}}\|^2.
\end{equation}
This inequality also confirms that any strongly convex function satisfies gradient dominance.

Our next key result connects the Cauchy direction $g_{\mathrm{c}}$ in \cref{eq:Cauchy-direction-construction} with the gradient  $\nabla f$.
\begin{lemma}\label{lemma:gradient-Cauchy_direction}
Consider the same setup as \cref{theorem:gradient-dominance-ECL}.
Let $x\in\mathcal{D}_\nu$ with $x\neq x^\star$, and let $\xi$ satisfy
$(x,f(x),\xi)\in\mathcal{L}_\mathrm{lft}.$
Then,
the following statements hold.  
\begin{enumerate}
\item The Cauchy direction $g_{\mathrm{c}}$ in \cref{eq:Cauchy-direction-construction} of $f_\mathrm{cvx}$ at $\zeta$,
where $\zeta=(\pi_{\zeta}\circ\Phi)(x,f(x),\xi)$,
satisfies
\begin{equation} \label{eq:gradient-Cauchy-direction}
    2c_\Phi \|g_{\mathrm{c}}\| \leq \|\nabla f(x)\|, 
\end{equation}
with some $c_\Phi>0$.  
\item If the diffeomorphism $\Phi$ is further of the form
$\Phi(x,\gamma,\xi) = (\gamma,\Psi(x,\xi))$ (i.e., the second output of $\Phi$ does not depend on $\gamma$),
    the constant $c_\Phi$ can be chosen as \cref{eq:C_Phi}.  
\end{enumerate}\end{lemma}

The proof of \cref{lemma:gradient-Cauchy_direction} utilizes the diffeomorphism in the \texttt{ECL} construction. The details are technically involved, and we postpone them to \cref{appendix:proof-lemma:gradient-Cauchy_direction}.
It is now clear that \Cref{theorem:gradient-dominance-ECL} is a direct consequence of 
the $\mu_\mathrm{qg}$-quadratic growth of $f_\cvx$ and
\Cref{lemma:gradient-Cauchy_direction}:  
\begin{align*}
    f(x)-f^* \leq 
    f_\mathrm{cvx}(\zeta )-f_\mathrm{cvx}(\zeta^\star)
     \overset{\cref{eq:strong-convexity-PL}}{\leq} \frac{2}{\mu_\mathrm{qg}}\|g_\mathrm{c}\|^2
     \overset{\cref{eq:gradient-Cauchy-direction}}{\leq}  \frac{1}{2\mu_\mathrm{qg}c_\Phi^2}\|\nabla f(x)\|^2,
\end{align*}
where the first inequality follows from $f(x)\leq f_\mathrm{cvx}(\zeta)$ and $f^*=f_\mathrm{cvx}(\zeta^\star)$. We finally remark that \cref{lemma:gradient-Cauchy_direction} itself does not require the quadratic growth of $f_\cvx$.

\subsection{Proof of \cref{theorem:J_LQR-gradient_dominance-general}
via $\ECL$ constructions}\label{subsubsection:LQR-PL-proof}

Here, we show that
the gradient dominance of the LQR cost $J_\LQR$ in
\cref{theorem:J_LQR-gradient_dominance-general} 
as a special case of
\cref{theorem:gradient-dominance-ECL}.
 We emphasize that
 under \cref{assumption:controllability,assumption:compactness}, the optimal LQR gain $K^\star$ is unique (see \cref{proposition:uniqueness-X>0}). 
It is well-known that the change of variable $KX = Y$ allows us to construct an $\ECL$ \cite[Section 4.2.1]{zheng2024benign}. Here, we provide another $\ECL$ that is more convenient for analysis.

\begin{subequations} \label{eq:ECL-LQR-constructions}
From  \cref{eq:LQR-Lyapunov-H2}, it is natural to 
introduce the following lifted set 
\begin{equation} \label{eq:ECL-LQR-constructions-a}
    \mathcal{L}_{\lft}=\left\{(K, \gamma, X) \middle|
    \substack{\displaystyle X \succ 0,(A+B K) X+X(A+B K)^\tr+W=0, \\
    \displaystyle \gamma \geq \operatorname{tr}\left(\left(Q+K^\tr R K\right) X\right)}\right\},
\end{equation}
where $X$ is a lifting Lyapunov variable, and the correspondence 
to the original definition of $\ECL$ is
\begin{equation*}
    x \equiv K,\quad \xi \equiv X,\quad
    \mathcal{D} \equiv\mathcal{K}.
\end{equation*}
This set $\mathcal{L}_{\lft}$ is nonconvex because of the coupling between $X$ and $K$. 
Using the change of variable $Y=(K-K^\star)X$ and setting $\zeta
= (Y,X)$,
we define a convex set 
\begin{equation} \label{eq:ECL-LQR-constructions-b}
    \mathcal{F}_{\cvx}=
    \left\{(\gamma, Y,X) \left\lvert\, \begin{array}{c}
X \succ 0, \,A^\star X+BY+(A^\star X+BY)^\tr +W=0 \\
\quad \text { and } \gamma \geq \operatorname{tr}\left(Q^\star X +
X^{-1}Y^\tr RY + (K^\star)^\tr 
RY + Y^\tr RK^\star 
\right)
\end{array}\right.\right\},
\end{equation}
where $A^\star := A+BK^\star$ and $Q^\star = Q+(K^\star)^\tr RK^\star $. This convex set is different from the standard result in \cite[Section 4.2.1]{zheng2024benign}, and we provide further details in \cref{appendix:ECL-construction-LQR}.
Finally, we define  
\begin{equation} \label{eq:LQR-map}
\Phi(K, \gamma, X)=
(\gamma, \Psi({K, X}))=
(\gamma, { (K-K^\star) X, X}),
\quad\forall(K, \gamma, X) \in \mathcal{L}_{\lft}. 
\end{equation}
\end{subequations}
{We can show it is a diffeomorphism between $\mathcal{F}_{\cvx}$ and $\mathcal{L}_{\lft}$}.
Then, we have the following result. 

\begin{proposition}\label{proposition:ECL-1}
   Suppose \Cref{assumption:controllability} and the first condition in \cref{assumption:compactness} hold. Then, 
    the tuple $(\mathcal{L}_{\lft},\mathcal{F}_{\cvx},\Phi)$ defined in \cref{eq:ECL-LQR-constructions} is an $\ECL$ for the LQR cost $J_\LQR$ in \cref{eq:LQR-H2}. 
\end{proposition}

The proof of \Cref{proposition:ECL-1} is straightforward from the invertibility of $X$ and \cite[Proposition 4.1]{zheng2024benign}.
We only remark that the first condition in \cref{assumption:compactness} is necessary since it guarantees 1) $X$ from \cref{eq:Lyapunov-H2} is positive definite and 2) $\operatorname{epi}_{\geq}(J_{\LQR}) = \pi_{x, \gamma}\left(\mathcal{L}_{\lft}\right)$. 
Note that this $\ECL$ is convenient for analysis but not for computation (as $K^\star$ only exists but is unknown). 
For this $\ECL$,  we always have a direct affine mapping from $Y$ to $X$ as $
    X  = - \mathcal{A}_\star^{-1}
    (BY+Y^\tr B^\tr +W ) $
with $\mathcal{A}_\star(\cdot)$ in \cref{theorem:J_LQR-gradient_dominance-general}
because $\mathcal{A}_\star(\cdot)$ is invertible from the stability of $A+BK^\star$.

Next, consider the convex function from this $\ECL$ as
\begin{equation}\label{eq:ECL-convex-function}
\begin{aligned}
    f_\mathrm{cvx}(Y,X)
    = &\min_{(\gamma,Y,X)\in\mathcal{F}_\cvx}\ \gamma \\
    = &
     \operatorname{tr}\left(Q^\star X +
X^{-1}Y^\tr RY + (K^\star)^\tr 
RY + Y^\tr RK^\star 
\right) 
\end{aligned}
\end{equation}
where its domain is $\mathrm{dom}(f_\mathrm{cvx})=\{ (Y,X) \in \mathbb{R}^{m \times n} \times \mathbb{S}^n \mid X \succ 0, \,A^\star X+BY+(A^\star X+BY)^\tr +W=0\}$. 
Note that 
$J_\LQR(K)=f_\mathrm{cvx}(Y,X)$ holds
for any $(Y,X)=\Psi(K,X)$ with $(K,X)\in\pi_{K,X}(\mathcal{L}_\mathrm{lft})$;
the first part of \cref{assumption:convex-set} is satisfied.
The sublevel set $\mathcal{X}_{\nu}$ in this case is 
$$
\mathcal{X}_{\nu} = \{(Y,X) \mid (Y,X) \in \mathrm{dom}(f_\mathrm{cvx}), \; f_\mathrm{cvx}(Y,X) \leq \nu \}. 
$$
It is not obvious that the set $\mathcal{X}_\nu$ is compact due to the strict inequality $X\succ0$;~however, its compactness follows from that of $\mathcal{K}_\nu$
in \cref{assumption:compactness} and the continuous dependence of the Lyapunov solution $X$ to \cref{eq:LQR-Lyapunov-H2} on the feedback gain $K \in \mathcal{K}_\nu$.
\begin{lemma} \label{lemma:compact-LQR-convex-set}
With \cref{assumption:controllability,assumption:compactness},
the set $\mathcal{X}_\nu
$ is compact for any $\nu\geq J_\LQR(K^\star)$.
\end{lemma}
\begin{proof}
Define an auxiliary set 
$$\mathcal{Y}_\nu=\left\{(K,X)\left| (A+B K) X+X(A+B K)^\tr+W=0, K \in \mathcal{K}_\nu\right.
\right\}$$
By \cref{assumption:compactness}, $\mathcal{K}_\nu$ is compact. Also, the Lyapunov solution $X$ can be written as an analytic function of $K \in \mathcal{K}$. Thus, we know the set $\mathcal{Y}_\nu$ is compact. 

By the construction in \cref{eq:LQR-map}, it is straightforward to verify that $\mathcal{X}_\nu = \Psi(\mathcal{Y}_\nu)$. Since $\Psi$ is continuous, we have $\mathcal{X}_\nu$ is also compact.
\end{proof}

We can show that
$f_\mathrm{cvx}$ satisfies quadratic growth over  $\mathcal{X}_\nu$.

\begin{lemma}\label{lemma:f_cvx-QG}
With \cref{assumption:controllability,assumption:compactness},
the function $f_\cvx$ satisfies $\mu_\mathrm{qg}$-quadratic growth over $\mathcal{X}_\nu$ 
with 
$\mu_\mathrm{qg}$
in \cref{eq:PL-constant}. In other words, \Cref{assumption:convex-set} holds for the \texttt{ECL} \cref{eq:ECL-LQR-constructions-a,eq:ECL-LQR-constructions-b,eq:LQR-map} of the LQR. 
\end{lemma}

 We provide the proof in \cref{appendix:subsubsec-proof-f_cvx-QG}.
Finally,  it is clear that the set $\mathcal{F}_{\cvx}$ in \cref{eq:ECL-LQR-constructions-b} satisfies the epi-graph property.  In other words, \Cref{assumption:convex-set} holds for the \texttt{ECL} \cref{eq:ECL-LQR-constructions-a,eq:ECL-LQR-constructions-b,eq:LQR-map} of the LQR.  
Therefore, the gradient dominance in \cref{theorem:J_LQR-gradient_dominance-general} directly follows from the second statement of \cref{theorem:gradient-dominance-ECL}. The constant $c_\LQR$ in
\cref{theorem:J_LQR-gradient_dominance-general} corresponding to $c_\Phi$ in \cref{theorem:gradient-dominance-ECL}
requires some further discussion; the computation is detailed in \cref{appendix:derivation-c_LQR}.


\section{Global Optimality of Static Linear Policies}\label{section:global-optimality_static-linear-policies}

It is well known that the globally optimal solution to the LQR \cref{eq:LQR} is 
achieved by \textit{linear static~policies}. This classical result has been established through various approaches, such as variational analysis \cite[Chapter 3.4 and Theorem 6.1]{liberzon2011calculus}, and the completion-of-squares technique \cite[Theorem 14.2]{zhou1996robust}.
Here, we present an alternative proof from a primal-dual perspective. 
The class of linear static policies remains optimal among all possible (nonlinear) policies for LQR, and
its proof boils down to an extension of \cref{theorem:strong-duality-1}.
While the primal-dual perspective is not new and it appears in several prior works~\cite{balakrishnan2003semidefinite,bamieh2024linear}, we provide extra details and additional insights.

\subsection{Gramian representation}
Let $\mathcal{L}_2^k[0,\infty)$ be the set of square-integrable (bounded energy) signals of dimension $k$, i.e.,
\begin{equation*}
    \mathcal{L}_2^m[0,\infty):=
    \left\{
    u:[0,+\infty)\to \mathbb{R}^{m}
    \middle|
    \|u\|_2:= \left(\int_0^\infty u(t)^\tr u(t) dt\right)^{1/2} <\infty
    \right\}.
\end{equation*}
We consider the following deterministic LQR formulation 
\begin{equation}\label{eq:LQR-w/o-state-fb}
    \begin{aligned}
J^* = \min_{u\in\mathcal{L}_2^m[0,\infty)}&\quad 
\int_0^\infty \left(x(t)^\tr Qx(t) + u(t)^\tr Ru(t)\right)dt   \\
\text{subject to}&\quad\dot{x}(t)=Ax(t)+Bu(t), \quad x(0)=x_0,
\end{aligned}
\end{equation}
where $R\succ 0$ and $Q\succeq 0$ are cost weight matrices.
We make the standard  \Cref{assumption:controllability}.

With a close examination, we see that the cost in \cref{eq:LQR-w/o-state-fb} only depends on the quadratic information of the trajectory $x(t)$ and $u(t)$. Similar to \Cref{section:strong-duality}, 
we consider a concept of \textit{Gramian} representation \cite{hotz1987covariance}.
For a signal $v\in \mathcal{L}_2^k[0,\infty)$, its Gramian representation $Z_v$ is defined by $Z_v = \int_{0}^\infty v(t)v(t)^\tr dt \succeq 0$.
Now, 
we
define
\begin{align} \label{eq:setV}
     \mathcal{V}
    = \left\{
    Z\in \mathbb{S}_+^{n+m} \middle|
    \substack{
    \displaystyle Z =\int_0^\infty 
    \begin{bmatrix}x(t)\\u(t)\end{bmatrix}
    \begin{bmatrix}x(t)\\u(t)\end{bmatrix}^\tr dt
    \text{ for some $u \in \mathcal{L}_2^m[0,\infty)$}\\
    \displaystyle 
    \text{ to } \dot{x}(t)=Ax(t)+Bu(t) \text{ with $x(0)=x_0$ such that $x\in\mathcal{L}_2^n[0,\infty)
    $}
    }
    \right\}, 
\end{align}
which contains the full quadratic information of all the stable trajectories starting from $x(0)=x_0$.
Note that
\cite[Lemma 14.1]{zhou1996robust}
guarantees
$x\in\mathcal{L}_2^n[0,\infty) \Leftrightarrow
z\in\mathcal{L}_2^{n+m}[0,\infty)
$ for $u\in\mathcal{L}_2^m[0,\infty)$
under \cref{assumption:controllability}.
Thus, using this set,
the problem \eqref{eq:LQR-w/o-state-fb} can be equivalently rewritten as
\begin{equation}\label{eq:LQR-w/o-state-fb-gram}
    J^* = \min_{Z\in\mathcal{V}} \ \ 
    \langle \text{diag}(Q,R), Z\rangle.
\end{equation}

\subsection{Bounds of $J^*$ via inner and outer approximations}

Despite being finite-dimensional, \cref{eq:LQR-w/o-state-fb-gram} is not immediately solvable in its current form due to the lack of an explicit characterization of $\mathcal{V}$. We here present inner and outer approximations for $\mathcal{V}$. As we will show in \Cref{subsection:global-optimality}, these two approximations do not sacrifice the optimality as they always contain the optimal solution to \cref{eq:LQR-w/o-state-fb-gram}. 

We first define an inner approximation of $\mathcal{V}$, which corresponds to the set of linear static policies

\begin{align} \label{eq:setV-inner}
    \mathcal{V}_\mathrm{static}
    = \left\{
    Z\in \mathbb{S}_+^{n+m} \middle|
    \substack{\displaystyle Z =\int_0^\infty 
    \begin{bmatrix}x(t)\\Kx(t)\end{bmatrix}
    \begin{bmatrix}x(t)\\Kx(t)\end{bmatrix}^\tr dt \\
    \displaystyle  \text{ for $\dot{x}(t)=(A+BK)x(t)$ with
    $x(0)=x_0$, } K \in \mathcal{K} }
    \right\}.
\end{align}
We recall that $\mathcal{K}$ represents the set of all the stabilizing gains, i.e.,
 $   \mathcal{K}:=
    \left\{
    K\in\mathbb{R}^{m\times n}
    \middle |
    A+BK \text{ is stable}
    \right\}$.
We also introduce an outer approximation of $\mathcal{V}$ below 
\begin{equation}  \label{eq:setV-outer}
\mathcal{V}_\mathrm{sdp}
    = \left\{
    Z = 
    \begin{bmatrix}Z_{11} & Z_{12}\\ Z_{12}^\tr & Z_{22}\end{bmatrix}\in \mathbb{S}_+^{n+m} \middle|
    AZ_{11}+BZ_{12}^\tr + Z_{11}A^\tr + Z_{12}B^\tr + x_0x_0^\tr = 0
    \right\}.
\end{equation}

The following lemma confirms that $\mathcal{V}_\mathrm{static}$ serves as an inner approximation of $\mathcal{V}$ while  $\mathcal{V}_\mathrm{sdp}$ provides an outer approximation of $\mathcal{V}$.
We illustrate the relationship in \cref{fig:inclusion-V}.

\begin{lemma}\label{lemma:V_static-V-V_sdp}
    Suppose $(A,B)$ is stabilizable. Consider $\mathcal{V}$ in \cref{eq:setV}, $\mathcal{V}_\mathrm{static}$ in \cref{eq:setV-inner}, and $\mathcal{V}_\mathrm{sdp}$ in \cref{eq:setV-outer}. Then, we have $
    \mathcal{V}_\mathrm{static}\subset \mathcal{V}\subseteq \mathcal{V}_\mathrm{sdp}.
    $
\end{lemma}
\begin{proof}
    It is obvious that $\mathcal{V}_\mathrm{static}\subset \mathcal{V}$ since $\mathcal{V}_\mathrm{static}$ just corresponds to the case of stabilizing static state-feedback policies. For any ${Z}\in\mathcal{V}$, by definition, we have
    \begin{equation*}
        Z_{11}= \int_0^\infty x(t)x(t)^\tr dt,
        \quad
        Z_{12}= \int_0^\infty x(t)u(t)^\tr dt,\quad
        Z_{22}= \int_0^\infty u(t)u(t)^\tr dt.
    \end{equation*}
    From the dynamics in \cref{eq:LQR-w/o-state-fb}, we have
    \begin{equation*}
    AZ_{11}+BZ_{12}^\tr + Z_{11}A^\tr + Z_{12}B^\tr 
    = \int_{0}^\infty
    \frac{d}{dt}\left(x(t)x(t)^\tr\right)dt
    = x(\infty)x(\infty)^\tr - x_0x_0^\tr = 
    - x_0x_0^\tr,
    \end{equation*}
    where we have applied the fact that $x(\infty) = \lim_{t \to \infty} x(t) = 0$. This proves 
    $Z\in\mathcal{V}_\mathrm{sdp}$.
\end{proof}

While the proof of \Cref{lemma:V_static-V-V_sdp} is almost immediate, it gives us useful inner and outer approximations for \cref{eq:LQR-w/o-state-fb-gram}: replacing $\mathcal{V}$ with $\mathcal{V}_\mathrm{static}$ gives an inner approximation while replacing  $\mathcal{V}$ with $\mathcal{V}_\mathrm{sdp}$ gives an outer approximation. It is clear that replacing $\mathcal{V}$ with $\mathcal{V}_\mathrm{sdp}$ in \cref{eq:LQR-w/o-state-fb-gram} leads to a familiar SDP formulation. We also note that replacing $\mathcal{V}$ with $\mathcal{V}_\mathrm{static}$ in \cref{eq:LQR-w/o-state-fb-gram} becomes the same as the LQR problem \cref{eq:LQR-H2} if $W = x_0x_0^\tr$. We 
 also have the following proposition, and the proof is straightforward from the property of Lyapunov equations. 

\begin{proposition}
    The set $\mathcal{V}_\mathrm{static}$ in \cref{eq:setV-inner} can be parameterized as 
    \begin{equation*}
    \mathcal{V}_\mathrm{static}
    = \left\{
    Z=\begin{bmatrix}Z_{11} & Z_{12} \\ Z_{12}^\tr & Z_{22} \end{bmatrix}\in \mathbb{S}_+^{n+m} \middle|
    \substack{
        \displaystyle 
         Z_{12} = Z_{11}K^\tr, Z_{22} = K Z_{11}K^\tr,  K \in \mathcal{K}, \\
        \displaystyle 
        AZ_{11}+BZ_{12}^\tr + Z_{11}A^\tr + Z_{12}B^\tr + x_0x_0^\tr = 0}
    \right\}.
    \end{equation*}
Furthermore, the problem $\min_{Z\in\mathcal{V}_\mathrm{static}} \,
    \langle \mathrm{diag}(Q,R), Z\rangle$ is the same as  \cref{eq:LQR-H2} if $W = x_0x_0^\tr$. 
\end{proposition}

When $x_0=0$,  
we further have 
$\mathcal{V}_\mathrm{sdp} =\mathrm{cl}\,\mathcal{V}$ (cf. \cite{you2016direct}), where $\mathrm{cl}$ denotes the closure of a set. However, for $x_0 \neq 0$, we are unclear about the relationship between  $\mathcal{V}_\mathrm{sdp}$ and $\mathrm{cl}\,\mathcal{V}$.  
\begin{figure}
\setlength{\abovecaptionskip}{0pt}
    \centering    \includegraphics[width=0.38\linewidth]{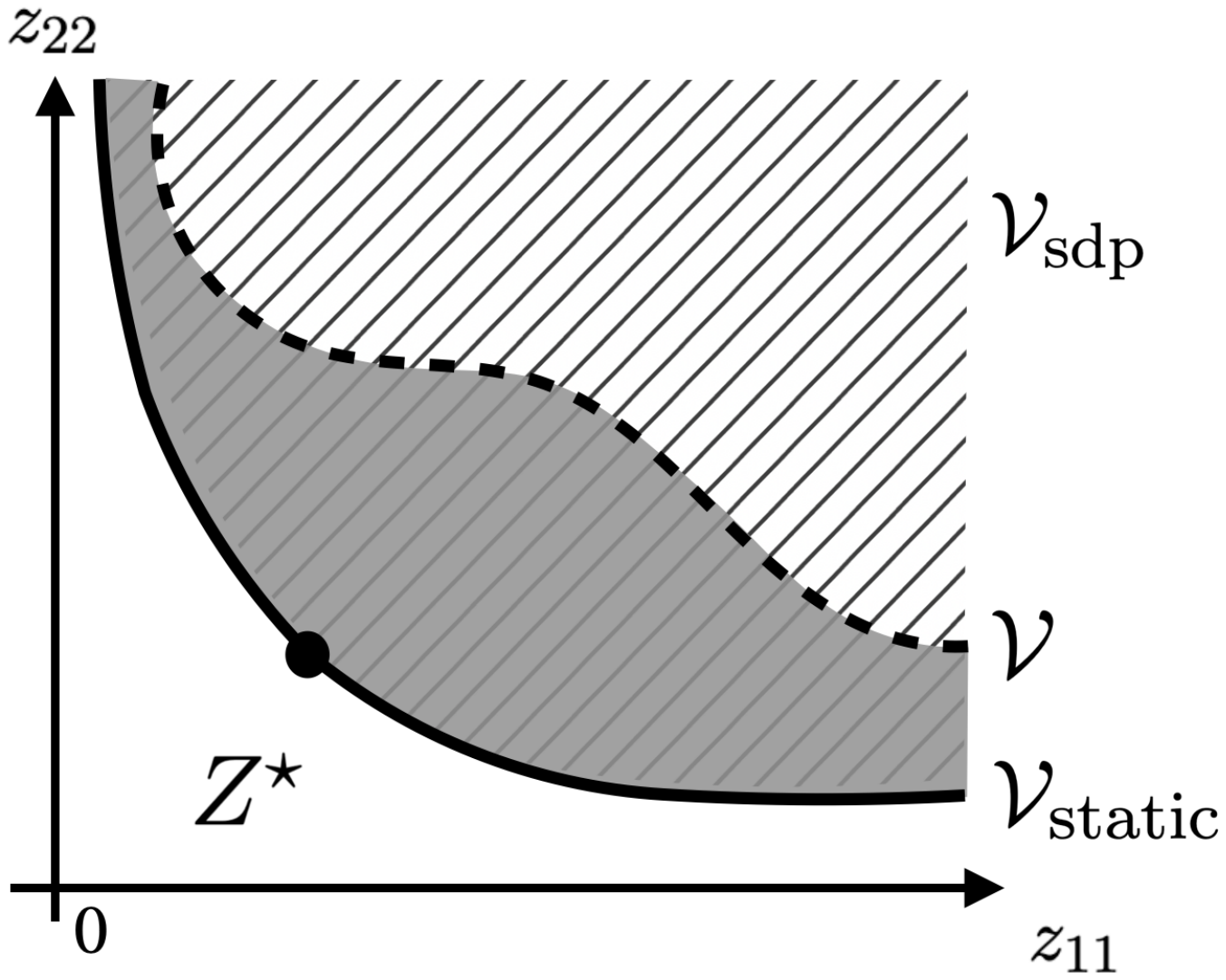}
    \caption{Illustration of the inclusion $\mathcal{V}_\mathrm{static} \subset \mathcal{V} \subseteq \mathcal{V}_\mathrm{sdp}$ in \cref{example:V-inclusion}.
    The shaded area denotes the convex set $\mathcal{V}_\mathrm{sdp}$, and the grey area sketches $\mathcal{V}$, which might be nonconvex.
    The solid black line denotes the boundary of $\mathcal{V}_\mathrm{sdp}$, which is the same as $\mathcal{V}_\mathrm{static}$. 
    From \cref{theorem:linear-optimality}, the optimal $Z^\star\in\mathcal{V}_\mathrm{sdp}$ to \cref{eq:SDP-primal-w/x_0} is attained on $\mathcal{V}_\mathrm{static}$. 
    }
    \label{fig:inclusion-V}
\end{figure}
We here present a simple instance to illustrate the inclusion in \Cref{lemma:V_static-V-V_sdp}.

\begin{example}[Inner and outer approximations]\label{example:V-inclusion}
    Consider a single integrator system $\dot{x}(t) = u(t)$ (i.e., $A = 0, B = 1$) with initial state $x_0 = 1$. 
    In this case, it is easy to see that the outer approximation \cref{eq:setV-outer} becomes  
$$
\mathcal{V}_\mathrm{sdp}
        =\left\{Z=
        \begin{bmatrix}
            z_{11} &-\frac{1}{2}\\
            -\frac{1}{2} & z_{22}
        \end{bmatrix}
        \middle|
        z_{11}z_{22}-\frac{1}{4}\geq 0,\,
        z_{11} > 0,z_{22}>0
        \right\}. 
$$
In this instance, we have 
$$
\int_{0}^\infty x(t)u(t)dt = \int_{0}^\infty x(t)\dot{x}(t)dt = \left.\frac{1}{2}x^2(t) \right|_0^\infty = -\frac{1}{2}x^2(0) = -0.5,
$$
since the input $u(t)$ stabilizes the state $x(t)$. Then, the Gramian representation in \cref{eq:setV} becomes  
    \begin{align*}
   \mathcal{V}
    = &\left\{
    Z=
        \begin{bmatrix}
            z_{11} &-\frac{1}{2}\\
            -\frac{1}{2}& z_{22}
        \end{bmatrix}
        \middle|        
    \substack{\displaystyle
    z_{11}z_{22}-\frac{1}{4}\geq 0,\,
    z_{11}=\int_0^\infty x^2(t)dt,\,
    \displaystyle
    z_{22}=\int_0^\infty u^2(t)dt\\
    \displaystyle
    \text{ for some $u \in \mathcal{L}_2^1[0,\infty)$}  \text{ such that $x\in\mathcal{L}_2^1[0,\infty)
    $}
    }
    \right\}. 
    \end{align*}
In this instance, it is clear we have $\mathcal{V} \subset \mathcal{V}_{\mathrm{sdp}}$. Moreover, the set of stabilizing gains is $\mathcal{K}=(-\infty,0)$, and we can further compute that 
    \begin{align*}
           \mathcal{V}_\mathrm{static}
    =& \left\{
    Z
    =
    \begin{bmatrix}
        z_{11} & kz_{11}\\
         kz_{11}& k^2z_{11}
    \end{bmatrix}
    \in \mathbb{S}_+^{2} \middle|
    2kz_{11}=-1,\, k <0
    \right\}
    = \left\{
    Z=
    \begin{bmatrix}
        z_{11} & -\frac{1}{2} \\
        -\frac{1}{2}& \frac{1}{4z_{11}}
    \end{bmatrix}
    \middle| z_{11}> 0
    \right\}.
    \end{align*}
    Thus, the set $\mathcal{V}_\mathrm{static}$ corresponds to the boundary $z_{11}z_{22}=\frac{1}{4}$. It is obvious that we have $\mathcal{V}_\mathrm{static} \subset \mathcal{V} \subseteq \mathcal{V}_\mathrm{sdp}$; 
    see \cref{fig:inclusion-V} for illustration.
    \hfill $\square$
\end{example}

\subsection{Global optimality via primal and dual analysis} \label{subsection:global-optimality}

By leveraging the sets $\mathcal{V}_\mathrm{sdp}$, $\mathcal{V}_\mathrm{static}$ and \cref{lemma:V_static-V-V_sdp}, we can get lower and upper bounds of $J^*$ in \cref{eq:LQR-w/o-state-fb}.
By \cref{lemma:V_static-V-V_sdp}, we can relax \eqref{eq:LQR-w/o-state-fb} into $\min_{Z\in\mathcal{V}_\mathrm{sdp}} \langle Q, Z_{11} \rangle +\langle R,Z_{22}\rangle$,
which is the same as the convex SDP below
\begin{equation} \label{eq:SDP-primal-w/x_0}
    \begin{aligned}
   J_1^* = \min_{Z} \quad & \langle Q, Z_{11} \rangle +\langle R,Z_{22}\rangle \\
    \mathrm{subject~to} \quad & AZ_{11}+BZ_{12}^\tr +Z_{11}A^\tr +Z_{12}B^\tr + x_0x_0^\tr= 0, \\
    &Z=\begin{bmatrix}
        Z_{11} & Z_{12} \\ Z_{12}^\tr & Z_{22}\end{bmatrix}\succeq  0,
\end{aligned}
\end{equation}
which takes the same form as \cref{eq:SDP-primal}.
We can
 restrict \eqref{eq:LQR-w/o-state-fb} into a non-convex problem below
\begin{equation} \label{eq:SDP-primal-restriction}
    \begin{aligned}
    J_2^* = \min_{Z \in \mathcal{V}_{\mathrm{static}}} \quad & \langle Q, Z_{11} \rangle +\langle R,Z_{22}\rangle.  
\end{aligned}
\end{equation}
It is clear that we have $J_1^* \leq J^* \leq J_2^*$. One expected (yet still surprising) result is that they are all equal under \cref{assumption:controllability}. 

\begin{theorem} \label{theorem:linear-optimality}
    Suppose \cref{assumption:controllability} holds, and let $x_0 \neq 0$. We have $J_2^* = J^* = J_1^*$. Furthermore, 
    an optimal solution to \eqref{eq:LQR-w/o-state-fb}
    is given by $u = K^\star x$ with $K^\star = -R^{-1}B^\tr P^\star$, where $P^\star$ is the unique positive semidefinite solution to \cref{eq:Riccati-eq}. 
\end{theorem}

While the search spaces are different in \cref{eq:SDP-primal-restriction,eq:SDP-primal-w/x_0} compared to that of \cref{eq:LQR-w/o-state-fb},  \Cref{theorem:linear-optimality} confirms that the inner and outer approximations \cref{eq:SDP-primal-restriction,eq:SDP-primal-w/x_0} do not compromise the optimality for \cref{eq:LQR-w/o-state-fb}, as they all contain the same optimal solution; see \Cref{fig:inclusion-V} for an illustration. 
This result also confirms that the class of linear policies is rich enough for the original LQR formulation \cref{eq:LQR-w/o-state-fb}. 
\textbf{Proof of \cref{theorem:linear-optimality}}: The proof idea is almost identical to that in \Cref{section:strong-duality}. We rely on the primal and dual analysis, especially the KKT condition. In particular, the Lagrange dual of \eqref{eq:SDP-primal-w/x_0} is
\begin{equation} \label{eq:Lagrange-dual-w/x_0}
    \begin{aligned}
        \tilde{J}^* = \max_{P}  &\quad\langle x_0x_0^\tr, P\rangle \\
        \text{subject to} &\quad \begin{bmatrix}
    A^\tr P + PA + Q & PB \\ B^\tr P & R
\end{bmatrix}\succeq 0. 
    \end{aligned}
\end{equation}
With \Cref{assumption:controllability}, \Cref{lemma:primal-dual-SDPs-LQR} also guarantees the strong duality between both \cref{eq:SDP-primal-w/x_0,eq:Lagrange-dual-w/x_0}, i.e., $J^*_1 = \tilde{J}^*$. 
 From \cref{lemma:Riccati-equation}, 
 the unique positive semidefinite solution $P^\star$ to the ARE \cref{eq:Riccati-eq} is one optimal solution to the dual problem \cref{eq:Lagrange-dual-w/x_0}. This dual problem \cref{eq:Lagrange-dual-w/x_0} is strictly feasible from \cref{lemma:Riccati-inequality}.

Then, from the KKT optimality analysis of \cref{eq:SDP-primal-w/x_0,eq:Lagrange-dual-w/x_0}, listed in \cref{eq:KKT} by setting  $W=x_0x_0^\tr$, 
 we can now apply the argument as
 \cref{eq:KKT}--\cref{eq:primal-feasibility-w/Zstar} and observe
that the optimal solution $Z^\star$ to \cref{eq:SDP-primal-w/x_0} must be non-zero and unique\footnote{Here, we  have used the assumption $x_0 \neq 0$; otherwise $Z^\star=0$ and any $K\in\mathcal{K}$ will be optimal.},
which is further in the form of 
$$
Z^\star = \begin{bmatrix}
    Z_{11}^\star & Z_{11}^\star K^{\star\tr} \\
    K^\star Z_{11}^\star & K^\star Z^\star_{11}K^{\star\tr}
\end{bmatrix}, \quad \text{with } K^\star = -R^{-1}B^\tr P^\star.
$$
Thus, the optimal solution $Z^\star$ to \cref{eq:SDP-primal-w/x_0} also belongs to $\mathcal{V}_{\mathrm{static}}$. This proves that $J^*_2 \leq J_1^*$.
Combining it with $J^*_1 \leq J^* \leq J_2^*$, we have established $J_2^* = J^* = J_1^*$.

\begin{remark}
The Gramian representation in \cref{eq:setV,eq:LQR-w/o-state-fb-gram} is one key ingredient to our primal and dual analysis. This idea was pioneered in \cite{hotz1987covariance}, and has recently been used in 
linear quadratic control \cite{gattami2009generalized,bamieh2024linear} and $\mathcal{H}_\infty$ analysis \cite{gattami2015simple}. It is interesting to observe that the SDP relaxation \cref{eq:SDP-primal-w/x_0} from $\mathcal{V} \subseteq \mathcal{V}_{\mathrm{sdp}}$ is the same as the rank relaxation in \cref{eq:SDP-primal}. This leads to an interesting interpretation: dropping the rank constraint in \cref{eq:SDP-primal-rank} essentially allows us to search over a space bigger than the space of all stabilizing (potentially nonlinear) policies. The KKT analysis from the inherent linear and quadratic structure in LQR allows us to certify that a linear policy can indeed achieve the globally optimal performance among all possible nonlinear stabilizing policies.  \hfill$\square$
\end{remark}

\begin{remark}[Random initial conditions]
Although \cref{eq:LQR-w/o-state-fb} is written with a fixed initial state, the global optimality in \cref{theorem:linear-optimality} holds for every 
$x_{0}\in\mathbb{R}^{n}$.~Hence, if $x_{0}$ is a random vector with
$\mathbb{E}\|x_{0}\|^{2}<\infty$, taking expectations yields
$
   \mathbb{E}[J(u;x_{0})]
   \geq 
   \mathbb{E}[J(K^\star;x_{0})],
$ 
which shows that $K^\star$ is also globally optimal for the stochastic LQR
formulation \cref{eq:LQR} over the class of admissible (potentially nonlinear) policies.  \hfill $\square$
\end{remark}

\section{Conclusion} \label{section:conclusion}

In this paper, we have revisited three fundamental properties of policy optimization in LQR: strong duality in the nonconvex policy optimization formulation, the gradient dominance property, and the global optimality of linear static policies. Through primal-dual analysis and the \texttt{ECL} framework, we have provided a clearer derivation of strong duality and established a broader and more precise characterization of gradient dominance. We hope these findings contribute to a deeper understanding of optimization landscapes in the classical LQR problem. More broadly, we believe that the primal-dual perspective can offer valuable insights into other fundamental control problems, including robust control and beyond. We are also interested in the principled design of first-order algorithms for both smooth and nonsmooth policy optimization in control problems. 

\paragraph*{Acknowledgements}
We thank Mr. Chih-Fan Pai at UC San Diego for discussions on the primal and dual formulation of the LQR problem. 

\small 
\addcontentsline{toc}{section}{References}
{
\bibliographystyle{ieeetr}
\bibliography{ref.bib}
}

\newpage
\tableofcontents
\newpage

\appendix

\numberwithin{equation}{section}
\numberwithin{example}{section}
\numberwithin{remark}{section}
\numberwithin{assumption}{section}
\numberwithin{lemma}{section}
\numberwithin{proposition}{section}

\vspace{10mm}
\noindent\textbf{\Large Appendix}
\vspace{5mm}

In this appendix, we review extra background, provide further discussions, and complete all the proofs. We organize the appendix as follows: 
\begin{itemize}
    \item \cref{appendix:strong-duality} reviews some technical background of this paper.
    \item \cref{appendix:details-of-section-3} presents technical details and proofs in \cref{section:strong-duality}.
    \item \cref{appendix:details-GD} presents technical proofs omitted in \cref{section:GD}.
    \item \cref{appendix:non-uniqueness} presents some discussions on the nonuniqueness of optimal gains in LQR.
    \item \cref{appendix:LQR} reviews a classical proof of the global optimality of linear static policies in LQR.
\end{itemize}


\section{Background} \label{appendix:strong-duality}

In this section, we review some background on strong duality in SDPs, gradient dominance and its related regularity conditions,  the S-lemma and its connection to LQR. 

\subsection{Strong duality in SDPs}\label{appendix:SDP-review}
We here review the primal and dual SDPs and their strong duality results (see \cite{vandenberghe1996semidefinite} for more detail). 
The standard primal SDP is a problem of the form:
\begin{align}\label{eq:sdp-primal-standard}
    \begin{aligned}
p^*=\min_{X} &\quad \langle C,X\rangle \\
\text {subject to }& \quad  
\mathcal{A}(X)=b,\\
& \quad X \in\mathbb{S}^r_+,
\end{aligned}
\end{align}
where $b\in\mathbb{R}^m,\,C\in\mathbb{S}^r$ and
$\mathcal{A}:\mathbb{S}^r\to\mathbb{R}^m$ is a linear map.
Its Lagrange dual problem is 
\begin{align}\label{eq:sdp-dual-standard}
    \begin{aligned}
d^*=
\max_{y\in\mathbb{R}^m}& \quad b^\tr y \\
 \text {subject to }& \quad 
 C - \mathcal{A}^*(y) \in\mathbb{S}^r_+, 
\end{aligned}
\end{align}
where $A^*:\mathbb{R}^m\to\mathbb{S}^r$ is the adjoint of $\mathcal{A}(\cdot)$ identified by
$
 \left\langle
    y,\mathcal{A}(X)
    \right\rangle=
    \left\langle
    \mathcal{A}^*(y),X
    \right\rangle, \forall y \in \mathbb{R}^m, X \in \mathbb{S}^r.  
$ 
For this primal-dual pair, the KKT optimality condition can be written as
\begin{itemize}
    \item Primal feasibility: 
    $X \in \mathbb{S}_{+}^r$ and
    $\mathcal{A}(X)=b$. 
    \item Dual feasibility: $C - \mathcal{A}^*(y)\in\mathbb{S}^r_+$.
    \item Complementary slackness: $\left\langle 
    X
    ,C-\mathcal{A}^*(y)\right\rangle=0$.
\end{itemize}

Here, \cref{eq:sdp-primal-standard} is said to be \textit{strictly feasible} if there exists $X \in \mathbb{S}^r_{++}$ such that $\mathcal{A}(X) = b$, and  \cref{eq:sdp-dual-standard} is \textit{strictly feasible} if there exists $y\in\mathbb{R}^m$ such that $C - \mathcal{A}^*(y)\in\mathbb{S}^r_{++}$. We have the following result \cite[Theorem 3.1]{vandenberghe1996semidefinite}. 

\begin{lemma}\label{theorem:SDP-duality}
    Consider the primal-dual pair \cref{eq:sdp-primal-standard} and \cref{eq:sdp-dual-standard}.
    The following statements hold.
    \begin{enumerate}
        \item Weak duality: $d^*\leq p^*$.
        \item Strong duality-1: if $d^*$ is finite and the dual problem \cref{eq:sdp-dual-standard} is strictly feasible, 
        then the primal problem \cref{eq:sdp-primal-standard} is solvable, i.e., its minimum is attained.
        Moreover, the duality gap is zero, i.e., $d^*=p^*$.
        \item Strong duality-2: if $p^*$ is finite and the primal problem \cref{eq:sdp-primal-standard} is strictly feasible, 
        then the dual problem \cref{eq:sdp-dual-standard} is solvable, i.e., its maximum is attained.
        Moreover, the duality gap is zero, i.e., $d^*=p^*$.
        \item Complementary slackness:
        Assume the strong duality holds and both problems are solvable.
        Then, 
        $\left\langle 
    X^\star
    ,C-\mathcal{A}^*(y^\star)\right\rangle=0$ holds for any primal-dual pair of solutions 
    $(X^\star,y^\star)$.
    \end{enumerate}
\end{lemma}

The complementary slackness is one key condition in the strong duality. 
One sufficient condition for strong duality is 
the existence of a strictly feasible solution.

\subsection{Gradient dominance}\label{appendix:GD-review}

Here, we give a quick review of gradient dominance and its relevant properties. Consider an $L$-smooth function $f:\mathcal{D}\to\mathbb{R}$ with 
an open domain $\mathcal{D}\subset\mathbb{R}^n$. 
Our discussion follows from \cite{karimi2016linear} that considers the case of $\mathcal{D}=\mathbb{R}^n$ and \cite{liao2024error} that addresses nonsmooth functions.

Let 
$x^\star \in S:={\arg\min}_{x\in\mathcal{D}} f(x) \neq \emptyset$, and consider a sublevel set $\mathcal{D}_\nu:=\{
x\in\mathcal{D}|
f(x)\leq \nu\},\,
\nu>f^*:=\min_{x\in\mathcal{D}}f(x)$. We here define four regularity conditions.  
\begin{enumerate}
    \item \textbf{Strong convexity} ($\mu$-SC):
    For all $x,y\in\mathcal{D}_\nu$, we have
    \begin{equation*}
        f(y) \geq f(x)+\langle\nabla f(x), y-x\rangle+\frac{\mu}{2}\|y-x\|^2. 
    \end{equation*}
    \item \textbf{Error Bound} ($\mu$-EB):
    For all $x\in\mathcal{D}_\nu$, we have
    $$ 
        \|\nabla f(x)\| \geq \mu\times
        \mathrm{dist}(x,S).
    $$
    \item \textbf{Gradient Dominance} ($\mu$-GD) (or \textbf{Polyak-Łojasiewicz} (PL)):
    For all $x\in\mathcal{D}_\nu$, we have
    \begin{equation*}
        \frac{1}{2}\|\nabla f(x)\|^2 \geq \mu\left(f(x)-f^*\right).
    \end{equation*}
    \item \textbf{Quadratic Growth} ($\mu$-QG):
    For all $x\in\mathcal{D}_\nu$, we have
    \begin{equation*}
        f(x)-f^* \geq \frac{\mu}{2}\times
        \mathrm{dist}^2(x,S).
    \end{equation*}
\end{enumerate}

For convex functions, the properties 2)--4) are all equivalent, i.e., $\text{(EB)}
        \equiv \text{(GD)} \equiv \text{(PL)}
        \equiv \text{(QG)}$ \cite[Theorem 3.1]{liao2024error}. For $L$-smooth functions, we have the following results. 
\begin{proposition}
    Assume that $f:\mathcal{D}\to\mathbb{R}$ is an $L$-smooth function with 
    $L>0$,
    $f^* = \min_{x\in\mathcal{D}}f(x)$, and $S:={\arg\min}_{x\in\mathcal{D}}f(x)\neq \emptyset$.
    Assume $\mathcal{D}_\nu$ is closed on $\mathbb{R}^n$.
    Then, over the sublevel set $\mathcal{D}_\nu$,
    we have
    \begin{equation*}
        \text{($\mu_1$-SC)}\to\text{($\mu_2$-EB)}
        \to \text{$(\mu_3\text{-GD/PL})$}
        \to\text{($\mu_2^\prime$-EB)}
        \to 
        \text{$\left(\mu_4\text{-QG}\right)$}
    \end{equation*}
    with 
    $\mu_1>0$,
    $\mu_2\in(0,\mu_1]$, $\mu_3\in \left(0,
    \frac{\mu_2^2}{L}
    \right]$, 
    $\mu_2^\prime\in(0,\mu_3]$,
    and $\mu_4\in \left(0,
    \frac{\mu_2^\prime}{2}
    \right]$.
\end{proposition}

One can show this result by combining the discussions in \cite{karimi2016linear,liao2024error} and Ekeland's
variational principle \cite[Proposition 1.43]{rockafellar1998variational}.
It is known that the LQR problem is $L$-smooth over a sublevel set; thus quadratic growth is weaker than gradient dominance. If quadratic growth fails, then gradient dominance cannot hold.

\subsection{S-lemma and its connections with LQR} \label{appendix:S-lemma}

Here, we discuss a similarity between LQR and the celebrated \textit{S-lemma} \cite{polik2007survey}
from the perspective of quadratically constrained quadratic programs (QCQPs): they both enjoy zero duality gap.

We observe that
the LQR problem \cref{eq:LQR-H2-SDP-1} has a QCQP-like structure (see \cref{eq:S-lemma-1}).
Both its Lagrange dual relaxation \cref{{eq:Lagrange-dual}} and primal Shor's relaxation \cref{eq:SDP-primal} achieve zero duality gaps: all their optimal values are the same under \cref{assumption:controllability} (\cref{theorem:strong-duality-1,theorem:strong-duality-2}). 
This tightness of the SDP relaxation (i.e., zero duality gap) might be viewed as an extension of the \textit{zero duality gap} in a QCQP with one single quadratic constraint that the S-lemma guarantees. 

A standard version of the S-lemma is as follows.
\begin{lemma}[S-lemma \cite{polik2007survey}]\label{proposition:S-lemma}
    Consider two quadratic functions $f(x)=x^\tr Q_0 x+2q_0^\tr x+c_0$ and $g(x)=x^\tr Q_1 x+2q_1^\tr x+c_1$. Suppose $\exists x_0 \in \mathbb{R}^n$ such that $g\left(x_0\right)>0$. Then, we have 
\begin{equation*}
\left\{x \in \mathbb{R}^n \mid g(x) \geq 0\right\} \subseteq\left\{x \in \mathbb{R}^n \mid f(x) \geq 0\right\}
\end{equation*}
if and only if there exist $\lambda \geq 0$, such that $f(x) \geq \lambda g(x), \forall x \in \mathbb{R}^n$.
\end{lemma}

\paragraph*{A QCQP with one single quadratic constraint.} 
Consider the following QCQP
\begin{equation}\label{eq:QCQP-f-g}
    \begin{array}{rl}
q_\mathrm{opt}=\displaystyle\inf _x & f(x) \\
\text { subject to } & g(x) \leq 0,
\end{array}
\end{equation}
where $f(x)=x^\tr Q_0 x 
 +2q_0^\tr x + c_0
$ and $g(x)=x^\tr Q_1 x
 +2q_1^\tr x
+c_1$.
Since $Q_0$ or $Q_1$ may not be positive semidefinite, this problem can be nonconvex. Still, it is well-known that \cref{eq:QCQP-f-g} has an exact SDP relaxation and zero duality gap when it is strictly feasible.   

\vspace{3pt}

\noindent \textit{Lagrange dual SDP}. It is not difficult to derive that the Lagrange dual of \cref{eq:QCQP-f-g} is an SDP of the form 
\begin{equation}\label{eq:dual-QCQP}
\begin{aligned}
    d_\mathrm{opt}=\sup_{\gamma \in \mathbb{R}, \lambda \geq 0}& \quad \gamma \\
    \text{subject to}&\quad
    \begin{bmatrix}
       c_0 + \lambda c_1-\gamma & (q_0 + \lambda q_1)^\tr\\
      q_0 + \lambda q_1&Q_0+\lambda Q_1
    \end{bmatrix}\succeq 0.
\end{aligned}
\end{equation}
By weak duality, we always have  $q_\mathrm{opt} \geq d_\mathrm{opt}$. 

\vspace{3pt}
\noindent  \textit{Primal SDP from Shor's relaxation}. Meanwhile, we can derive a primal SDP from Shor's relaxation \cite{shor1987quadratic} by introducing $X = xx^\tr$: 
\begin{equation} \label{eq:primal-SDP-S-lemma}
    \begin{aligned}
p_\mathrm{opt} = \inf_{X, x} \quad &
\left\langle Q_0,X\right\rangle + 2q_0^\tr x + c_0
\\
\text { subject to } \quad & 
\left\langle Q_1,X \right\rangle +2q_1^\tr x +c_1 \leq 0,
\\
& \begin{bmatrix}
    1 & x^\tr \\ x & X
\end{bmatrix} \succeq 0,
\end{aligned}
\end{equation}
where we have relaxed $ X =xx^\tr$ into $X \succeq xx^\tr$. It is straightforward to verify that \cref{eq:dual-QCQP,eq:primal-SDP-S-lemma} are primal and dual to each other. Therefore, we have 
$
q_\mathrm{opt} \geq p_\mathrm{opt} \geq d_\mathrm{opt}. 
$
From S-lemma in \cref{proposition:S-lemma}, it is well-known that the zero duality gap holds for \cref{eq:QCQP-f-g,eq:dual-QCQP,eq:primal-SDP-S-lemma}
\begin{proposition}[Zero duality gap] \label{proposition:zero-duality-gap}
    Consider the QCQP \cref{eq:QCQP-f-g} and its Lagrange dual and Shor's relaxations \cref{eq:dual-QCQP,eq:primal-SDP-S-lemma}. Suppose there exists $x_0 \in \mathbb{R}^n$ such that $g(x_0) < 0$. Then we have $q_\mathrm{opt} = p_\mathrm{opt} = d_\mathrm{opt}$.
\end{proposition}

We remark that the LQR problem \cref{eq:LQR-H2-SDP-1} may be viewed as a QCQP with a quadratic matrix equation after a suitable transformation; see \cref{eq:S-lemma-1}. Its Lagrange dual and Shor's relaxations are given in \cref{eq:Lagrange-dual,eq:SDP-primal} respectively. The process from \cref{eq:LQR-H2-SDP-1} to \cref{eq:Lagrange-dual,eq:SDP-primal} are analogous to that from \cref{eq:QCQP-f-g} to \cref{eq:dual-QCQP,eq:primal-SDP-S-lemma}, with one main difference being that \cref{eq:LQR-H2-SDP-1} involves a quadratic matrix equation. As a result, the zero duality gap established in \cref{theorem:strong-duality-1,theorem:strong-duality-2} mirrors the result in \cref{proposition:zero-duality-gap}. The proofs for all these results are based on a primal and dual perspective, with one main technical difference: the proof of \cref{proposition:zero-duality-gap} relies on the S-lemma, while the proofs of \cref{theorem:strong-duality-1,theorem:strong-duality-2} depend on results in ARE and KKT analysis.   

\textit{Proof of \Cref{proposition:zero-duality-gap}}.  
 Similar to the LQR problem, the proof takes the following two steps:
\begin{enumerate}
    \item The Lagrange dual problem of \cref{eq:QCQP-f-g}, given in \Cref{eq:dual-QCQP}, is also the dual of the Shor's relaxation \cref{eq:primal-SDP-S-lemma}. We thus have  
    $$
q_\mathrm{opt} \geq p_\mathrm{opt} \geq d_\mathrm{opt}. 
$$
    \item We next apply S-lemma to \cref{eq:QCQP-f-g} to show that
    strong duality holds for the nonconvex QCQP \cref{eq:QCQP-f-g} and its Lagrange dual \Cref{eq:dual-QCQP}, i.e., $q_\mathrm{opt} = d_\mathrm{opt}$. We thus confirm that dropping a rank constraint in \cref{eq:primal-SDP-S-lemma} preserves the optimality.
\end{enumerate}
The first step is standard. For the second step, we apply S-lemma to \cref{eq:QCQP-f-g}. 
Applying the S-lemma, we obtain the following equivalent formulation:
\begin{equation*}
    \begin{aligned}
q_\mathrm{opt}=\sup _{\gamma, \lambda} &\quad \gamma \\
\text { subject to } & f(x)-\gamma \geq-\lambda g(x), \forall x \in \mathbb{R}^n \\
& \lambda \geq 0, 
\end{aligned}
\end{equation*}
which is equivalent to \cref{eq:dual-QCQP}. We thus have $q_\mathrm{opt} = d_\mathrm{opt}$. This concludes the proof.

\section{Technical details in \Cref{section:strong-duality}}\label{appendix:details-of-section-3}

In this section, we provide technical details of \cref{section:strong-duality}, including a proof of \cref{lemma:Lagrangian-explicit} (the dual problem of \cref{eq:LQR-H2-SDP-1}), 
equivalence between \eqref{eq:LQR-H2-SDP-1} and \eqref{eq:SDP-primal-rank}, and a proof of \cref{lemma:primal-dual-SDPs-LQR} (the strong duality between \cref{eq:Lagrange-dual} and \cref{eq:SDP-primal}).

\subsection{Proof of \cref{lemma:Lagrangian-explicit}} \label{appendix:proof-of-dual-function}

   Note that the Lagrangian can be written as 
$$
L(K,X,P) = \left\langle \begin{bmatrix}
    A^\tr P + PA + Q & PB \\ B^\tr P & R
\end{bmatrix}, \begin{bmatrix}
    X & XK^\tr \\ K X & KXK^\tr
\end{bmatrix} \right\rangle + \langle W,P \rangle. 
$$
Since $R \succ 0$, the Schur complement confirms the following equivalence 
\begin{equation} \label{eq:equivalence-Schur}
\begin{bmatrix}
    A^\tr P + PA + Q & PB \\ B^\tr P & R
\end{bmatrix} \succeq 0 \quad \Leftrightarrow \quad A^\tr P + PA + Q - PBR^{-1}B^\tr P \succeq 0. 
\end{equation}
If the dual variable $P$ does not satisfy the left LMI in \cref{eq:equivalence-Schur}, then $A^\tr P + PA + Q - PBR^{-1}B^\tr P$ has at least one negative eigenvalue $\lambda < 0$, with a non-zero eigenvector $v \in \mathbb{R}^n$. In this case, we choose 
$$
K = -R^{-1}B^\tr P, \qquad X=\eta vv^\tr,\,\eta>0.
$$
Then, it is straightforward to see that 
$$
\begin{aligned}
\left\langle \begin{bmatrix}
    A^\tr P + PA + Q & PB \\ B^\tr P & R
\end{bmatrix}, \begin{bmatrix}
    X & XK^\tr \\ KX  & KXK^\tr
\end{bmatrix} \right\rangle &= \left\langle \begin{bmatrix}
    I \\K 
\end{bmatrix}^\tr \begin{bmatrix}
    A^\tr P + PA + Q & PB \\ B^\tr P & R
\end{bmatrix} \begin{bmatrix}
    I \\ K
\end{bmatrix}, X\right\rangle  \\
&= \left\langle A^\tr P + PA + Q - PBR^{-1}B^\tr P, X \right\rangle\\
&= \eta \lambda \|v\|^2 < 0.
\end{aligned} 
$$
This indicates that the dual function results in $g(P) = -\infty$ (as we choose $\eta \to \infty$).  

If the dual variable $P$ satisfies the left LMI in \cref{eq:equivalence-Schur}, then the minimum in \cref{eq:dual-function} over $X\succeq 0$ is clearly $\langle W, P\rangle$ by choosing $X = 0, K = 0$. This gives the explicit form in \cref{eq:dual-function-explicit-form}.

\subsection{The equivalence between \eqref{eq:LQR-H2-SDP-1} and \eqref{eq:SDP-primal-rank}}\label{section:equivalence}

Here, we show the equivalence between \eqref{eq:LQR-H2-SDP-1} and \eqref{eq:SDP-primal-rank} with $W\succ0$, which follows from a well-known procedure to relax QCQP problems to SDPs.
In particular, we can construct a feasible point of \eqref{eq:SDP-primal-rank} with the same cost value from that of \eqref{eq:LQR-H2-SDP-1}, and vice versa.
For the case of $W\nsucc 0$, see \Cref{remark:W-psd-rank}.

First, we show that a feasible pair $(K,X)$ to \eqref{eq:LQR-H2-SDP-1} constitutes a feasible $Z$ to \eqref{eq:SDP-primal-rank} that also preserves the cost value.
Consider the LQR problem \eqref{eq:LQR-H2-SDP-1}. Since $X\succeq 0$, we can write $X = VV^\tr$ for some $V \in \mathbb{R}^{n \times n}$. 
This problem can be equivalently rewritten as
\begin{equation}\label{eq:S-lemma-1}
\begin{aligned}
\min_{K, V}\ \ & \mathrm{tr}\left(QVV^\tr+V^\tr K^\tr RKV\right)   \\
\text{subject to}\ \ & \left(AV+BKV\right)V^\tr + V(AV+BKV)^\tr + W = 0. 
\end{aligned}
\end{equation}
This problem is quadratic in terms of the matrix variable $V \in \mathbb{R}^{n \times n}$ and $KV \in \mathbb{R}^{m \times n}$. For the Shor's relaxation process,  by defining 
\begin{equation*}
    Z = 
    \begin{bmatrix}
        Z_{11} &Z_{12}\\
        Z_{12}^\tr & Z_{22} 
    \end{bmatrix}
    :=
    \begin{bmatrix}
        V\\
        KV
    \end{bmatrix}
    \begin{bmatrix}
        V\\
        KV
    \end{bmatrix}^{\tr},
\end{equation*}
we obtain
\begin{equation} \label{eq:S-lemma-2}
\begin{aligned}
\min_{Z}\ \ & \mathrm{tr}\left(QZ_{11}\right)+\mathrm{tr}\left(RZ_{22}\right)   \\
\text{subject to}\ \ & 
AZ_{11}+BZ_{12}^\tr+Z_{11}A^\tr +Z_{12}B^\tr + W =0,\\
&Z\succeq 0,\, \mathrm{rank}(Z)\leq n,
\end{aligned}
\end{equation}
where the rank constraint comes from the definition of $Z$. 
Notice that this problem coincides with \eqref{eq:SDP-primal-rank}.

Conversely, for a feasible point $Z$ to \eqref{eq:SDP-primal-rank} with $W\succ0$, we construct one $(K,X)$ that is feasible to \eqref{eq:LQR-H2-SDP-1} and preserves the cost value.
First, since $\mathrm{rank}(Z)\leq n$, we can factorize $Z$ as
\begin{equation*}
    Z =\begin{bmatrix}
        V_1\\
        V_2
    \end{bmatrix}
    \begin{bmatrix}
        V_1\\
        V_2
    \end{bmatrix}^{\tr} = \begin{bmatrix}
        V_1V_1^\tr &V_1V_2^\tr\\
        V_2V_1^\tr &V_2V_2^\tr
    \end{bmatrix},
\end{equation*}
where $V_1\in\mathbb{R}^{n\times n}$ and $V_2\in\mathbb{R}^{m\times n}$.
It now suffices to show $V_1$ is invertible when $W\succ0$;
we can then set $K=V_2V_1^{-1}$ and $X=V_1V_1^\tr$ and complete the proof.
To prove this claim by contradiction, assume 
there is a nonzero vector $v$ such that $V_1^\tr v=0$.
Then, we have
\begin{equation*}
    v^\her (AZ_{11}+BZ_{12}^\tr+Z_{11}A^\tr +Z_{12}B^\tr + W )v
    = v^\her Wv =0,
\end{equation*}
which contradicts $v^\her W v>0$. Hence, $V_1$ is non-singular.
By noting that
\begin{align*}
    \mathrm{tr}\left(QZ_{11}\right)+\mathrm{tr}\left(RZ_{22}\right)   
    = &
    \mathrm{tr}\left(QV_1V_1^\tr\right)
    +
    \mathrm{tr}\left(V_2^\tr RV_{2}\right)\\
    =&
     \mathrm{tr}\left(QX\right)
    +
    \mathrm{tr}\left(V_1^\tr (V_1^\tr) ^{-1}V_2^\tr RV_{2}V_1^{-1}V_1\right)
    =
    \mathrm{tr}\left(QX\right)+
    \mathrm{tr}\left(K^\tr RKX\right),
\end{align*}
the pair $(K,X)=(V_2V_1^{-1},V_1V_1^\tr )$ is feasible to \eqref{eq:LQR-H2-SDP-1}, and the cost value remains equal.

\begin{remark}[When $W$ is positive semidefinite]\label{remark:W-psd-rank}
It is worth noting that this rank-constraint SDP \eqref{eq:SDP-primal-rank} for $W\nsucc0$ gives a lower-bound of \eqref{eq:LQR-H2-SDP-1} in general by its construction \cref{eq:construction-Z}.
Indeed, the rank constraint $\mathrm{rank}(Z)\leq n$ does not necessarily imply the factorization \cref{eq:construction-Z}.
Nevertheless, by explicitly solving the primal-dual problems \cref{eq:LQR-H2-SDP-1,eq:Lagrange-dual} as in \cref{theorem:strong-duality-2},
we observe that 
a solution $Z^\star$ to \eqref{eq:SDP-primal-rank} always admits the factorization of the form \cref{eq:construction-Z}, and
the optimal values of  \eqref{eq:LQR-H2-SDP-1} and \eqref{eq:SDP-primal-rank} turn out to be equal.
Therefore, the equivalence between \eqref{eq:LQR-H2-SDP-1} and \eqref{eq:SDP-primal-rank} still holds in terms of their optimal values. \hfill$\square$
\end{remark}

\subsection{Proof of \Cref{lemma:primal-dual-SDPs-LQR}} \label{appendix:primal-dual-SDPs-LQR}

We here present a proof of \cref{lemma:primal-dual-SDPs-LQR} by leveraging the standard results in SDPs (see \cref{theorem:SDP-duality}).
We first transform \eqref{eq:SDP-primal} into the primal problem \cref{eq:sdp-primal-standard}, from which \eqref{eq:Lagrange-dual} naturally arises as its dual.
Then, the desired results follow from \cref{theorem:SDP-duality}. 

The problem \cref{eq:SDP-primal} can be rewritten as
\begin{equation} \label{eq:sdp-primal-2}
\begin{aligned}
   p_1^* = \min_{Z} \quad & \left\langle \begin{bmatrix}
       Q & 0\\
       0& R
   \end{bmatrix}, Z\right\rangle \\
    \mathrm{subject~to} \quad & -\left(AZ_{11}+BZ_{12}^\tr +Z_{11}A^\tr +Z_{12}B^\tr \right) = W, \\ & Z=\begin{bmatrix}
        Z_{11} & Z_{12} \\ Z_{12}^\tr & Z_{22}
    \end{bmatrix}\in \mathbb{S}^{n+m}_+.
\end{aligned}
\end{equation}
Thus, by defining 
$
    \mathcal{A}(Z):=-\left(AZ_{11}+BZ_{12}^\tr +Z_{11}A^\tr +Z_{12}B^\tr\right),\,
 C:=\begin{bmatrix}
    Q &0\\
    0& R
\end{bmatrix}, 
$
we see that \eqref{eq:sdp-primal-2} is of the form \eqref{eq:sdp-primal-standard}.
Now,
for $P\in\mathbb{S}^n$,
we can compute the adjoint of $\mathcal{A}(\cdot)$ as
\begin{align*}
\langle\mathcal{A}(Z),P
 \rangle
 = -\langle 
 AZ_{11}+BZ_{12}^\tr +Z_{11}A^\tr +Z_{12}B^\tr,P
 \rangle
=&- \mathrm{tr}\left(
 \begin{bmatrix}
     Z_{11}& Z_{12}\\
     Z_{12}^\tr & Z_{22}
 \end{bmatrix} \begin{bmatrix}
     A^\tr P+PA & PB \\
     B^\tr P & 0
\end{bmatrix}
 \right).
\end{align*}
This confirms that 
$
\mathcal{A}^*(P) = -\begin{bmatrix}
     A^\tr P+PA & PB \\
     B^\tr P & 0
\end{bmatrix}. 
$ 
Therefore, the dual problem of \cref{eq:sdp-primal-2} becomes 
\begin{align}\label{eq:Lagrange-dual-2}
\begin{aligned}
    d^* = \max_{P\in \mathbb{S}^n} &\quad \langle W,P\rangle \\
    \text{subject to}& \quad 
    C-\mathcal{A}^*(P)=
    \begin{bmatrix}
      A^\tr P+PA+Q & PB \\
     B^\tr P & R\end{bmatrix}\in\mathbb{S}_+^{n+m},
    \end{aligned}
\end{align}
which implies that the SDP \cref{eq:Lagrange-dual} is the dual problem of \cref{eq:SDP-primal}.

Thanks to \cref{lemma:Riccati-equation}, we know that \cref{eq:Lagrange-dual-2} has an optimal solution $P^\star$ with a finite cost value. Furthermore, \Cref{lemma:Riccati-inequality} guarantees that  \cref{eq:Lagrange-dual-2} is strictly feasible. 
Hence, 
the primal problem \cref{eq:SDP-primal} (or \cref{eq:sdp-primal-2}) is solvable, and we also obtain
strong duality. This completes the proof of \Cref{lemma:primal-dual-SDPs-LQR}.

Finally, we note that for any solution pair $(Z^\star,P^\star)$, the following complementary slackness is also fulfilled:
\begin{equation*}
    \left\langle Z^\star,
    C-\mathcal{A}^*(P^\star ) 
    \right\rangle
    =
    \left\langle
    Z^\star ,\begin{bmatrix}
     A^\tr P^\star+P^\star A+Q & P^\star B \\
     B^\tr P^\star & R
\end{bmatrix}
    \right\rangle
    =0.
\end{equation*}


\section{Technical details for \cref{section:GD}}\label{appendix:details-GD}

In this section,
we provide the technical details in \cref{section:GD},
    including \cref{lemma:convergence-gradient-descent-GD} (linear convergence and feasibility of iterates in the gradient descent), \cref{lemma:gradient-Cauchy_direction} (connection between gradients and Cauchy directions), 
    \cref{theorem:J_LQR-gradient_dominance-general} (quadratic growth of $f_\cvx$ in LQR and some details), 
    and \cref{proposition:sufficient-conditions} (three sufficient conditions for \cref{assumption:compactness}).

\subsection{Proof of \cref{lemma:convergence-gradient-descent-GD}}\label{appendix:proof-GD-2}

For self-completeness, we provide some arguments for \cref{lemma:convergence-gradient-descent-GD}, consisting of three main steps. This proof is adapted from \cite[Theorem 4.2]{fatkhullin2021optimizing}. 

        \begin{itemize}
            \item 
        First,
        we show the smoothness of $f$. The coerciveness of $f$  ensures the compactness of  $\mathcal{D}_\nu$. 
        Since $f$ is twice continuously differentiable, its gradient $\nabla f$ is Lipschitz continuous 
        over the compact set $\mathcal{D}_\nu$ with a constant $L>0$. This implies the $L$-smoothness of $f$ on $\mathcal{D}_\nu$.
        
    \item     Next, we show all the gradient iterates are feasible, i.e., $x_k\in\mathcal{D}_v$ for $k\geq 1$, when we choose the step size $ 0< \alpha < 2/L$. 
        It suffices to show $x-\alpha \nabla f(x) \in \mathcal{D}_\nu, \forall x\in\mathcal{D}_\nu$ if $ 0< \alpha < 2/L$.
        Let $x \in \mathcal{D}_\nu$ be arbitrary. Without loss of generality, we suppose $\nabla f(x)\neq 0$.
        Let us define 
        \begin{equation*}
            \varphi(\tau) = f(x-\tau \nabla f(x) ), \quad \text{with}\;\; \tau \in [0,T)
        \end{equation*}
        where 
         $   T= \max\{s| [x, x-s\nabla f(x)] \subset \mathcal{D}_\nu,\,s\geq 0\}$\footnote{For two points $x$ and $y$, $[x,y]$ represents the line segment connecting them.}. This constant $T$ exists since $\mathcal{D}_\nu$ is compact. 
Then, via the chain rule, we have 
$$
\varphi'(\tau) = \langle -\nabla f(x), \nabla f(x-\tau \nabla f(x) )\rangle, \qquad \forall \tau \in [0,T]. 
$$
From this, we can infer that 1) $      \varphi'(0)
            = - \|\nabla f(x)\|^2 <0, $ and 2) $\phi(\tau)$ is $L_1$-smooth over the segment $[0,T]$ with $L_1 = L \|\nabla f(x)\|^2$ since 
\begin{align*}
            |\varphi'(\tau_1)-\varphi'(\tau_2)|
            &\leq
            \|\nabla f(x)\|\times
            \|\nabla f(x-\tau_1\nabla f(x))
            -\nabla f(x-\tau_2\nabla f(x))\|\\
            &\leq \|\nabla f(x)\|\times L\|(\tau_1-\tau_2)\nabla f(x)\| \\
            &\leq L \|\nabla f(x)\|^2 |\tau_1-\tau_2|, \qquad \forall \tau_1, \tau_2 \in [0,T].
        \end{align*}
        
     We then have the following quadratic upper bound:
        \begin{equation*}
            \varphi(\alpha)
            \leq \varphi(0) + \varphi'(0)\alpha + \frac{L_1\alpha^2}{2}
            =\varphi(0) -\|\nabla f(x)\|^2\alpha\left(1-\frac{L}{2}\alpha\right), \qquad \forall \alpha \in [0, T].
        \end{equation*}
        As long as $\alpha \leq 2/L$, we have $\varphi(\alpha)\leq \varphi(0)$, which indicates
         $  x-\alpha\nabla f(x)\in \mathcal{D}_\nu.$ 
        { This also shows that the constant $T$ is at least as large as $2/L$.}
        {We further see that the line segment connecting two iterates $x_k$ and $x_{k+1}=x_k-\alpha \nabla f(x_k)$ also lies in $\mathcal{D}_\nu$.}
        
  \item       Finally, 
        we establish the linear convergence of gradient descent $x_{k+1}=x_k-\alpha \nabla f(x_k)$. 
        The $L$-smoothness of $f(x)$ over $\mathcal{D}_\nu$ yields 
        \begin{equation} \label{eq:linear-rate-GD-upper-bound}
        \begin{aligned}
            f(x_{k+1})-f(x_k)
            \leq & \langle \nabla f(x_k), x_{k+1}-x_k\rangle + \frac{L}{2}\|x_{k+1}-x_k\|^2
            \\
            \leq& -\alpha\|\nabla f(x_k)\|^2+\frac{L}{2}\|\alpha \nabla f(x_k)\|^2
            = - \frac{2\alpha-L\alpha^2}{2} \|\nabla f(x_k)\|^2.
        \end{aligned}
        \end{equation}
        Note that we have $2\alpha-L\alpha^2>0$ for any $\alpha\in(0,2/L)$.
        Substituting the gradient dominance \cref{eq:PL-f} into \cref{eq:linear-rate-GD-upper-bound}, we obtain
        \begin{equation*}
            f(x_{k+1})-f(x_k)
            \leq -\mu(2\alpha-L\alpha^2)\left(f(x_k)-f(x^\star)\right).
        \end{equation*}
        Hence, by subtracting $f(x^\star)$ from both sides and rearranging $f(x_k)$, we arrive at
        \begin{equation} \label{eq:Linear-rate-GD}
            f(x_{k+1})-f(x^\star)
            \leq \gamma \times \left(f(x_k)-f(x^\star)\right) \qquad \text{with}\;\; \gamma = 1-\mu(2\alpha-L\alpha^2). 
        \end{equation}
        Since $\alpha \in(0,2/L)$, we have
    \begin{equation*}
            \gamma = 1-\mu(2\alpha-L\alpha^2)
            =
            \mu L\left(\alpha^2 - \frac{2\alpha}{L}\right)+1
            =\mu L \left(\alpha-\frac{1}{L}\right)^2 + 1 -  \frac{\mu}{L}
            < 1.
        \end{equation*}
        \end{itemize}

We re-emphasize that it is important to establish the feasibility of all iterates from the gradient descent before applying \cref{eq:Linear-rate-GD,eq:linear-rate-GD-upper-bound} (since the gradient dominance \cref{eq:PL-f} may only hold on the sublevel set $\mathcal{D}_\nu$ and iterates $x_{k+1}$ may become infeasible).   

\subsection{Proof of \cref{lemma:gradient-Cauchy_direction}
}\label{appendix:proof-lemma:gradient-Cauchy_direction}

We here prove the first statement of \cref{lemma:gradient-Cauchy_direction} utilizing the $\ECL$ framework.
The second statement can be shown in the same way (by using \cref{eq:phi_1-upperbound-2} below instead of \cref{eq:phi_1-upperbound}). 
We define $f_\lft:\pi_{x,\xi}(\mathcal{L}_\lft)\to\mathbb{R}$
by
\begin{equation}\label{eq:lifted-function-ECL}
f_\lft(x,\xi) := 
\min_{(x,\gamma,\xi)\in\mathcal{L_\lft}} \ \gamma.
\end{equation}

For $x\in\mathcal{D}_\nu$,
we define $(x,f(x),\xi(x)) \in \mathcal{L}_\lft$,
where $\xi(x) \in \argmin_{\xi\in\{\xi\mid (x,\xi)\in\pi_{x,\xi}(\mathcal{L}_\lft)\}}f_{\lft}(x,\xi)\neq \emptyset$ is guaranteed to exist by the definition of the $\ECL$ and $f_\lft$ in \cref{eq:lifted-function-ECL}.
Using this point,
we also define two points on the convex set $\mathcal{F}_\cvx$ as
\begin{equation}
\label{eq:zeta-zeta_star}
\zeta := (\pi_\zeta\circ\Phi)(x,f(x),\xi(x))
, \quad \zeta^\star:= (\pi_\zeta\circ\Phi)(x^\star,f(x^\star),\xi^\star),
\end{equation}
where $(x^\star,\xi^\star)$ is a minimizer of $f_\lft$.
By \cref{assumption:convex-set},
this point $\zeta$ satisfies
\begin{equation}\label{eq:J_LQR-f_cvx-equality}
    f_\mathrm{cvx}(\zeta)=f_\mathrm{lft}(x,\xi(x))
    = \min_{\xi\in\{
    \xi|(x,\xi)\in \pi_{x,\xi}(\mathcal{L}_\lft)
    \}} f_\lft(x,\xi)
    =f(x),
\end{equation}
which also implies $\zeta,\,\zeta^\star\in\mathcal{X}_\nu$.
For the convex function $f_\mathrm{cvx}$
 in \cref{eq:ECL-convex-function-general}, we can also define the Cauchy direction $g_{\mathrm{c}}$ of $f_{\mathrm{cvx}}$ at $\zeta$ as \eqref{eq:Cauchy-direction-construction}.

Before establishing \cref{lemma:gradient-Cauchy_direction}, we provide a supporting result using the Cauchy direction $g_\mathrm{c}$.
This result 
presents a differentiable curve starting from $(x,f(x),\xi(x))$ that is contained by
the nonconvex lifted set around $(x,f(x),\xi(x))$.
This curve allows us to 
connect the Cauchy direction $g_\mathrm{c}$ to the gradient $\nabla f(x)$.
The proof is provided in \cref{appendix:proof-curve}.

\begin{proposition} \label{proposition:curve}
Suppose the same assumptions as \cref{theorem:gradient-dominance-ECL} (except for the quadratic growth of $f_\cvx$) hold.
Consider $(x,f(x),\xi(x))\in\mathcal{L}_{\lft}$
with $x\in\mathcal{D}_\nu$
and the Cauchy direction $g_\mathrm{c}$ in \cref{eq:Cauchy-direction-construction} at $\zeta$ 
in \cref{eq:zeta-zeta_star}.
Then, the following statements hold:
\begin{itemize}
     \item[a)]
     (Linear upper bound of $f_\mathrm{cvx}$).
    for arbitrary $\varepsilon>0$,
    there exists a $\bar{\tau}>0$ such that
    \begin{equation}\label{eq:upperbound-f_cvx}
    f_\mathrm{cvx}\left(\zeta^+(\tau)\right)
    <
    f_\mathrm{cvx}(\zeta) + \tau\varepsilon - \tau  \|g_{\mathrm{c}}\|^2, \qquad \forall \tau\in(0,\bar{\tau}),
\end{equation}
where $\zeta^+(\tau) := \zeta-\tau g_{\mathrm{c}}$.
    \item[b)]
    (A smooth curve on the nonconvex lifted set).
    For sufficiently small $\hat{\tau}>0$,
    the function $\phi:[0,\hat \tau)\to \mathbb{R}^{d+d_\xi }$
    \begin{align} \label{eq:curve-nonconvex-set}
          \phi(\tau)
      := \left(\phi_1(\tau),\phi_2(\tau)\right)
        = (\pi_{x,\xi}\circ \Phi^{-1})(
         f_\mathrm{cvx}(\zeta) + \tau\varepsilon - \tau  \|g_{\mathrm{c}}\|^2,\zeta-\tau g_\mathrm{c}
        ) 
\end{align}
    satisfies $\phi(0)=(x,\xi(x))$ and
    lies in
    $ \pi_{x,\xi}(\mathcal{L}_\lft)$ for any $\tau\in [0,\hat{\tau}).$
    Moreover,
    \begin{equation*}
        \partial_+ \phi(0) :=
        \lim_{\tau\downarrow 0} \frac{\phi(\tau)-\phi(0)}{\tau}
    \end{equation*}
    exists.
\item[c)] 
(An upper bound of $\|\partial_+ \phi_1(0)\| $ with $\|g_\mathrm{c}\|$). For the smooth curve in \cref{eq:curve-nonconvex-set}, 
we have
\begin{equation}\label{eq:phi_1-upperbound}
    2c_\Phi
    \|\partial_+ \phi_1(0)\| 
    \leq 
\varepsilon' + \|g_{\mathrm{c}}\|
\quad\text{with}\quad
\varepsilon' = 2c_\Phi\|\nabla   (\pi_x\circ\Phi^{-1})(f_\cvx(\zeta),\zeta) \|_2 \varepsilon
\end{equation}
    for some $c_\Phi>0$.
Moreover, if the diffeomorphism $\Phi$ is of the form
    $\Phi(x,\gamma,\xi) = (\gamma,\Psi(x,\xi))$ with
    a diffeomorphism $\Psi:\pi_{x,\xi}(\mathcal{\mathcal{L}_\lft})\to\pi_{\zeta}(\mathcal{F}_\cvx)$,
    we have
\begin{equation}\label{eq:phi_1-upperbound-2}
    2c_\Phi
    \|\partial_+ \phi_1(0)\| 
    \leq 
    \|g_{\mathrm{c}}\|,
\end{equation}
and $c_\Phi>0$ can be chosen as \cref{eq:C_Phi}.
\end{itemize}
\end{proposition}

We illustrate the role of this proposition in \cref{fig:proof_figs}.
This proposition and function $\phi(\tau)$ bridge the gap between the nonconvex sets $\mathcal{D},\,\mathcal{L}_\lft$ and convex set $\mathcal{F}_\cvx$.
Since $\mathcal{L}_\lft$ is nonconvex and $(f_\cvx(\zeta),\zeta)=(f(x),\zeta)\in\mathcal{F}_\cvx$ may be on its boundary, it is  nontrivial to evaluate the relation between $\|\nabla f(x)\|$ and $\partial f_\cvx(\zeta)$.
To resolve this issue, we leverage the Cauchy direction $g_\mathrm{c}$ in $\mathcal{F}_\cvx$ (instead of $\partial f_\cvx(\zeta)$) and build a (semi-)differentiable curve $\phi(\tau)$ that is not on the boundary (as \cref{fig:proof-1}) but always in the domain (as \cref{fig:proof-2}).
This allows us to evaluate 
$\|\nabla f(x)\|$ through 
the directional derivative along $\partial_+ \phi(0)$.

We now establish
\cref{lemma:gradient-Cauchy_direction} as follows.
Let $\varepsilon>0$ be an arbitrary positive number.
Then, it follows from \cref{proposition:curve} (a), (b) that
for any sufficiently small $\tau>0$,
\begin{equation*}
    f(\phi_1(\tau)) \leq 
    f_\mathrm{cvx}\left(\zeta^+(\tau)\right)
<f_\mathrm{cvx}(\zeta)+\tau\varepsilon -\tau\|g_\mathrm{c}\|^2.
\end{equation*}
Using this inequality and $f(x) =f_\mathrm{cvx}(\zeta)$ in \cref{eq:J_LQR-f_cvx-equality},
we have
\begin{align*}
    \lim_{\tau\to+0}\frac{f\left(\phi_1(\tau)\right) - f(x)}{\tau} 
    \leq
    \lim_{\tau\to+0}
    \frac{f_\mathrm{cvx}(\zeta)+\tau\varepsilon -\tau \|g_{\mathrm{c}}\|^2
    -f_\mathrm{cvx}(\zeta)
    }{\tau}
    =\varepsilon -\|g_{\mathrm{c}}\|^2.
\end{align*}
Thanks to the existence of $\partial^+\phi_1(0)$, the left-hand side can be rewritten as
\begin{equation}\label{eq:J_LQR-phi-directional-derivative}
\lim_{\tau\to+0}\frac{f(\phi_1(\tau)) - f(x)}{\tau} 
= \left\langle \nabla f(\phi_1(0)),\partial^+\phi_1(0) \right\rangle
= \left\langle \nabla f(x),\partial^+\phi_1(0) \right\rangle.
\end{equation}
We thus obtain
$
    \left\langle \nabla f(x),\partial^+\phi_1(0) \right\rangle \leq \varepsilon - \|g_{\mathrm{c}}\|^2,
$
Together with the chain rule and \eqref{eq:phi_1-upperbound}, we have 
\begin{align*}
    \|g_{\mathrm{c}}\|^2\leq
    -\left\langle \nabla f(x),\partial^+\phi_1(0) \right\rangle + \varepsilon
    \leq&
    \left\| \nabla f(x)\|\times\|\partial^+\phi_1(0) \right\| + \varepsilon\\
    \leq& \|\nabla f(x)\|\times
\frac{1}{2c_\Phi}
    (\varepsilon' +
    \|g_\mathrm{c}\|)
    + \varepsilon\\
    =&
    \frac{1}{2c_\Phi}
    \times
    \|\nabla f(x)\|\|g_\mathrm{c}\|+ 
    \left(\varepsilon+ \frac{\|\nabla f(x)\|\varepsilon'}{2c_\Phi}
    \right).
\end{align*}
Therefore, recalling that $\varepsilon$ can be chosen arbitrarily
and $\varepsilon'\to0$ as $\varepsilon\to0$,
we arrive at
\begin{equation}\label{eq:Cauchy_direction-gradient-K-X}
    2c_\Phi\|g_{\mathrm{c}}\| \leq
    \|\nabla f(x)\|.
\end{equation}
\begin{figure}
\setlength{\abovecaptionskip}{2pt}
    \centering
  \begin{subfigure}{0.48\linewidth}
           \centering
        \includegraphics[width=\linewidth]{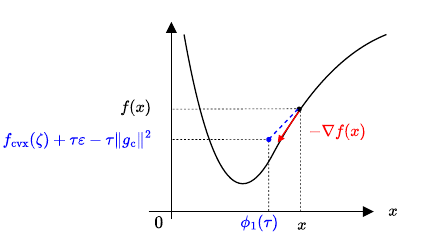}
            \vspace{-5mm}\caption{}
            \label{fig:proof-1}
  \end{subfigure}%
 \begin{subfigure}{0.48\linewidth}
           \centering
        \includegraphics[width=\linewidth]{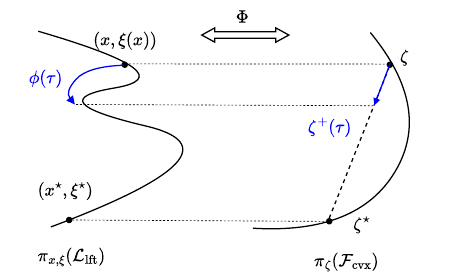}
            \vspace{-5mm}\caption{}
            \label{fig:proof-2}
  \end{subfigure}%
  \caption{Geometric intuition of our proof strategy for \cref{lemma:gradient-Cauchy_direction}:
  (a) an interpretation of
  the upper bound in \cref{proposition:curve} (a).
  The value of $f_\cvx(\zeta)+\tau\varepsilon -\tau\|g_\mathrm{c}\|^2$ provides a linear and strict upper bound of $f$ around $x$;
   (b) an illustration of function $\phi(\tau)$ in \cref{proposition:curve}.
  The function $\phi(\tau)$ is contained by $\pi_{x,\xi}(\mathcal{L}_\lft)$ thanks to the convexity of $\pi_\zeta(\mathcal{F}_\cvx)$ and Cauchy direction $g_\mathrm{c}$
  }
    \label{fig:proof_figs}
\end{figure}

\subsubsection{Proof of \cref{proposition:curve}}\label{appendix:proof-curve}

We here prove the three statements (a)--(c) in \cref{proposition:curve} based on the $\ECL$ framework.

a) Since $f_{\mathrm{cvx}}$ is convex, we know that the function 
$$
l(\tau):= \frac{f_\mathrm{cvx}(\zeta-\tau g_{\mathrm{c}}) - f_\mathrm{cvx}(\zeta)}{\tau}
=
 \frac{f_\mathrm{cvx}\left(\zeta^+(\tau)\right) - f_\mathrm{cvx}(\zeta)}{\tau},
$$
for $\tau$ such that $\zeta-\tau g_{\mathrm{c}} \in \mathrm{dom}(f_\mathrm{cvx})$, is monotonically non-decreasing for $\tau\in\mathbb{R}_+$ \cite[Theorem 23.1]{rockafellar1970convex}.
Here, $\zeta^+(\tau)=\zeta-\tau g_{\mathrm{c}} \in \mathrm{dom}(f_\mathrm{cvx})$ is satisfied for sufficiently small $\tau>0$
 from the convexity of $\mathcal{F}_\mathrm{cvx}$ and \cref{eq:Cauchy-direction-construction}.
By the definition of Cauchy directions, we have $\lim_{\tau \downarrow 0} l(\tau) \leq -\|g_{\mathrm{c}}\|^2 $. Then, for arbitrary positive $\varepsilon>0$, there exists a $\bar{\tau}>0$ such that
\begin{equation*}
    \frac{f_\mathrm{cvx}(\zeta^+(\tau)) - f_\mathrm{cvx}(\zeta)}{\tau} <
     -\|g_{\mathrm{c}}\|^2 + \varepsilon, \qquad \forall \tau\in(0,\bar{\tau}),
\end{equation*} 
which is equivalent to \cref{eq:upperbound-f_cvx}.

b) 
By the definition of $\phi(\tau)$, it is clear that
$\phi(0) = \left(x,\xi(x)\right) $.
Moreover,
{for 
$\tau\in[0,\hat\tau)$ with
sufficiently small $\hat \tau>0$, 
we have
$(f_\cvx(\zeta)+\tau\epsilon-\tau \|g_\mathrm{c}\|^2,\zeta^+(\tau)) \in \mathcal{F}_\cvx$}
from $(f_\cvx(\zeta^+(\tau)),\zeta^+(\tau))\in\mathcal{F}_\cvx$,
which implies\footnote{
The set $ \mathrm{epi}_>(f)$
for $f:\mathcal{D}\to\mathbb{R}$ represents
the strict epigraph 
$\operatorname{epi}_{>}(f) : = \{(x,\gamma) \mid x \in \mathcal{D}, f(x) < \gamma\}.$
}
\begin{equation*}
    (\pi_{x,\gamma}\circ\Phi^{-1})
    \left(
    f_\mathrm{cvx}(\zeta) + \tau\varepsilon - \tau  \|g_{\mathrm{c}}\|^2,
    \zeta^+(\tau)
    \right)
    \in \mathrm{epi}_>(f)\subset
    \mathrm{epi}_\geq(f)
    = \pi_{x,\gamma} (\mathcal{L}_\lft) ,\quad
    \tau \in (0,\hat \tau)
\end{equation*}
from \cref{eq:upperbound-f_cvx} and the definition of the $\ECL$.
This implies that
\begin{align*}
    &\Phi^{-1}\left(f_\mathrm{cvx}(\zeta) + \tau\varepsilon - \tau  \|g_{\mathrm{c}}\|^2,
    \zeta^+(\tau)\right)
    \in \mathcal{L}_\lft,\quad
    \forall \tau\in [0,\hat\tau).
\end{align*}
Thus, for any $\tau\in(0,\hat\tau)$,
the curve $\phi(\tau)$ satisfies
\begin{align*}
      \phi(\tau)
        = \left(\pi_{x,\xi}\circ\Phi
        ^{-1}\right)
        \left(
        f_\mathrm{cvx}(\zeta) + \tau\varepsilon - \tau  \|g_{\mathrm{c}}\|^2,\zeta^+(\tau)
        \right).
\end{align*}

The existence of semi-derivative $\partial_+ \phi(0)$ follows from the differentiability of $\Phi$ in the definition of the $\ECL$ because $(f_\cvx(\zeta)+\tau\varepsilon - \tau\|g_\mathrm{c}\|^2)$ is linear in terms of $\tau$\footnote{Note that for $\tau<0$, $\zeta^+(\tau)$ may be out of the domain $\pi_\zeta(\mathcal{F}_\mathrm{cvx})$ in which $\Phi$ may not be defined.}.


c)
Finally,
for $x\in\mathcal{D}$ with the open set $\mathcal{D}$
and  $T=\pi_{x}\circ\Phi^{-1}$,
we can compute
the semi-derivative of the first element $\phi_1$ at the origin as
\begin{align*}
        \partial_+ \phi_1(0) :=
\nabla T \left(f_\cvx(\zeta),\zeta\right)
        \begin{bmatrix}
        \varepsilon - \|g_\mathrm{c}\|^2\\
            -g_\mathrm{c}
        \end{bmatrix}
\end{align*}
where we use $\pi_x(\cdot)$ in the sense that we extract the first $d$ entries corresponding to $x$.
Therefore, 
the norm of $\phi_1'(0)$ can be upper-bounded as 
\begin{align}\label{eq:lemma-curve-proof-c}
    \partial_+ \phi_1(0) & \leq 
\|\nabla T \left(f_\cvx(\zeta),\zeta\right)\|_2
    \left(\varepsilon +\left(1+\|g_\mathrm{c}\|\right)
    \|g_\mathrm{c}\|
    \right)
    \nonumber
    \\
    &\leq
    \|\nabla T\left(f_\cvx(\zeta),\zeta\right)\|_2\epsilon+
    \|\nabla T\left(f_\cvx(\zeta),\zeta\right)\|_2
    (1+\|g_\mathrm{c}\|)
    \times\|g_\mathrm{c}\|
\end{align}

We next show that $\|g_\mathrm{c}\|$ is bounded for $\zeta$ 
defined by \cref{eq:zeta-zeta_star} with $x\in\mathcal{D}_\nu$
as follows.
The point
$(f_\mathrm{cvx}(\zeta),\zeta^+(\tau))$ is feasible for all sufficiently small $\tau>0$.
Thus,
for any nonzero Cauchy direction $g_\mathrm{c}$
and $\zeta $
the first part of \cref{assumption:convex-set} ensures
\begin{equation*}
    f(T(f_\mathrm{cvx}(\zeta),\zeta-tg_\mathrm{c}))\leq f_\mathrm{cvx}(\zeta-tg_c),
\end{equation*}
with equality at $\tau=0$.
By the definition of Cauchy directions,
the right derivative at $\tau=0$ yields
\begin{equation*}
     -\langle\nabla f(x),\nabla_\zeta T(f_\mathrm{cvx}(\zeta),\zeta)g_\mathrm{c}\rangle
     \leq 
     \lim_{\tau\downarrow 0}
     \frac{f_\mathrm{cvx}(\zeta-\tau g_\mathrm{c})-f_\mathrm{cvx}(\zeta)}{\tau}
    \leq-\|g_\mathrm{c}\|^2.
\end{equation*}
Consequently,
\begin{equation*}
     \|g_c\|^2
    \leq \|\nabla_\zeta T(f_\mathrm{cvx}(\zeta),\zeta)\| \|\nabla f(x)\|\times \|g_c\|.
\end{equation*}
This confirms $\|g_\mathrm{c}\| \leq L_\nu \times D_\nu<\infty$
with $L_\nu\geq \max_{x\in T(\Gamma_\nu)}\|\nabla f(x)\|$
and $D_\nu\geq \max_{(\gamma,\zeta)\in\Gamma_\nu}\|\nabla T(f_\mathrm{cvx}(\zeta),\zeta)\|$, where
$\Gamma_\nu$ is the compact set defined as
$\Gamma_\nu=\{(\gamma,\zeta)\in\mathcal{F}_\mathrm{cvx}\mid f^\star \leq \gamma\leq \nu\}$.
This bound also holds when $g_c=0$.

Consequently,
by setting $   c_\Phi = \frac{1}{2D_\nu (1+D_\nu L_\nu)}>0$,
the inequality \cref{eq:phi_1-upperbound} follows.
Note that 
the second statement \cref{eq:phi_1-upperbound-2} can also be obtained in the same way because 
the semi-derivative of $\phi_1(\tau)$ at $\tau=0$ is reduced to $\partial_+ \phi_1(0) :=
\nabla(\pi_x\circ\Psi^{-1})\left(\zeta\right)
\times (-g_\mathrm{c})$, and we have
\begin{align*}
    \|\partial_+ \phi_1(0)\|\leq
    \|
\nabla\Psi^{-1}\left(\zeta\right)\|
\times \|g_\mathrm{c}\|.
\end{align*}

\subsection{Technical details of \cref{theorem:J_LQR-gradient_dominance-general}}

\label{appendix:GD-LQR-details}

Here, we mainly provide 
a proof of the quadratic growth in \cref{lemma:f_cvx-QG} with the construction of the $\ECL$ in \cref{subsubsection:LQR-PL-proof}.
We first discuss the derivation of the $\ECL$ and show the lemma using the nice properties of the $\ECL$.
Finally, we also present a derivation of the constant $c_\LQR$ in \cref{theorem:J_LQR-gradient_dominance-general}.

\subsubsection{Derivation of the $\ECL$ for LQR}\label{appendix:ECL-construction-LQR}

Here, we provide the derivation of the $\ECL$ in \cref{subsubsection:LQR-PL-proof}.
Consider the mapping
\begin{equation*}
    \Phi(K,\gamma,X) = (\gamma,Y,X)=
    \left(\gamma,(K-K^\star)X,X
    \right).
\end{equation*}
Then, 
based on this transformation, 
the Lyapunov equation \cref{eq:Lyapunov-H2} in $\mathcal{L}_\lft$ 
becomes
\begin{align*}
   0=&(A+BK)X + X(A+BK)^\tr
    + W\\
   =&(A+B(K^\star+K-K^\star))X + X(A+B(K^\star+K-K^\star))^\tr
    + W\\
    =&
    (A+BK^\star)X + X(A+BK^\star)^\tr
    + BY+Y^\tr B^\tr + W.
\end{align*}
Notice that the linear mapping $\mathcal{A}_\star(X)
=(A+BK^\star)X + X(A+BK^\star)^\tr
$ is invertible since $A+BK^\star$ is stable. 

Moreover, we have 
\begin{align*}
    \gamma\geq&\mathrm{tr}(Q+K^\tr RK)X
    \\
    =&\mathrm{tr}\left(Q
    +(K-K^\star+K^\star)^\tr R(K-K^\star+K^\star)
    \right)X\\
    = &\mathrm{tr}\left(QX
    +
    (Y+K^\star X)^\tr R(Y+K^\star X) X^{-1}
    \right) \\
    =&
   \mathrm{tr}\left(
   Q^\star
   X
    +
    Y^\tr RYX^{-1}
    + (K^\star)^\tr RY + Y^\tr RK^\star
    \right)\\
    =& f_\mathrm{cvx}(Y,X)
\end{align*}
with $Q^\star = 
    Q+(K^\star)^\tr RK^\star$.
    Hence, the set $\mathcal{F}_\cvx$ is indeed of the form \cref{eq:ECL-LQR-constructions}.
Since the mapping $\Phi$ becomes a diffeomorphism for $X\succ 0$ under \cref{assumption:compactness}, we can define the $\ECL\,(\mathcal{L}_\lft,\mathcal{F}_\cvx,\Phi)$
for $J_\LQR(K)$ under \cref{assumption:controllability,assumption:compactness}.

From the construction,
the optimal $Y^\star$ is clearly $0$, and we have a compact form of the optimal value as
\begin{equation*}
   f_\mathrm{cvx} (Y^\star,X^\star)
   =f_\mathrm{cvx} (0,X^\star)
   = \mathrm{tr}(Q^\star X^\star).
\end{equation*}
These properties are utilized in the proof of \cref{lemma:f_cvx-QG} below.

\subsubsection{Proof of \cref{lemma:f_cvx-QG}}\label{appendix:subsubsec-proof-f_cvx-QG}

First, note that the gradient of $f_\mathrm{cvx}$ is 
\begin{equation*}
    \nabla f_\mathrm{cvx}(Y,X)
    = \begin{bmatrix}
2RYX^{-1}
        + 2RK^\star\\
    Q^\star - X^{-1}Y^\tr RYX^{-1}\end{bmatrix}.
\end{equation*}
Then,
it is not difficult to see that
\begin{align*}
     f_\mathrm{cvx}(Y,X)
    -f_\mathrm{cvx}(Y^\star,X^\star)
    =&
    f_\mathrm{cvx}(Y,X)
    -f_\mathrm{cvx}
    (0,X^\star)\\
    =&
   \mathrm{tr}\left(
    Y^\tr RYX^{-1}
    \right)
    + \mathrm{tr}Q^\star(X-X^\star) 
    +\mathrm{tr}
    (2RK^\star )^\tr Y\\
    \geq &  \mathrm{tr}\left(
    Y^\tr RYX^{-1}
    \right) \\
    \geq &
    \lambda_\mathrm{min}(R)
    \lambda_\mathrm{min}(X^{-1})
    \|Y\|_F^2\\
    \geq&
    \lambda_\mathrm{min}(R)\overline{\kappa}^{-1}
    \|Y-Y^\star\|_F^2
\end{align*}
from the optimality inequality at $(0, X^\star)$
\begin{equation*}
     \mathrm{tr}Q^\star(X-X^\star) 
    +\mathrm{tr}
    (2RK^\star )^\tr Y
    = \left\langle 
\underbrace{\begin{bmatrix}
0+2RK^\star\\
    Q^\star -0 \end{bmatrix}}_{\nabla f_\mathrm{cvx}(0,X^\star)},
    \begin{bmatrix}
        Y-0\\
    X-X^\star\end{bmatrix}
    \right\rangle
    \geq 0.
\end{equation*}
Now, 
by comparing
\begin{align*}
    &(A+BK^\star) X+X(A+BK^\star)^\tr
    +BY+Y^\tr B^\tr+W = 0\\
    &(A+BK^\star) X^\star+X^\star (A+BK^\star)^\tr
    +BY^\star + (Y^\star)^\tr B^\tr +W = 0
\end{align*}
we also have $
    X-X^\star = -\left(\mathcal{A}_\star ^{-1}\circ\mathcal{B}\right)
    (Y-Y^\star)
$ 
with $\mathcal{B}(Y)=BY+Y^\tr B^\tr $. This further  implies
\begin{equation*}
    \|X-X^\star\|_F
    \leq \|
    \mathcal{A}_\star ^{-1}\circ\mathcal{B}
    \|_2
    \|Y-Y^\star\|_F.
\end{equation*}
Since $\|
    \mathcal{A}_\star ^{-1}\circ\mathcal{B}
    \|_2 >0$,
we have the quadratic growth as
\begin{equation*}
    f_\mathrm{cvx}(Y,X)
    -f_\mathrm{cvx}(Y^\star,X^\star)
    \geq \frac{\mu_\mathrm{qg}}{2}
    \left(
    \|Y-Y^\star\|_F^2+
    \|X-X^\star\|_F^2\right)
\end{equation*}
with $\mu_{\mathrm{qg}}=
        \min\left\{\lambda_\mathrm{min}(R)
        \overline{\kappa}^{-1}
        ,
\lambda_\mathrm{min}(R)
        \overline{\kappa}^{-1}
    {
     \|
    \mathcal{A}_\star ^{-1}\circ\mathcal{B}
    \|_2^{-2}
    }
    \right\}$.

\subsubsection{Derivation of constant 
$c_\LQR$ in \cref{theorem:J_LQR-gradient_dominance-general} 
}\label{appendix:derivation-c_LQR}

Here, we derive the constant $c_\LQR$ in \cref{theorem:J_LQR-gradient_dominance-general},
which
corresponds to $c_\Phi$ in \cref{theorem:gradient-dominance-ECL}.
In the proof, this constant arises in \cref{proposition:curve} (c), and thus we present a tailored statement to LQR with $c_\LQR$
instead of $c_\Phi$.
The rest of the proof 
of \cref{theorem:J_LQR-gradient_dominance-general}
can be done in the same way as \cref{theorem:gradient-dominance-ECL}.

\begin{proposition}
Consider the LQR cost $J_\LQR:\mathcal{K}\to \mathbb{R}$
and the associated $\ECL$ in \cref{subsubsection:LQR-PL-proof}.
With \cref{assumption:controllability,assumption:compactness}, the function $\phi(\tau)$ in \cref{proposition:curve}
is differentiable on $\tau\in(-\epsilon,\epsilon)$ for some $\epsilon>0$
and satisfies
\begin{equation*}
    2 c_\LQR
    \|\phi_1'(0)\|_F
    \leq 
    \|g_\mathrm{c}\|_F.
\end{equation*}
\end{proposition}
\begin{proof}   
In the LQR problem, 
since $x=K$ and $\xi(x)=X$ (uniquely determined by \cref{eq:Lyapunov-H2} for $K\in\mathcal{K}$),
the function $\Phi(\cdot)$ can be written as
$\Phi(K,\gamma,X)= (\gamma,\Psi(K,X))$ with
\begin{equation*}
    \Psi(K,X)
    =\zeta=
    \begin{bmatrix}
        Y\\
        X
    \end{bmatrix}
    = \begin{bmatrix}
      (K-K^\star)X\\
      X
    \end{bmatrix}
\end{equation*}
with $(K,X)\in \pi_{K,X}(\mathcal{L}_{\lft})$.
We can also define the Cauchy direction $g_\mathrm{c}$ of $f_\mathrm{cvx}$ at $\zeta=[Y^\tr,X^\tr]^\tr$.

Now,
we can explicitly write
\begin{equation*}
    \begin{bmatrix}
        \phi_1(\tau)\\
        \phi_2(\tau)
    \end{bmatrix}
    =
    \begin{bmatrix}
    Y^+(\tau)
    (X^+(\tau))^{-1}
    +K^\star\\
    X^+(\tau)
    \end{bmatrix},
\end{equation*}
where
\begin{equation*}
    \begin{bmatrix}
        Y^+(\tau)\\
        X^+(\tau)
    \end{bmatrix}
    = \zeta^+(\tau)
    = \begin{bmatrix}
        Y\\
        X
    \end{bmatrix}
    -\tau \underbrace{\begin{bmatrix}
        [g_\mathrm{c}]_Y\\
        [g_\mathrm{c}]_X
    \end{bmatrix}}_{=g_\mathrm{c}}.
\end{equation*}
Then, since $X\succ 0$, the function $\phi(\tau)$
is differentiable around $\tau = 0$, and thus we can compute
\begin{align*}
    \phi_1'(0) 
    &= \frac{\partial Y^+}{\partial \tau}(0)X^{-1}
    + Y \frac{\partial \left(X^+(\tau)\right)^{-1}}{\partial \tau}(0) \\
    &= -[g_{\mathrm{c}}]_Y X^{-1}
    -Y \times
    \left(X^{-1}\left(-[g_{\mathrm{c}}]_X\right)X^{-1}
    \right)\\
    &=
    -\left([g_{\mathrm{c}}]_Y 
    -Y
    X^{-1}[g_{\mathrm{c}}]_X\right)X^{-1}. 
\end{align*}
Therefore, we obtain the upper bound as $
    \|\phi_1'(0)\|_F
    \leq (1+\|K-K^\star \|_F)\|X^{-1}\|_F
    \times \|g_\mathrm{c}\|_F.
$ 

Since $\mathcal{K}_\nu$ is compact, it is clear that the coefficient $\left(1 + \|K-K^\star\|_F\right)\|X^{-1}\|_F$ is bounded. Indeed, we show it is bounded as \cref{eq:PL-constant}, which is
\begin{equation*}
    \left(1 + \|K -K^\star \|_F\right)\|X^{-1}\|_F\leq \frac{1}{2c_\LQR}.
\end{equation*}
First, we know that
    \begin{equation*}
        \|X^{-1}\|_F
        = \sqrt{\|X^{-1}\|_F^2}
        \leq \frac{\sqrt{n}}{\sqrt{\lambda_\mathrm{min}(X)^2}}
        \leq \frac{\sqrt{n}}{\underline{\kappa}}.
    \end{equation*}
    To upper bound $\|K-K^\star \|_F$ over $\mathcal{K}_\nu$,
    it is not difficult to see that 
    \begin{align*}
        &\nu -J_\LQR(K^\star)\geq \mathrm{tr}\left(
        (K-K^\star )^\tr R(K-K^\star)X
        \right)
        \geq \lambda_\mathrm{min}(R)\lambda_\mathrm{min}(X)\|K-K^\star \|_F^2
        \geq \lambda_\mathrm{min}(R)\underline{\kappa}\|K-K^\star\|_F^2 \\
        \Rightarrow \quad& 
        \|K-K^\star\|_F \leq \sqrt{\frac{\nu-J_\LQR(K^\star)}{\lambda_\mathrm{min}(R)\underline{\kappa}}},
    \end{align*}
    where the first inequality is from the proof of \cref{lemma:f_cvx-QG} (or \cite[Corollary 3.5.1]{bu2020policy}).
    Therefore, we obtain
    \begin{equation*}
       \|\phi_1'(0)\|_F  \leq
        \underbrace{
        \left(1 + \sqrt{\frac{\nu-J_\LQR(K^\star)}{\lambda_\mathrm{min}(R)\underline{\kappa}}}\right)
        \times
        \frac{\sqrt{n}}{\underline{\kappa}}}_{=\frac{1}{2c_\LQR}>0\text{ in \cref{eq:PL-constant}}}
    \times
    \|g_{\mathrm{c}}\|_F.
    \end{equation*}
This completes the proof. 
\end{proof}

\subsection{Proof of \cref{proposition:sufficient-conditions}}\label{appedix:proof-proposition-sufficient-conditions-compactness}

Here, we present the proof of \cref{proposition:sufficient-conditions} that gives three sufficient conditions for \cref{assumption:compactness}. For convenience, we restate the proposition as follows.
\begin{proposition}[Restatement of \cref{proposition:sufficient-conditions}]\label{proposition:sufficient-conditions-restate}
    Suppose \cref{assumption:controllability} holds, and one of the following conditions is true.
    \begin{enumerate}
        \item[a)] The covariance matrix $W$ is positive definite;
        \item[b)] 
        The matrix $Q$ is positive definite, $(A, B)$ is controllable, and $\operatorname{Im} B \subseteq \operatorname{Im} W$;
        \item[c)] $(A, B)$ is controllable and
        $\operatorname{Im} B \subseteq \operatorname{Im} W$. For any sequence $\{K_j\}\subset\mathcal{K}$ such that
        $\lim_{j\to\infty}\|K_j\|_F=\infty$, we have
        ${\lim\inf}_{j\to\infty} \|P_j^{1/2}W^{1/2}\|_F/\|K_j\|_F^{1/2}>0,$ where $P_j$ is the solution to the Lyapunov equation 
        \begin{equation*}
            (A+BK_j)^\tr P_j+P_j(A+BK_j)+K_j^\tr RK_j+Q =0.
        \end{equation*}
    \end{enumerate}
    Then, we have that 1) the sublevel set $\mathcal{K}_\nu := \{K \in \mathcal{K} \mid J_\LQR(K) \leq \nu\}$ with any finite $\nu \geq J_\LQR(K^\star)$ is compact and 2) $X$ from  \cref{eq:Lyapunov-H2} is positive definite if and only if $K \in \mathcal{K}$. 
\end{proposition} 

In the following, we present the proof for each condition (a)--(c). 

\subsubsection{Proof with the sufficient condition (a)}\label{appendix:proof-sufficient-condition-a}

The condition ($W\succ 0$) in \cref{proposition:sufficient-conditions} (a) is standard and has widely been used in \cite{mohammadi2019global,bu2020policy,fatkhullin2021optimizing}. 
For self-completeness,
we prove the first statement
(the compactness of $\mathcal{K}_\nu$) with a more careful treatment of matrices $Q$ and $W$ than \cite{bu2020policy,fatkhullin2021optimizing}.
Note that
the second statement ($X$ from  \cref{eq:Lyapunov-H2} is positive definite if and only if $K \in \mathcal{K}$) is  standard when $W\succ 0$ (see \cite[Chapter 3.8]{zhou1996robust} for detail). 

To prove the compactness, we establish two facts 
\begin{enumerate}
    \item $J_\LQR(K)\to\infty$ as $K\to\partial\mathcal{K}$; \item the set $\mathcal{K}_\nu$ is bounded.
\end{enumerate}
Essentially, these two facts also establish the function $J_\LQR$ is coercive. Together with the continuity of $J_\LQR$, we know that its sublevel set $\mathcal{K}_\nu$ is compact. 

(1) We first show $J_\LQR(K)\to\infty$ as $K\to\partial\mathcal{K}$. This result can be considered as a special case in \cite[Lemma D.1]{zheng2023benign}, once we establish that $(A+BK,W^{1/2})$ is stabilizable and 
$
\left(\begin{bmatrix}
    Q^{1/2} \\ R^{1/2}K
\end{bmatrix}, A + BK\right)
$
is detectable.
(The detectability follows from \cref{assumption:controllability}.)
In the following, we give a direct argument. 

Consider a sequence of stabilizing controllers $\{K_i\}\subset \mathcal{K}$ such that $K_i\to K\in\partial\mathcal{K}$.
Since $K\in\partial\mathcal{K}$, we know $(A+BK)^\tr$ has at least one pure imaginary eigenvalue. Then, there must exist a pair $(\lambda,v)$ with $\|v\|=1$ such that $\lambda=\beta j$ with $\beta\in\mathbb R$ and $(A+BK)^\tr v=\lambda v$.
Now, consider a sequence of eigensystem $(\lambda_i,v_i)$ 
with $\|v_i\|=1$
such that 
$(\lambda_i,v_i)\to (\lambda,v)$ and
each $(\lambda_i,v_i)$ is
associated with the eigenvalue of the stable matrix $(A+BK_i)^\tr $ that has the largest real part.
Then, 
for the sequence $\{X_i\}$ from the Lyapunov equation \cref{eq:Lyapunov-H2} associated with $\{K_i\}$,
we have
\begin{align*}
v_i^\her X_i v_i 
=&v_i^\her\left( \int_{0}^\infty 
e^{(A+BK_i)t}We^{(A+BK_i)^\tr t}dt
\right)v_i\\
=& \int_0^\infty e^{2\mathrm{Re}(\lambda_i) t}
v_i^\her Wv_idt
 = \frac{v_i^\her W v_i}{-2\mathrm{Re}(\lambda_i)}.
\end{align*}
Since
$v^\her Wv> 0$ by
the stabilizability of $(A+BK,W^{1/2})$,
we have $v_i^\her X_i v_i\to \infty$ as $i\to \infty$ from $-\mathrm{Re}(\lambda_i)>0$ and $-\mathrm{Re}(\lambda_i)\to -\mathrm{Re}(\lambda)=0$.

For contradiction, assume $J_\LQR(K_i)$ remains bounded when $i\to\infty$.
Define a new sequence $\{\hat{X}_i\}$ as $\hat{X}_i = X_i/\|X_i\|_F$.
With a suitable relabeling, pass to a subsequence with $J_{\LQR}(K_i)\leq \nu$ for a finite $\nu\in(0,\infty)$ and $v_i^\her X_i v_i\to \infty.$
Then, by $\|\hat{X}_i\|_F=1$,
we can take a convergent subsequence $\{X_{i_j}\}$  such that $\hat X_{i_j}\to \hat{X}$.
By taking the limit for \cref{eq:Lyapunov-H2} after dividing it by $\|X_{i_j}\|_F$, we have $(A+BK)\hat X+ \hat{X}(A+BK)^\tr=0$.
Further, since $J_\LQR(K_{i_j})/\|X_{i_j}\|_F \to 0$ by the contradiction assumption,
we have 
\begin{equation*}
    (Q+K^\tr RK)\hat{X}=0.
\end{equation*}
Since $\left(Q+K^\tr RK, A + BK\right)$ is also detectable, 
we must have $\hat X=0$.
Indeed, since there exists $L$ such that $M=A+BK+L(Q+K^\tr RK)$ is stable, 
we have $M\hat X+\hat XM^\tr =0$ and $\hat X=0$.
However, this contradicts $\|\hat X\|_F=1$,
and
hence $J_\LQR(K_i)$ becomes unbounded when $K_i$ approaches the boundary.

(2) Next, we show any $K\in\mathcal{K}_\nu$ is bounded.
We first see that 
\begin{equation*}
    \nu \geq J_\LQR(K)=\mathrm{tr}(Q+K^\tr RK)X
    \geq 
    \lambda_\mathrm{min}(X)\lambda_\mathrm{min}(R)
    \|K\|_F^2
\end{equation*}
for all $K\in\mathcal{K}_\nu$. The boundedness of $K$ will follow if $\lambda_\mathrm{min}(X)$ is lower bounded by $\mathcal{O}(1/\|K\|_F)$. This fact would also establish $J_{\LQR}(K) \to \infty$ as $K \to \infty$.

Let $w$ be a normalized eigenvector of $X$ associated with $\lambda_\mathrm{min}(X)$.
Then, the Lyapunov equation \cref{eq:Lyapunov-H2}  
implies that 
\begin{equation*}
\lambda_\mathrm{min}(X)
= \frac{w^\tr Ww}{-w^\tr\left((A+BK)+(A+BK)^\tr\right)w}
\geq
\frac{\lambda_\mathrm{min}(W)}{
2\|A+BK\|
}
\geq 
\frac{\lambda_\mathrm{min}(W)}{
2\|A\|+2\|B\|\|K\|_F
}>0, 
\end{equation*}
where we have used $w^\tr Ww \geq \lambda_\mathrm{min}(W)$, and the fact that the spectral norm of a symmetric matrix is always no less than the absolute value of its smallest eigenvalue. So we have 
$$
\begin{aligned}
-w^\tr\left((A+BK)+(A+BK)^\tr\right)w
&\leq -\lambda_\mathrm{min}\left((A+BK)+(A+BK)^\tr\right) \\
&\leq \|
(A+BK)+(A+BK)^\tr
\| \leq 2\|A+BK\|.
\end{aligned}
$$
Thus, $K$ is bounded since  
 $\nu \geq J_\LQR(K)
    \geq {\lambda_\mathrm{min}(W)\lambda_\mathrm{min}(R)\|K\|_F^2}/(
    2\|A\|+2\|B\|\|K\|_F)
    . $
We note that this lower bound of $J_\LQR(K)$ also appeared in \cite[Lemma 3.8]{fatkhullin2021optimizing}.

\subsubsection{Proof with the sufficient condition (b)}\label{appendix:PL-W>=0-Q>0}

First, we check the second statement ($K\in\mathcal{K}$ iff $X\succ 0$).
By the assumption of $\operatorname{Im}B\subseteq \operatorname{Im}W$, for all $K\in\mathbb{R}^{m\times n}$ (including non-stabilizing gains), we have the controllability of $(A+BK,W)$ from the controllability of $(A,B)$ (recall the PBH test \cite{zhou1996robust}). 
Thus, the Lyapunov equation \cref{eq:Lyapunov-H2} has a positive definite solution $X\succ0$ if and only if 
$(A+BK)$ is stable, i.e., $K$ is stabilizing
(see \cite[Lemma 3.18]{zhou1996robust}). 

We now establish the first statement: the compactness of $\mathcal{K}_\nu$
for any $\nu\geq J_\LQR(K^\star)$.
Here, since \cref{assumption:compactness} (1) is true from the discussion above, we can similarly construct the $\ECL$ in \cref{subsection:ECL-GD}.
Recalling the 
invertibility of the affine map $\mathcal{A}_\star(\cdot)$,
we have
\begin{equation}\label{eq:K-C}
\mathcal{K}_\nu = 
\Xi\left(
\mathcal{C}_\nu
\right)
\end{equation}
with the continuous map $\Xi:\pi_Y(\mathcal{F}_\cvx)
\to\mathcal{K}$, given by 
$
    \Xi (Y) = YX_Y^{-1}+K^\star
$ 
with $X_Y =- \left( \mathcal{A}_\star^{-1}(\mathcal{B}(Y)+W) \right)\succ 0$,
and
$
    \mathcal{C}_\nu := \left\{
    Y
    \in \pi_Y(\mathcal{F}_\cvx)
    | 
    f_\cvx(Y,X_Y) \leq \nu
    \right\}.
$ 
Thus,
we show the compactness of $\mathcal{C}_\nu$ instead.

First, we show that the set ${\mathcal{C}}_\nu$ is bounded. 
By $Q\succ0$, we see that
\begin{equation*}
    \nu\geq {f}_\mathrm{cvx}(Y)
    \geq \mathrm{tr}(QX_Y) \geq \lambda_{\min}(Q)
    \mathrm{tr}(X_Y)
    {\geq \lambda_\mathrm{min}(Q)
    \lambda_\mathrm{max}(X_Y),} 
\end{equation*}
which gives
\begin{equation*}
    \lambda_\mathrm{min}(X_Y^{-1})
    =\frac{1}{
    \lambda_\mathrm{max}(X_Y)}
    \geq \frac{\lambda_\mathrm{min}(Q)}{\nu}>0.
\end{equation*}
Thus, we observe that 
from the proof of \cref{appendix:subsubsec-proof-f_cvx-QG},
$Y$ is bounded as
\begin{align*}
    &\nu -J_\LQR(K^\star) \geq
    \mathrm{tr}(X_Y^{-1}Y^\tr RY)
    \geq \lambda_\mathrm{min}(X_Y^{-1})
    \lambda_\mathrm{min}(R)\|Y\|_F^2
    \geq \frac{\lambda_\mathrm{min}(R)\lambda_\mathrm{min}(Q)\|Y\|_F^2}{\nu} \\
\Rightarrow \quad &\|Y\|_F \leq \frac{\nu}{\sqrt{\lambda_\mathrm{min}(R)\lambda_\mathrm{min}(Q)}}.
\end{align*}

Finally, we prove that the set $\mathcal{C}_\nu$ is closed by showing that
for any sequence $\{Y_i\}\subset \mathcal{C}_\nu$ such that $Y_i\to Y^*$,
we have $Y^*\in \mathcal{C}_\nu$.
To establish this, we consider $X_{Y^*}=- \lim_{Y_i\to Y^*}\left( \mathcal{A}_\star^{-1}(\mathcal{B}(Y_i)+W) \right)
$ and show that $X_{Y^*}$ is positive definite; this yields $Y^*\in \mathcal{C}_\nu$ by the continuity of $f_\mathrm{cvx}(Y)$ for $X_Y\succ0$.
(Since the mapping $X_Y=-\left( \mathcal{A}_\star^{-1}(\mathcal{B}(Y)+W) \right)$ is affine, $X_{Y^*}$ exists and is positive semidefinite.)
The proof is as follows.
By $f_\mathrm{cvx}(Y_i)-J^\star = \operatorname{tr}(YX_{Y_i}^{-1}Y^\tr R)\geq \lambda_\mathrm{min}(R)\operatorname{tr}(X_{Y_i}^{-1/2}Y^\tr Y X_{Y_i}^{-1/2})$,
we obtain
\begin{equation*}
    \|Y_iv\|^2 \leq  \frac{\nu}{ \lambda_\mathrm{min}(R)} v^\tr X_{Y_i}v.
\end{equation*}
for any real vector $v.$
If $X_{Y^*}v=0 $, this inequality yields $Y^*v=0$.
From the Lyapunov equation
\begin{equation*}
    A^\star X_{Y^*}+X_{Y^*} (A^\star)^\tr 
    + BY^*+(Y^*)^\tr B^\tr +W=0,
\end{equation*}
we then obtain $v^\tr Wv=0,$ i.e., $Wv=0$,
Consequently, we obtain $B^\tr v=0$ due to the assumption $\operatorname{Im}B\subseteq \operatorname{Im}W$,
which implies $\ker X_{Y^*}\subseteq \ker B^\tr.$
Now, by multiplying the Lyapunov equation by $v$, we obtain $ X_{Y^*}(A^\star)^\tr v=0$, which implies
 $   (A^\star)^\tr v\in \ker X_{Y^*} \subseteq \ker B^\tr.$
Therefore,
$v^\tr \begin{bmatrix}
        B, A^\star B,\ldots, (A^\star)^{n-1} B
    \end{bmatrix} =0.$
By the controllability of $(A^\star,B)$,
we obtain $v=0$ and $\ker X_{Y^*}=\{0\}$.
Since $X_{Y^*}$ is positive semidefinite, 
we conclude $X_{Y^*}\succ0$, 
and this completes the proof.

\subsubsection{Proof with the sufficient condition (c)}\label{appendix:PL-QG-W>=0}

Since $(A+BK,W)$ for any $K\in\mathbb{R}^{m\times n}$ is controllable
by the controllability of $(A,B)$ and $\operatorname{Im}B\subseteq \operatorname{Im}W$, we have the second statement
from the same argument as \cref{proposition:sufficient-conditions} (b).

Thus, we here prove the first statement, the compactness of $\mathcal{K}_\nu$.
Now, it suffices to see the coercivity of $J_\LQR(K)$ over $\mathcal{K}$: 
$J_\LQR(K)\to\infty$ as $K\to\partial\mathcal{K}$ and
$J_\LQR(K)\to\infty$ as $\|K\|_F\to\infty$.
The first property follows from the proof of \cref{proposition:sufficient-conditions} (a)
because $(A+BK,W)$ is controllable.
The second property can also be verified as follows.
With $P$ in \cref{eq:dual-Lyapunov-eq},
we have  $J_\LQR(K)
    = \mathrm{tr}(P W)$.
Thus, it is not difficult to see that
when $\|K\|_F\to\infty$,
\begin{equation*}
    J_\LQR(K)
    = \left\|P ^{1/2}
    W^{1/2}
    \right\|_F^2
    = \|K\|_F\times\frac{\left\|P ^{1/2}
    W^{1/2}
    \right\|_F^2}{\|K\|_F}
    =
    \|K\|_F\times\left(\frac{\left\|P ^{1/2}
    W^{1/2}
    \right\|_F}{\|K\|_F^{1/2}}\right)^2
    \to \infty.
\end{equation*}
This completes the proof.

\section{Discussions on (non)-unique optimal LQR feedback gains}\label{appendix:non-uniqueness}

Here, we provide some details regarding the (non)-uniqueness of optimal control gain $K^\star$ for the LQR problem.
While the optimal gain is unique under \cref{assumption:controllability,assumption:compactness}, we may have an infinite number of solutions in some other cases as \cref{example:W-singular}.

In our discussions below, the set
\begin{align*}
   \mathcal{K}_{\mathrm{ctrb}} :=&\{
    K\in\mathcal{K}|
    (A+BK,W^{1/2}) \text{ is controllable}
    \}
\end{align*}
plays a pivotal role.
As seen in our main results in \cref{section:GD}, 
if $\mathcal{K}=\mathcal{K}_{\mathrm{ctrb}}$
(as \cref{proposition:sufficient-conditions} (a,b)),
we can formulate $\ECL$s and have the gradient dominance
beyond the uniqueness.
Having $\mathcal{K}\setminus \mathcal{K}_\mathrm{ctrb}\neq \emptyset$, however, can cause the non-uniqueness.
Recall that 
the matrix $W$ is the covariance matrix of the initial point $x_0$, i.e., $W = \mathbb{E}[x_0x_0^\tr]$ (see \cref{section:LQR-formulation}).

First, we summarize
two important properties of
the set $\mathcal{K}_\mathrm{ctrb}$.
Our main results in \cref{section:GD} 
and \cref{assumption:compactness}
heavily rely on the following properties.
\begin{proposition}[Properties of matrix $X$ and $J_\LQR$ on $\mathcal{K}_\mathrm{ctrb}$]\label{proposition:uniqueness-X>0}
With \cref{assumption:controllability}, the following statements~hold:
\begin{enumerate}
    \item (Positive definiteness of $X$).
    The solution $X$ to \cref{eq:Lyapunov-H2} is positive definite if and only if
    $K\in\mathcal{K}_\mathrm{ctrb}$.
    \item (Uniqueness of the stationary point and its optimality). If
    $\nabla J_\LQR(K)=0$ holds for $K\in\mathcal{K}_\mathrm{ctrb}$,
    we uniquely have $K= - R^{-1}B^\tr P^\star $
    with the unique solution $P^\star$ to the ARE in \cref{eq:Riccati-eq}, and $K$ is an optimal solution.
\end{enumerate}
\end{proposition}
\begin{proof}
    The first statement is straightforward from \cite[Lemma 3.18]{zhou1996robust}.
    The second statement follows from $X\succ 0$. Indeed,
    when $X\succ 0$ for $K$, 
    \begin{equation*}
    \nabla J_\LQR(K)
    = 2(RK+B^\tr P)X =0
    \quad \Leftrightarrow
    \quad
    RK+B^\tr P =0
    \quad \Leftrightarrow
    \quad K=-R^{-1}B^\tr P
    \end{equation*}
    where $P$ is from \cref{eq:dual-Lyapunov-eq}.
    For the derivation of $\nabla J_\LQR(K)$, see for example
    \cite[Section IV]{levine1970determination}.
    Now, by substituting $K=-R^{-1}B^\tr P$ into \cref{eq:dual-Lyapunov-eq}, we obtain the ARE \cref{eq:Riccati-eq}.
    The uniqueness of $K$ is guaranteed by the uniqueness of $P^\star$ in \cref{lemma:Riccati-equation}.
    Thus, we have $P=P^\star$, and $K^\star$ is clearly optimal (see e.g., \cref{appendix:LQR}).
\end{proof}

This proposition implies that the first part of \cref{assumption:compactness} is crucial for our main results based on
the $\ECL$ (i.e., the exact convex reformulation).
For $K\in \mathcal{K}\setminus \mathcal{K}_\mathrm{ctrb}$, we have $X\succeq 0$, under which we cannot formulate $\ECL$s.
Moreover, we also observe that
the non-uniqueness of optimal gains cannot happen on
$\mathcal{K}_\mathrm{ctrb}$
under \cref{assumption:controllability}.
Indeed, the following proposition shows that
if we see the non-uniqueness,
the solutions lie in $\mathcal{K}\setminus\mathcal{K}_\mathrm{ctrb}$.
This is consistent with \cref{example:W-singular} because $(A+BK,W^{1/2})$ is uncontrollable and $X$ becomes singular for $K=[k,k]\in\mathcal{K}$ where we have the uncountable optimal  solutions.
Note that the second part is a restatement of \cref{proposition:non-unique-sol_line}.
\begin{proposition}[Properties of non-unique optimal gains]
Suppose \cref{assumption:controllability}.
Let $K^\star= -R^{-1}B^\tr P^\star $
with $P^\star$ from the ARE \cref{eq:Riccati-eq}.
Then,
the following statements hold.
\begin{enumerate}
    \item 
    If we have another optimal solution $K^\star_1\in\mathcal{K}$ satisfying $K_1^\star \neq K^\star $,
    we have $K^\star_1 \in \mathcal{K}\setminus
    \mathcal{K}_\mathrm{ctrb}$.
    \item 
    If we have two optimal solutions $K^\star_1,\,K^\star_2 \in\mathcal{K}$, then
    \begin{equation*}
        K^\star_\eta
        := K^\star_1 + \eta (K^\star_1 - K^\star_2)
    \end{equation*}
    is also optimal
    for any $\eta\in\mathbb{R}$
    as long as $K^\star_\eta\in\mathcal{K}$.
\end{enumerate}
\end{proposition}
\begin{proof}
The first part is shown as follows.
Leveraging the proof of the lower bound of $J_\LQR(K)$ in \cite[Corollary 3.5.1]{bu2020policy},
we have
\begin{equation*}
0=J_\LQR(K^\star_1)
-J_\LQR(K^\star)
= \mathrm{tr}(K_1^\star -K^\star)^\tr R(K_1^\star -K^\star)X^\star_1,
\end{equation*}
where $X^\star_1$ is the solution to the Lyapunov equation \cref{eq:Lyapunov-H2} associated with $K_1^\star$.
Then, by $R>0$, we observe that $\left\|(K^\star_1-K^\star)(X_1^\star)^{1/2}\right\|_F=0$.
Since $K^\star_ 1-K^\star\neq0$, $X_1^\star$ must be singular and thus $K^\star_1\notin \mathcal{K}_\mathrm{ctrb}$.
The second part follows from \cref{proposition:non-unique-sol_line}.
\end{proof}

Without controllability,
we may have many optimal LQR gains.
On the other hand, if we only have stabilizability, then the unboundedness of $J_\LQR(K)$ on the boundary still holds.
This shows that
the stabilizability of $(A+BK,W^{1/2})$ for $K\in\partial \mathcal{K}$ is sufficient for confirming $J_\LQR(K_j)\to \infty$ as $K_j\to K\in \partial\mathcal{K}$.

\begin{proposition}
With \cref{assumption:controllability},
for a sequence of stabilizing gains  
$\{K_j\}\subset \mathcal{K}$ satisfying
$K_j\to K\in \partial \mathcal{K}_\mathrm{stab}$ where
$
    \partial\mathcal{K}_{\mathrm{stab}} :=\{
    K\in\partial\mathcal{K}|
    (A+BK,W^{1/2}) \text{ is stabilizable}
    \},
$ 
we have $J_\LQR(K_j)\to \infty$ as $K_j\to K$. 
\end{proposition}
The proof of this result follows the first part of \cref{appendix:proof-sufficient-condition-a}.

\section{Completion of squares in the classical LQR proof} \label{appendix:LQR}

Here, we review a classical proof of the global optimality of a static state feedback policy in LQR using a completion-of-squares technique. This approach can be found 
in classical textbooks such as 
\cite[Theorem 14.2]{zhou1996robust} and \cite[Theorem 21.1]{hespanha2018linear}.
In the following proof, we assume \cref{assumption:controllability} holds.

Consider the LQR problem \cref{eq:LQR-w/o-state-fb}, and let $P$ be a $n\times n$ symmetric matrix.
Then, for any input $u\in\mathcal{L}_2^m[0,\infty)$ that stabilizes the system, we have that 
\begin{equation} \label{eq:LQR-squares-step-1}
\begin{aligned}
\int_0^\infty \frac{d}{dt} (x(t)^\tr Px(t))dt = &\lim_{T \to \infty}\int_0^T \frac{d}{dt} (x(t)^\tr Px(t))dt \\
= & \lim_{T \to \infty }x(T)^\tr Px(T) - x(0)^\tr Px(0) = - x(0)^\tr Px(0). 
\end{aligned}
\end{equation}
We also have 
\begin{equation} \label{eq:LQR-squares-step-2}
\begin{aligned}
    \frac{d}{dt} (x(t)^\tr Px(t)) &= \dot{x}(t)^\tr Px(t) + x(t)^\tr P \dot{x}(t) \\
    &= \left(Ax(t)+Bu(t)\right)^\tr Px(t)
    +x(t)^\tr P \left(Ax(t)+Bu(t)\right),
\end{aligned}
\end{equation}
where we have applied the dynamics $\dot{x} = Ax + Bu$. 
Then, we can reformulate the LQR cost function as 
\begin{equation} \label{eq:LQR-squares-step-3}
\begin{aligned}
    &\int_{0}^\infty \left(x(t)^\tr Qx(t) + u(t)^\tr Ru(t)\right)dt \\
    =& \int_{0}^\infty \left(x(t)^\tr Qx(t) + u(t)^\tr Ru(t)\right)dt
    + \int_0^\infty \frac{d}{dt} (x(t)^\tr Px(t))dt
    - \int_0^\infty \frac{d}{dt} (x(t)^\tr Px(t))dt \\
   =&\int_{0}^\infty \left(x(t)^\tr Qx(t) + u(t)^\tr Ru(t)
    +\left(Ax(t)+Bu(t)\right)^\tr Px(t)
    +x(t)^\tr P \left(Ax(t)+Bu(t)\right)
    \right)dt
    + x_0^\tr P x_0
    \\
    =&
    \int_{0}^\infty
    \left(x(t)^\tr \left(A^\tr P+PA+Q\right)x(t) + u(t)^\tr Ru(t)
    + 2u(t)^\tr B^\tr Px(t)  \right)dt  
    + x_0^\tr P x_0, 
\end{aligned}
\end{equation}
where the second equality applies \Cref{eq:LQR-squares-step-1,eq:LQR-squares-step-2}. We observe that the last equality is a quadratic functional of $x(t)$ and $u(t)$, and moreover, we can perform a \textit{completion of squares} as follows
\begin{equation} \label{eq:LQR-squares-step-4}
\begin{aligned}
    &\int_{0}^\infty
    \left(x(t)^\tr \left(A^\tr P+PA+Q\right)x(t) + u(t)^\tr Ru(t)
    + 2u(t)^\tr B^\tr Px(t)  \right)dt  
    + x_0^\tr P x_0
    \\
    =& \int_{0}^\infty
    (u(t)+R^{-1}B^\tr Px(t))^\tr R(u(t)+R^{-1}B^\tr Px(t))dt \\
    &
    +\int_{0}^\infty
    x(t)^\tr \left(A^\tr P+PA+Q - PBR^{-1}B^\tr P \right)x(t) dt
    +x_0^\tr P x_0.
\end{aligned}
\end{equation}
By completing the square in \cref{eq:LQR-squares-step-4}, we clearly see a nonnegative square term $(u(t)+R^{-1}B^\tr Px(t))^\tr R(u(t)+R^{-1}B^\tr Px(t))$ and the familiar Riccati equation. 
We then select $P$ as the unique stabilizing solution $P^\star$ to the ARE \eqref{eq:Riccati-eq}, which is guaranteed to exist by \Cref{lemma:Riccati-equation}. From \cref{eq:LQR-squares-step-3,eq:LQR-squares-step-4}, it is clear that we have  
\begin{equation} \label{eq:LQR-squares-step-5}
\int_{0}^\infty\! \left(x(t)^\tr Qx(t) \!+\! u(t)^\tr Ru(t)\right)dt = \int_{0}^\infty\!
    \left(u(t)+R^{-1}B^\tr P^\star x(t)\right)^\tr R\left(u(t)+R^{-1}B^\tr P^\star x(t)\right)dt + x_0^\tr P^\star x_0
\end{equation}
for any input $u\in\mathcal{L}_2^m[0,\infty)$ that stabilizes the system. 

From \cref{eq:LQR-squares-step-4,eq:LQR-squares-step-5}, we can see that the optimal control input to the LQR problem \cref{eq:LQR-w/o-state-fb} is a linear static feedback $u(t)= -R^{-1}B^\tr P^\star x(t)$. Indeed, if we use any stabilizing input $\tilde{u}(t)= -R^{-1}B^\tr P^\star x(t) + v(t)$ in \Cref{eq:LQR-squares-step-4}, we have 
$$
\int_{0}^\infty \left(x(t)^\tr Qx(t) + u(t)^\tr Ru(t)\right)dt 
    =\int_0^\infty v^\tr (t) R v(t) dt + x_0^\tr P^\star x_0 \geq x_0^\tr P^\star x_0.
$$
Since $R \succ 0$, the optimal LQR cost value is achieved only when $v(t) \equiv 0$. 
Therefore, ${u}(t) = -R^{-1}B^\tr P^\star x(t)$ is the \textit{unique} optimal control input to the LQR problem.

\begin{remark}[Unique optimal control input vs unique optimal feedback gain]
The completion of squares in \Cref{eq:LQR-squares-step-4} plays a key role in the proof above. It confirms that the optimal control input is in the form of \textit{linear static state feedback}, and this optimal input signal $u(t)$ is \textit{unique}. However, it does not guarantee the uniqueness of the optimal feedback gain $K$. 
In fact,  as demonstrated in \cref{example:W-singular}, the LQR instance  under \cref{assumption:controllability}  admits non-unique optimal feedback gains, even though the optimal input signal $u(t)$ remains unique. If the covariance matrix $W$ is strictly positive definite, the optimal LQR feedback gain is then unique (\cref{theorem:strong-duality-2}). 
While the completion-of-squares proof is relatively straightforward, the technique of adding and subtracting $\int_0^\infty \frac{d}{dt}x^\tr Px dt$ in \Cref{eq:LQR-squares-step-3} is non-trivial.
Our alternative proof strategy, based on primal-dual analysis (\cref{theorem:linear-optimality}), provides additional insights --- specifically, it shows that the linear static feedback policy naturally emerges from the KKT conditions of a standard convex optimization problem. \hfill $\square$ 
\end{remark}

\end{document}